\documentclass[11pt,twoside]{preprint}
\usepackage{comment}
\usepackage{hyperref}
\usepackage{breakurl}
\usepackage{times}
\usepackage{tikz}
\usepackage{microtype}
\usepackage{booktabs}

\usepackage{amsmath}
\makeatletter
\let\over\@@over
\let\atop\@@atop
\makeatother
\usepackage{amssymb}
\usepackage{mathrsfs}
\usepackage{mhequ}
\usepackage{mhenvs}
\usepackage{mhsymb}

\usepackage{enumitem}

\definecolor{darkgreen}{rgb}{0.1,0.7,0.1}
\definecolor{darkred}{rgb}{0.7,0.1,0.1}
\hypersetup{
citecolor=blue,
colorlinks=true, 
urlcolor= blue, 
linkcolor= black, 
bookmarksopen=true, 
pdftitle={mSHE}, 
}

\newtheorem{assumption}[lemma]{Assumption}

\newop{diag}

\newcommand\minus{%
  \setbox0=\hbox{-}%
  \vcenter{%
    \hrule width\wd0 height \the\fontdimen8\textfont3%
  }%
}

\def\m{\mathfrak{m}}

\def\s{\mathfrak{s}}

\def\${|\!|\!|}

\def\m{\mathfrak{m}}

\newcommand{\p}{\mathrm{p}}
\newcommand{\e}{\mathrm{e}}

\newcommand{\w}{\rw}
\newcommand{\wun}{\rw^{(1)}}
\newcommand{\wde}{\rw^{(2)}}

\newcommand{\wi}{\rw^{(i)}}

\newcommand{\wPi}{\mathrm{w}_\Pi}

\newcommand{\pa}{\mathrm{p}_a}

\def\tun{\mathbf{1}}

\newcommand{\rw}{\mathrm{w}}

\def\qed{ \hfill $\square$}

\newcommand{\bbE}{\mathbb{E}}

\newcommand{\cA}{\mathcal{A}}
\newcommand{\cB}{\mathcal{B}}
\newcommand{\cC}{\mathcal{C}}
\newcommand{\cD}{\mathcal{D}}
\newcommand{\cE}{\mathcal{E}}
\newcommand{\cF}{\mathcal{F}}
\newcommand{\cG}{\mathcal{G}}

\newcommand{\cI}{\mathcal{I}}

\newcommand{\cM}{\mathcal{M}}

\newcommand{\cO}{\mathcal{O}}
\newcommand{\cP}{\mathcal{P}}
\newcommand{\cQ}{\mathcal{Q}}
\newcommand{\cR}{\mathcal{R}}
\newcommand{\cS}{\mathcal{S}}
\newcommand{\cT}{\mathcal{T}}
\newcommand{\cU}{\mathcal{U}}

\newcommand{\ccD}{\mathscr{D}}
\newcommand{\ccE}{\mathscr{E}}

\begin{document}

\title{Multiplicative stochastic heat equations\\on the whole space}
\author{Martin Hairer$^1$ and Cyril Labb\'e$^2$}
\institute{University of Warwick, \email{M.Hairer@Warwick.ac.uk}
\and Universit\'e Paris Dauphine, \email{Labbe@Ceremade.Dauphine.fr}}

\date{\today}

\maketitle

\begin{abstract}
We carry out the construction of some ill-posed multiplicative stochastic heat equations on unbounded domains. 
The two main equations our result covers are, on the one hand the parabolic Anderson model on $\R^3$, and on 
the other hand the KPZ equation on $\R$ via the Cole-Hopf transform. To perform these constructions, we adapt 
the theory of regularity structures to the setting of weighted Besov spaces. One particular feature of 
our construction is that it allows one 
to start both equations from a Dirac mass at the initial time. 
\end{abstract}

\tableofcontents

\section{Introduction}

In the present paper, we consider the following stochastic partial differential equation:
\begin{equ}[e:E]
\tag{E}
\partial_t u = \Delta u + u\cdot\xi \;,\qquad u(0,\cdot) = u_0(\cdot)\;,
\end{equ}
where $u$ is a function of $t\geq 0$, $x\in\R^d$, and $\xi$ is an irregular noise process. 
While large parts of our analysis are dimension-independent, a natural subcriticality condition 
restricts the dimensions in which we can consider the most-interesting case of delta-correlated noise. 
We will henceforth be mainly concerned with two instances of this equation: $d=3$ and $\xi$ is a 
white noise in space only, 
we refer to this case as (PAM); $d=1$ and $\xi$ is a space-time white noise, we call this case (SHE).

When $\xi$ is a white noise in space, without dependence in time, this equation is indeed called the parabolic 
Anderson model (PAM). In dimension $d\geq 2$, the equation is ill-posed, due to the very singular product 
$u\cdot\xi$. Indeed, $u$ is expected to be $(2+\alpha)$-H\"older where the regularity of the noise $\alpha$ is strictly lower than $-d/2$, so that the sum of the regularities of $u$ and $\xi$ is strictly negative, and therefore, the product $u\cdot\xi$ does not fall in the scope of classical integration theories~\cite{BookChemin,Young}. To make sense of this product, one actually needs to perform some \textit{renormalisation} which boils down to, roughly speaking, subtracting some infinite linear term from the equation.

When the space variable is restricted to a torus of dimension $2$, the solution of a generalised version of (PAM) has been constructed independently by Gubinelli, Imkeller and Perkowski~\cite{GubImkPer} using paracontrolled distributions, and by Hairer~\cite{Hairer2014} via the theory of regularity structures. The construction on a torus of dimension $3$ follows immediately from recent results of 
Hairer and Pardoux~\cite{Etienne}.
The construction of (PAM) on the full space $\R^2$ has been obtained recently~\cite{HairerLabbePAM2}, using a simple change of unknown that spares one from requiring elaborate renormalisation theories. This is not possible anymore in dimension $3$: in the present paper, we adapt the theory of regularity structures to perform the construction of (PAM) on the full space $\R^3$.

When $\xi$ is a space-time white noise, we refer to \eqref{e:E} as the multiplicative stochastic heat 
equation (SHE). Already in dimension $d=1$, the product $u\cdot\xi$ is ill-defined. However, in 
dimension $1$, the It\^o integral allows one to make sense of this equation: it requires the noise to 
be a martingale in time and the solution $u$ to be adapted to the filtration of the noise. This 
construction breaks down for space-time regularisations of the white noise so that it does not 
imply convergence of space-time mollified versions of the original equation. When the space variable is 
restricted to a torus of dimension $1$, this equation has been constructed by Hairer and 
Pardoux~\cite{Etienne} in the framework of regularity structures: they define the solution map on 
a space of noises that contains a large class of space-time mollifications of the white noise. In 
the present paper, we lift the restriction of the torus and perform the construction on the whole line $\R$.

This equation is intimately related to the KPZ equation~\cite{KPZ86}. Indeed, formally, the Cole-Hopf transform 
sends the ill-posed KPZ equation to (SHE); Bertini and Giacomin~\cite{BG97} exploited this fact to prove the 
convergence of the fluctuations of the weakly asymmetric simple exclusion process to the KPZ equation on $\R$. 
A more direct interpretation of the KPZ equation itself has recently been obtained by Hairer~\cite{HairerKPZ}, 
when the space variable is restricted to a torus of dimension $1$.

In addition to the ill-defined product $u\cdot\xi$ that needs to be renormalised for both (PAM) and (SHE), there are two major 
issues that we address in this work: first, we construct these SPDEs on an unbounded underlying space 
instead of a torus; second, we consider a Dirac mass as the initial condition.

\smallskip

Let us first comment on the specific difficulty arising from the unboundedness of the underlying space, when constructing the solutions to these SPDEs. Since the white noise is not uniformly H\"older on an unbounded space, one cannot expect to 
obtain solutions that are uniformly bounded over the underlying space and one needs to weight the H\"older spaces of functions/distributions at infinity. This is a classical problem when dealing with stochastic PDEs in unbounded domains, 
see for example \cite{MR933824,MR1113223}, as well as the recent work \cite{2015arXiv150106191M} 
which is somewhat closer in spirit to the equations
considered here.
A priori, these weights cause some trouble in obtaining a fixed point for the map $u\mapsto P*(u\cdot\xi) + P*u_0$, where $P$ is the heat kernel. Indeed, since the weight needed for the product $u\cdot\xi$ is a priori larger than the weight of $u$ itself, the map would take values in a space bigger than the one $u$ lives in and the fixed point argument would not apply.

There is a way of circumventing this problem by considering a time-increasing weight and by using the averaging in time of the weight due to the time convolution with the heat kernel. More precisely, the white noise can be weighted by a polynomial weight $\pa(x)=(1+|x|)^a$ with $a$ as small as desired, so that, if we weight the solution by $\e_t(x)= e^{t(1+|x|)}$, then $\int_0^t P_{t-s}*(\xi\cdot u_s) ds$ can be weighted by $\int_0^t \pa(x)\e_s(x) ds$ which is smaller than $\e_t(x)$. We refer to~\cite{HairerLabbePAM2} for a construction of (PAM) on $\R^2$ using this idea, and to~\cite{HaPaPi2013} where this trick already appeared.
The main difficulty is therefore to incorporate the trick outlined above into the theory of regularity structures, and this will require to have an accurate control on the weights arising along the construction. In particular, a major issue comes from the fact that $\e_t(x)/\e_s(y)$ is not bounded from above and below, uniformly over all $(t,x), (s,y)$ lying at distance, say, $1$ from each other.

Let us point out that the important feature of the exponential weight is not the growth at infinity in $x$, but the exponential growth in $t$. In particular, the trick presented above works the same with $e_{t+\ell}(x)$, where $\ell\in\R$ is arbitrary. Consequently, if the initial condition has bounded support, then one can choose $\ell$ arbitrarily small so that the weight of the solution at any time $t < -\ell$ decays exponentially fast in $x$.

At this point, let us mention that our method to construct solutions of SPDEs on the whole space would still apply to equations of the form
\begin{equ}
\partial_t u = \Delta u + F(u) + G(u)\cdot\xi\;,\quad u(0,\cdot) = u_0(\cdot)\;,
\end{equ}
where $F,G$ are smooth functions whose growth at infinity is at most linear, and such that $G(0)=0$, and $u_0$ decays exponentially fast at infinity. Let us give a brief explanation. Under these hypotheses, if we weigh the solution at time $t > 0$ by $e_{t+\ell}(x)$ for some $\ell < 0$, then the weight of $F(u)$ at time $t$ is of order $e_{t+\ell}(x)$, and the weight of $G(u)\xi$ is of order $e_{t+\ell}(x) \pa(x)$. Notice that, if $G(0)\ne 0$ then the latter is not true anymore when $t+\ell < 0$. Due to the non-linearities, the fixed point map is only locally Lipschitz with a proportionality constant which is quadratic in the solution. At the level of the weights, this is not a problem as long as they are uniformly bounded over all $x$: this is the case for all $t < -\ell$, and thus, we can obtain a local solution.\\
On the other hand, this approach does not yield a solution theory for the KPZ equation (without Hopf-Cole transform) on the whole line: the non-linearity is then quadratic in the (derivative of) the solution so that its weight would be the square of the weight for the solution, and the trick presented above would not work anymore.

\smallskip

Regarding the initial condition, let us point out that the Picard iterations associated to (\ref{e:E}) involve products of the form $(P*u_0) \cdot \xi$. By the classical integration theories~\cite{BookChemin,Young}, this product makes sense as soon as the regularity of $P*u_0$ is strictly larger than $-\alpha$, where $\alpha$ is the regularity of the noise. $P*u_0$ is smooth away from $t=0$, but its space-time regularity near $t=0$ coincides with the space regularity of $u_0$. Since the time regularity counts twice in the parabolic scaling, it is possible to make sense of $(P*u_0) \cdot \xi$ as long as $u_0$ has a regularity better than $-2-\alpha$, using \textit{integrable} weights around time $0$. The H\"older regularity of the Dirac mass being equal to $-d$, this would prevent us from choosing $u_0 =\delta_0$ for both (PAM) and (SHE).

One way of circumventing this problem is to exploit the fact that on the other hand the Dirac distribution
is ``almost'' an $L^1$ function. In particular, it belongs to the Besov spaces $B_{p,\infty}^\beta$
as soon as  $\beta < -d+d/p$. Since the classical integration theories allow one to multiply $\cC^\alpha$ by $B_{p,\infty}^\beta$ as soon as $\alpha+\beta>0$, the threshold on the regularity of the initial condition would not be modified upon this change of distributions spaces. Choosing $p$ small enough, one would then be able to take a Dirac mass as the initial condition. Let us emphasise that we do not choose an $L^p$-type space for the white noise but still consider the H\"older space $\cC^\alpha$: otherwise, the integrability of the solution would be deteriorated upon multiplication with the noise. We now present the main steps of the construction of the solution to (\ref{e:E}).

First, we define a \textit{regularity structure}, which is an abstract framework that allows one to associate to a function/distribution some \textit{generalised} Taylor expansion around any space/time point. The building blocks of this regularity structure are, on the one hand, polynomials in the space/time indeterminates, and on the other hand, some abstract symbols $\Xi$, $\cI(\Xi)$, $\ldots$,  associated with the noise. Then, one needs to reformulate the solution map that corresponds to (\ref{e:E}) into the abstract framework of the regularity structure. Namely, one has to provide abstract formulations of the multiplication with the noise $\xi$ and the convolution with the heat kernel $P$.

Second, we build a so-called \textit{model} which associates to the abstract symbols some analytical values. Actually, we start with a mollified version of the noise $\xi_\epsilon=\rho_\epsilon*\xi$, where $\rho_\epsilon(t,x)=\epsilon^{-2-d}\rho(t\epsilon^{-2},x\epsilon^{-1})$ is a smooth, compactly supported function which is such that 
$\rho(t,x) = \rho(t,-x)$, and we build a model $(\Pi^\epsilon, F^\epsilon)$ which, in particular, associates to the symbol $\Xi$ the smooth function $\xi_\epsilon$. One important feature is that the abstract solution given by the solution map, with this particular model, coincides (upon an operation called \textit{reconstruction}) with the classical solution of the well-posed SPDE
\begin{equ}[e:Eeps]\tag{E$_\epsilon$}
\partial_t u_\epsilon = \Delta u_\epsilon + u_\epsilon\cdot\xi_\epsilon \;,\qquad u_\epsilon(0,\cdot) = u_0(\cdot)\;.
\end{equ}

Third, we renormalise the model $(\Pi^\epsilon,F^\epsilon)$ by modifying the values associated to some symbols: namely, those symbols that stand for ill-defined products. Roughly speaking, the modification of these values consists in subtracting some divergent constant $C_\epsilon$. The effect of this renormalisation procedure is actually very clear at the level of the SPDE, since the abstract solution then corresponds to
\begin{equ}[e:Eepshat]\tag{$\widehat{\text{E}}_\epsilon$}\partial_t \hat{u}_\epsilon = \Delta \hat{u}_\epsilon + \hat{u}_\epsilon\cdot(\xi_\epsilon-C_\epsilon) \;,\qquad \hat{u}_\epsilon(0,\cdot) = u_0(\cdot)\;.
\end{equ}

The final step consists in proving that the sequence of renormalised models converges as $\epsilon\downarrow 0$ in a sense that will be made clear later on. The continuity of the solution map then ensures that the sequence of abstract solutions converge, and furthermore, the limit is the fixed point of an abstract fixed point equation. This yields the convergence of the sequence of renormalised solutions $\hat{u}_\epsilon$ to a limit $u$.

Let us now outline some major modifications that we bring to the original theory of regularity 
structures~\cite{Hairer2014}. 
First, since we want to start (\ref{e:E}) from a Dirac mass, we need to 
choose an appropriate space of distributions. As explained earlier in the introduction, we are 
led to using (some analog in the context of regularity structures of) 
the $B_{p,\infty}^\beta$ spaces. Therefore, we present a new 
version of the reconstruction operator in this setting, see Definition~\ref{Def:HolderLp} and Theorem~\ref{Th:Reconstruction} for precise formulations. 
Second, our spaces of \textit{modelled distributions} are weighted at 
infinity; therefore, the reconstruction theorem and the abstract convolution with the heat kernel 
need to be modified in consequence, we refer to Theorems \ref{Th:ReconstructionWeight} and 
\ref{Th:Integration}. One major difficulty we run into is that one would like to consider
the same kind of weights as in \cite{HaPaPi2013,HairerLabbePAM2}, which are of the type $w(t,x) = \exp(t(1+|x|))$.
Unfortunately, such weights do \textit{not} satisfy the very desirable property
$c \le |w(z)/w(z')| \le C$ for some constants $c,C>0$, uniformly over space-time points $z,z'$ with
$|z-z'| \le 1$, although they \textit{do} satisfy this property for pairs of points that are
only separated in space. As a consequence, we need extremely fine control on the behaviour of
our objects as a function of time, see for example the bound \eref{Eq:RefinedBoundReconstructionp}
and the illustration of Figure~\ref{Fig}. Note that in the case of (PAM), where the noise varies only
in space, we could have defined our regularity structure on space only and viewed the solution as a function
of time with values in a space of modelled distributions, thus substantially shortening some of the arguments. We chose not to do that, and instead we present results that work indifferently for (PAM) and (SHE). This is possible since the algebraic and scaling properties of these two equations coincide. In particular, the regularity structure can be built in the same way for both. On the other hand, the analytic values associated to the elements in the regularity structure, i.e.~\textit{the model}, are specific: in particular the renormalisation constants are specific to (PAM) and (SHE).

The main result of the present work is as follows.
\begin{theorem}\label{Th:Main}
We consider either (PAM) where $d=3$, or (SHE) where $d=1$. Let $u_0 \in \cC^{\eta,p}_{w_0}(\R^d)$ with $\eta > -1/2$, $p\in [1,\infty)$ and $w_0(x) = e^{\ell(1+|x|)}$ for some $\ell \in \R$. There exists a divergent sequence of constants $C_\epsilon$ such that, on any interval of time $(0,T]$, the sequence of solutions $\hat{u}_\epsilon$ of (\ref{e:Eepshat}) converges uniformly on compact sets of $(0,\infty)\times\R^d$ to a limit $u$, in probability. 

Furthermore, the limit depends continuously on the initial condition $u_0$ and,
provided that $C_\eps$ is suitably chosen, it is independent of the choice of mollifier $\rho$. 
In the case of (SHE), the limit can be chosen to coincide with the classical solution to the 
multiplicative stochastic heat equation \cite{Walsh,DPZ}.
\end{theorem}

\begin{remark}
In the framework of the regularity structure, we will define a lifted version of (PAM) or (SHE). This lifted version will admit a unique fixed point, see Theorem \ref{Th:FixedPt}. The fixed point does not live in a space of usual distributions, but in a space of so-called \textit{modelled distributions}. However, the reconstruction operator defined in Theorem \ref{Th:ReconstructionWeight} associates continuously (with the respective topologies) a genuine distribution to such an object.
\end{remark}

\begin{remark}
We refer to Definition \ref{Def:DistribSpace} for the precise space of distributions in which the convergence holds. Moreover, the space $\cC^{\eta,p}_{w_0}(\R^d)$ is defined in Subsection \ref{SubsectionIC}.
We would like to point out however that for $p$ sufficiently close to $1$ and $\eta$ negative
one has $\delta_0 \in \cC^{\eta,p}_{w_0}$.
\end{remark}

\begin{remark}
A result of M\"uller~\cite{Mueller} ensures that the classical (It\^o) solution of (SHE) is 
strictly positive at any time $t>0$ as soon as it starts from an initial condition which 
is non-negative and non-zero. In the same way as in \cite[Cor.~6.5]{Etienne}, this classical solution 
can be shown to coincide with our limit $u$ so that, 
by setting $h=\log u$, we recover the Cole-Hopf solution to the KPZ equation
formally given by
\begin{equ}
\tag{KPZ}
\partial_t h = \partial_x^2 h + (\partial_x h)^2 + \xi \;.
\end{equ}
In particular, if we start $u$ from $\delta_0$, the solution $h$ is referred to as 
the ``infinite wedge" solution to the KPZ equation, see~\cite{ACQ}. In this case,
our convergence result is new also in the case of solutions on the circle rather than
the whole real line.
Let us point out again however that our approach does not provide a direct solution theory 
to the KPZ equation on the line, as explained earlier in the introduction.
\end{remark}

\begin{remark}
The exponent $-{1\over 2}$ obtained in this result is sharp. 
Indeed, since the equation is linear in the initial condition, it is sufficient to be able 
to take $u_0 = \delta_y$, which is allowed in our setting. Denoting the corresponding 
solution - or solution kernel - by $K_t(x,y)$, general solutions are given by $u(t,x) = \int K_t(x,y)u_0(y)\,dy$.
Furthermore, in the case of (PAM), it is straightforward to see by an approximation argument
that $K_t$ is symmetric in both of its arguments. (In the case of (SHE) it is only symmetric in law.)
At this stage we then note that in both cases
we expect $K_t$ to inherit the regularity of the linearised problem, 
namely to be of H\"older regularity $\CC^\alpha$ for  $\alpha < {1\over 2}$ in both of its arguments, 
but no better. (In the case of (SHE) this is of course a well-known fact.)
Such functions cannot be tested against a generic distribution in $\cC^{\eta,1}$ if $\eta \le -1/2$.
\end{remark}

\begin{remark}
In the case of (PAM), denote by $K_t$ the integral operator on $L^2(\R^3)$ with kernel $(x,y) \mapsto K_t(x,y)$.
Then $K_t$ is in general an unbounded selfadjoint operator (with a domain depending on 
the realisation of the underlying noise!). Furthermore, $K_t$ is positive definite since its kernel
is obtained as a pointwise limit of positive kernels. Finally, for any fixed $t>0$, $K_t$ does not
admit any $\phi\in L^2$ with $K_t \phi = 0$.
Indeed, since the operators $K_t$ satisfy
$K_t K_s = K_{t+s}$, one would have $K_{t/n} \phi = 0$ for every $n > 0$, which would contradict the
fact that $K_t \phi$ converges to $\phi$ weakly as $t \to 0$.
As a consequence, we can define an operator $L = {1\over t}\log K_t$ by functional
calculus. This operator is naturally interpreted as a suitably renormalised version of the 
random Schr\"odinger operator
\begin{equ}
L_\xi = -\Delta + \xi\;,
\end{equ}
on $\R^3$. See \cite{Khalil} for more details on a similar construction in dimension 2 
(and bounded domain).
\end{remark}

In both cases, the renormalisation constant $C_\epsilon = c_\epsilon + c^{(1,1)}_\epsilon + c^{(1,2)}_\epsilon$ is given by
\begin{equs}
c_\epsilon &:= \int G(z) \rho_\epsilon^{*2}(z)dz\;, \\
c^{(1,1)}_\epsilon &:= \int G(z_1) G(z_2) G(z_3) \rho_\epsilon^{*2}(z_1+z_2) \rho_\epsilon^{*2}(z_2+z_3) \prod_{i=1}^3 dz_i\;,\label{Eq:RenormCsts}\\
c^{(1,2)}_\epsilon &:= \int G(z_1)G(z_2)\Big( G(z_3)\rho_\epsilon^{*2}(z_3)-c_\epsilon \delta_0(z_3)\Big)\rho_\epsilon^{*2}(z_1+z_2+z_3) \prod_{i=1}^3 dz_i\;.
\end{equs}
In the case of (PAM), $G$ is a compactly supported function that coincides with the Green's function of the $3$ dimensional Laplacian in a neighbourhood of the origin, and the integration variables lie in the underlying space $\R^3$ and do not depend on time. In the case of (SHE), $G$ is taken to be the heat kernel in dimension $1$, and the integration variables $z=(t,x)$ take values in $\R^2$. (With the usual convention
that the heat kernel takes the value $0$ for negative times.)
In both cases, $c_\epsilon = c\epsilon^{-1}$ with a proportionality constant $c$ that depends on $\rho$ and on the equation under consideration. The other two constants behave differently according to the equation: for (PAM), $c^{(1,1)}_\epsilon= -\frac{1}{16\pi}\log\epsilon + \cO(1)$ and $c^{(1,2)}_\epsilon =  \cO(1)$; while for (SHE) both $c^{(1,1)}_\epsilon$ and $c^{(1,2)}_\epsilon$ have finite limits as $\epsilon \to 0$ as shown in \cite{Etienne}.

Let us point out that we do not provide the details on the convergence of the models. Instead, we refer the 
reader to~\cite{Etienne} where the convergence of the mollified model associated with (SHE) on the 
one-dimensional torus has been proven. Since the models are ``local'' objects, the renormalisation is 
not affected upon passing to the whole line. Regarding (PAM), the algebraic and scaling properties of 
the equation coincide with those of (SHE) so that the proof works verbatim: only the actual values of the
renormalisation constants differ.

The remainder of the article is structured as follows. We start by giving a short introduction to the theory
of regularity structures, as used in our particular example, in Section~\ref{sec:regstruct}. The reader unfamiliar
with the theory may find \cite{Hairer2014} or the shorter introductions \cite{Intro,ICM} useful.
In all existing works, the spaces of ``modelled distributions'' on which the theory is built are based
on the standard H\"older spaces. In Section~\ref{sec:Besov}, we introduce new spaces of modelled distributions that are 
instead based on a class of inhomogeneous Besov spaces and we prove the reconstruction theorem in this context.
In Section~\ref{sec:Weights}, we then leverage the local results of Section~\ref{sec:Besov} to build suitable weighted spaces.
Section~\ref{sec:Schauder} contains a Schauder estimate for these spaces, which is the main ingredient for
building local solutions to the limiting problem. Finally, we combine all of these ingredients in Section~\ref{sec:final},
where we give the proof of Theorem~\ref{Th:Main}.

\subsection{Notations}\label{sec:notations}
From now on, we work in $\R^{d+1}$ where $d$ is the dimension of space and $1$ the dimension of time. We choose the parabolic scaling $\s=(2,1,\ldots,1)$, where $\s_0=2$ stands for the time scaling and $\s_i=1,i=1\ldots d$ for the scaling of each direction of space. We let $|\s|=\sum_{i=0}^{d} \s_i$. We denote by $\|z\|_{\s} = \max(\sqrt{|t|},|x_1|,\ldots,|x_d|)$ the $\s$-scaled supremum norm of a vector $z=(t,x)\in\R^{d+1}$. We will also use the notation $|k|=\sum_{i=0}^d \s_i k_i$ for any element $k\in\N^{d+1}$. To keep notation clear, we restrict the letters $s,t$ to denoting elements in $\R$, $x,y$ to denoting elements in $\R^d$, while the letters $k,m,\ell$ will stand for elements of $\N$ or $\N^{d+1}$. Moreover, in some cases we will use the letter $z$ to denote an element in $\R^{d+1}$.

For any smooth function $f:\R^{d+1}\rightarrow\R$ and any $k\in\N^{d+1}$, we let $D^k f$ be the function 
obtained from $f$ by differentiating $k_0$ times in direction $t$ and $k_i$ times in each direction $x_i$, 
$i\in\{1,\ldots,d\}$. For any $r>0$, we let $\cC^r$ be the space of functions $f$ on $\R^{d+1}$ such that 
$D^k f$ is continuous for all $k\in\N^{d+1}$ such that $|k|\leq r$. We denote by $\cB^r$ the subset of 
$\cC^r$ whose elements are supported in the unit parabolic ball and have their $\cC^r$-norm smaller 
than $1$. For all $\eta \in \cC^r$, all $(t,x)\in \R^{d+1}$ and all $\lambda > 0$, we set
\begin{equ}
\eta^\lambda_{t,x}(s,y) := \lambda^{-|\s|} \eta\Big(\frac{s-t}{\lambda^2},\frac{y_1-x_1}{\lambda},\ldots,\frac{y_d-x_d}{\lambda}\Big) \;,\quad\forall (s,y)\in\R^{d+1}\;.
\end{equ}
This rescaling preserves the $L^1$-norm.

Finally, for all $z\in\R^{d+1}$ and all $\lambda > 0$, we let $B(z,\lambda) \subset \R^{d+1}$ be the 
ball of radius $\lambda$ centred at $z$; here we implicitly work with the $\s$-scaled supremum norm 
$\|.\|_\s$. For $x\in\R^d$, we use the same notation $B(x,\lambda)$ to denote the ball in $\R^d$ of 
radius $\lambda$ and center $x$.

\subsection*{Acknowledgements}
We are grateful to Khalil Chouk for pointing out that the regularity index for 
the Dirac mass is higher in Besov / Sobolev type spaces than in H\"older type spaces. 
MH gratefully acknowledges financial support from the Philip Leverhulme Trust 
and from the European Research Council.

\section{Regularity structures and Besov-type spaces}

In the first subsection, we recall the basic definitions of regularity structures and models - this material is essentially taken from~\cite{Hairer2014}. In the second subsection, we adapt the definition of the spaces of modelled distributions from~\cite{Hairer2014} to the setting of Besov spaces. Then, we prove the corresponding reconstruction theorem. In the third subsection, we introduce the weighted spaces of modelled distributions by adding weights around $t=0$ and $x=\infty$ in the spaces previously introduced.

\subsection{Regularity structures and models}\label{sec:regstruct}
A regularity structure consists of two objects. First, a graded vector space $\cT=\bigoplus_{\zeta\in \cA} \cT_\zeta$ where $\cA$, called the set of homogeneities, is a subset of $\R$ which is locally finite and bounded from below. Second, a group $\cG$ of continuous linear transformations of $\cT$ whose elements $\Gamma\in\cG$ fulfil the following property
\[ \Gamma \tau - \tau \in \cT_{<\beta}\;,\quad\forall \tau\in\cT_\beta\;,\quad\forall\beta\in\cA \;,\]
where we wrote $\cT_{<\beta}$ as a shorthand for $\bigoplus_{\zeta<\beta} \cT_\zeta$.
A simple example of regularity structure is given by the polynomials in $d+1$ indeterminates $X_0,\ldots, X_d$. For every $\zeta\in\N$, let $\cT_\zeta$ be the set of all formal polynomials in $X_i$, $i=0\ldots d$ with $\s$-scaled degree equal to $\zeta$. Let us recall that the $\s$-scaled degree of $X^k=\prod_{i=0}^d X_i^{k_i}$, for any given $k\in\N^{d+1}$, is equal to $|k|=\sum \s_i k_i$. The set of homogeneities in this example is $\cA=\N$, while a natural structure group is the group of translations on $\R^{d+1}$.

In the case of our original class of equations (\ref{e:E}), the regularity structure, together with a set of canonical basis vectors for $\cT$, 
can be constructed as follows. We set $\alpha=-\frac{3}{2}-\kappa$ for a given $\kappa>0$ and we let $\cT_{\alpha}$
be a one-dimensional real vector space with basis vector
$\Xi$. Then we define two collections $\cU$ and $\cF$ of formal expressions by setting $\tun, X^k \in \cU$ for all $k\in\N^{d+1}$ and by imposing that they are the smallest sets satisfying the following two rules
\begin{equs}
\tau \in \cU  \Longleftrightarrow \tau\Xi \in \cF\;,\qquad \tau \in \cF \Longrightarrow \cI(\tau) \in \cU\;.
\end{equs}
(The product $(\Xi,\tau)\mapsto \tau\Xi$ is taken to be commutative so we will also write $\Xi\tau$ instead.) 
Writing $\scal{\CU}$ for the free real vector space generated by a set $\CU$, we then set
$\cT(\cU) = \scal{\cU}$, $\cT(\cF) = \scal{\cF}$ and $\cT = \scal{\cU \cup \cF}$.
Moreover, we write $\bar{\cT} \subset \cT(\cU)$ for the set of all polynomials in the $X_i$, $i=0,\ldots,d$.

The homogeneity $|\tau|$ of an element $\tau\in\cU\cup\cF$ is computed by setting $|\Xi| = \alpha$ , $|\tun| = 0$, $|X_i| = 1$ and by imposing the following rules
\begin{equs}
|\tau\bar{\tau}| = |\tau|+|\bar{\tau}|\;,\qquad|\cI(\tau)| = |\tau| + 2\;.
\end{equs}
The corresponding sets of homogeneities are denoted $\cA(\cU)$, $\cA(\cF)$ and $\cA=\cA(\cU) \cup \cA(\cF)$.
This also yields a natural decomposition of $\cT$ by $\cT_\alpha = \scal{\{\tau\,:\, |\tau| = \alpha\}}$.
It was shown in \cite[Sec.~8]{Hairer2014} that there is a natural group $\CG$ acting on $\cT$
in a way that is compatible with the definition of an ``admissible model'', see Definition~\ref{def:admissible} below.
The precise definition of $\CG$ does not matter for the purpose of the present article, so
we refer the interested reader to \cite[Sec.~8.1]{Hairer2014} and \cite[Sec.~3.2]{Etienne}.


\begin{figure}
\begin{center}
\begin{tabular}{p{2.1cm} l  | p{2.4cm} l  }\toprule
$\cU$ & $\cA(\cU)$  & $\cF$ & $\cA(\cF)$ \\
\midrule
$\tun$ &  $0$ & $\Xi$ &  $-\frac{3}{2}-\kappa$ \\
$\cI(\Xi)$  & $\frac{1}{2}-\kappa$ & $\Xi\cI(\Xi)$ & $-1-2\kappa$ \\
$\cI(\Xi\cI(\Xi))$ & $1-2\kappa$  & $\Xi\cI(\Xi\cI(\Xi))$ & $-\frac{1}{2}-3\kappa$ \\
$X_i$ & $1$ & $\Xi X_i$ & $-\frac{1}{2}-\kappa$\\
$\cI(\Xi\cI(\Xi\cI(\Xi)))$ & $\frac{3}{2}-3\kappa$ & $\Xi\cI(\Xi\cI(\Xi\cI(\Xi)))$ & $-4\kappa$ \\
$\cI(\Xi X_i)$ & $\frac{3}{2}-\kappa$ & $\Xi\cI(\Xi X_i)$ & $-2\kappa$\\
\bottomrule
\end{tabular}\end{center}
\caption{The canonical basis vectors for the regularity structure for (\ref{e:E}) with $\gamma \in (3/2,2-4\kappa)$. Notice that here $i$ ranges in $\{1,\ldots,d\}$, while $X_0$ has homogeneity $2$ and therefore does not appear.}
\end{figure}

The regularity structure $\cT(\cU)$ is the abstract framework to which the solution $u$ 
of (\ref{e:E}) will be lifted. $\cT(\cF)$, which is simply obtained by multiplying all the elements in 
$\cT(\cU)$ by $\Xi$, will therefore allow us to lift $u\cdot\xi$. It turns out that it will suffice 
to restrict $\cT(\cU)$ to those homogeneities lower than a certain threshold $\gamma > 0$, to 
be fixed later on. Similarly, we will restrict $\cT(\cF)$ to those homogeneities lower than 
$\gamma+\alpha > 0$. We will write $\cT_{<\gamma}(\cU)$ and $\cT_{<\gamma+\alpha}(\cF)$ to 
denote these two subspaces, eventually we will omit these subscripts since the restriction 
will be clear from the context.  Finally, we let $\cQ_\zeta:\cT\rightarrow\cT_\zeta$ 
denote the canonical projection on $\cT_\zeta$ and we denote by $|a|_\zeta$ the norm of $\cQ_\zeta a$.

Let us consider the heat kernel in dimension $d$:
\[ P(t,x) := \frac{1}{(4\pi t)^{\frac{d}{2}}} e^{-\frac{|x|^2}{4t}}\;,\;\;\;x\in\R^d,\; t>0\;.\]
We will need the following decomposition of $P$ into a series of smooth functions, 
which was already used in \cite[Lem.~5.5]{Hairer2014}. Actually, there is a slight 
difference here since we consider the $\s$-scaled supremum norm in $\R^{d+1}$ instead of the 
$\s$-scaled Euclidean norm, but this makes no difference.
\begin{lemma}\label{Lemma:Kernel}
Fix $r > 0$. There exist a collection of smooth functions $P_-, P_n, n\geq 0$ on $\R_+\times\R^d$, such that the following properties hold:
\begin{enumerate}
\item\label{KernelSeries} For every $z\in\R^{d+1}\backslash\{0\}$, $P(z) = \sum_{n\geq 0} P_n(z) + P_-(z)$,
\item\label{KernelScaling} The function $P_0$ is supported in the unit ball, and for every $n\geq 0$, we have
\begin{equ}
P_n(t,x) = 2^{nd}P_0(2^{2n}t,2^n x) \;,\qquad t\in\R_+\;,\quad x\in\R^d\;,
\end{equ}
\item\label{KernelPoly} For every $n\geq 0$, we have $\int_{z} P_n(z) z^k dz = 0$ for all $k\in\N^{d+1}$ such that $|k| \leq r$.
\end{enumerate}
As a consequence, for every $k\in\N^{d+1}$, there exists $C>0$ such that
\begin{equs}\label{Eq:ScalingKernel}
|D^k P_n(z)| \leq C 2^{n(d+|k|)}\;,
\end{equs}
uniformly over all $n\geq 0$ and all $z\in \R^{d+1}$.
\end{lemma}
We will use the notation $P_+=\sum_{n\geq 0} P_n$.

From now on, we deal with $\cT_{<\gamma}$ for a given $\gamma$ that will be fixed later on. To simplify notation, we will omit the subscript $\gamma$. We now associate to our regularity structure $(\cT,\cG)$ some analytical features. To that end, recalling the definition of the sets of test functions $\cB^r$ in Section~\ref{sec:notations},
we introduce a set of \textit{admissible models} $\cM$. 

\begin{definition}\label{def:admissible}
An admissible model is a pair $(\Pi,\Gamma)$ that satisfies the following assumptions:
\begin{enumerate}
\item The map $\Pi \colon z\mapsto\Pi_z$ goes from $\R^{d+1}$ into the space $L(\cT,\cD'(\R^{d+1}))$ 
of linear transformations from $\cT$ into 
distributions on space-time $\cD'(\R^{d+1})$ such that
\begin{equs}\label{Eq:BoundPi}
\left\| \Pi\right\|_z  &:= \sup_{\eta\in\cB^r}\sup_{\lambda\in(0,1]}\sup_{\zeta\in\cA}\sup_{\tau\in \cT_\zeta}\frac{|(\Pi_{z} \tau)\big(\eta^\lambda_{z}\big)|}{|\tau| \, \lambda^{\zeta}} \lesssim 1\;,
\end{equs}
locally uniformly over $z \in \R^{d+1}$, for some fixed $r > |\alpha|$. We then define $\left\| \Pi \right\|_B$ as the supremum of $\left\| \Pi \right\|_z $ over all $z \in B$, where $B$ is a given subset of $\R^{d+1}$.
\item The map $\Gamma\colon (z,z') \mapsto \Gamma_{z,z'}$ goes from $\R^{d+1}\times\R^{d+1}$ into $\cG$. It is such that
\begin{equs}\label{Eq:BoundGamma}
\left\| \Gamma \right\|_{z,z'} &:= \sup_{\beta\leq \zeta}\sup_{\tau\in \cT_\zeta}\,\frac{|\Gamma_{z,z'} \tau|_{\beta}}{|\tau| \,\|z-z'\|_{\s}^{\zeta-\beta}} \lesssim 1\;,
\end{equs}
locally uniformly over $z,z' \in \R^{d+1}$ such that $\|z-z'\|_{\s}\leq 1$. We let $\left\| \Gamma \right\|_B := \sup_{z,z'\in B}\left\| \Gamma \right\|_{z,z'}$ for any $B\subset\R^{d+1}$.
\item For every $z,z' \in \R^{d+1}$
\begin{equs}\label{Eq:AlgPpty}
\Pi_{z} \Gamma_{z,z'} = \Pi_{z'}\;.
\end{equs}
\item For every $k\in\N^{d+1}$ we have the identities
\begin{equs}
(\Pi_{z} X^k)(z') &= (z'-z)^k \;, \label{Eq:CondModel2}\\
(\Pi_{z}\cI\tau)(z') &= \big\langle \Pi_{z}\tau , P_+\big(z'-\cdot\big)\big\rangle - \sum_{|k| < |\cI\tau|} \frac{\big(z'-z\big)^k}{k!}\big\langle \Pi_{z}\tau , D^k P_+\big(z -\cdot\big)\big\rangle\;.
\end{equs}
\end{enumerate}
\end{definition}

\begin{remark}
It is not clear a priori that the last point in this definition makes sense, since $P_+$ is not a smooth
test function. One should interpret expressions of the type $\scal{\mu,P_+}$ for a distribution $\mu$ as a shorthand for
$\sum_{n \ge 0}\scal{\mu,P_n}$ (and similarly for expressions involving $D^kP_+$). The bound \eqref{Eq:BoundPi}
then guarantees that these sums converge absolutely.
\end{remark}

The mere existence of non-trivial admissible models is not obvious. However, it turns out that every smooth function
$\xi_\eps$ can be lifted in a canonical way to an admissible model $(\Pi^{(\eps)}, \Gamma^{(\eps)})$ 
by setting
\begin{equs}
(\Pi^{(\eps)}_z \Xi)(z') = \xi_\epsilon(z') \;,\qquad (\Pi^{(\eps)}_z \tau \bar\tau)(z') = (\Pi^{(\eps)}_z \tau)(z')(\Pi^{(\eps)}_z\bar\tau)(z') \;,\quad \forall \tau,\bar\tau\in \cT\;,
\end{equs}
and then imposing (\ref{Eq:CondModel2}). Observe that all the products 
appearing in this definition are well-defined since $\xi_\epsilon$ is a function. 
It was shown in \cite[Prop~8.27]{Hairer2014} that this is indeed an admissible model and we will henceforth
refer to this model as the ``canonical model'' associated to $\xi_\eps$.

\begin{notation}
From now on, instead of writing $\Gamma_{(t,x) , (t,y)}$, we will simply write $\Gamma_{x,y}^t$. Similarly, we will write $\Gamma_{t,s}^x$ instead of $\Gamma_{(t,x), (s,x)}$.
\end{notation}

\subsection{The reconstruction theorem in a Besov-type space}\label{sec:Besov}

In order to build solution to our SPDEs, we need to introduce appropriate spaces of distributions. 
For the moment, we consider un-weighted spaces for the sake of clarity, but we will consider weighted versions later on. We refer the reader to Section \ref{sec:notations} for the notations.
\begin{definition}\label{Def:HolderLp}
Let $\alpha < 0$ and $p\in [1,\infty]$. We let $\cE^{\alpha,p}$ be the space of distributions $f$ on $\R^{d+1}$ such that
\begin{equ}
\| f \|_{\alpha,p} := \sup_{\lambda \in (0,1]} \sup_{t\in\R} \bigg\| \sup_{\eta \in \cB^r(\R^{d+1})} \frac{|\langle f , \eta^\lambda_{t,x} \rangle|}{\lambda^\alpha} \bigg\|_{L^p(\R^d,dx)} < \infty\;.
\end{equ}
\end{definition}
When $p=\infty$, we implicitly consider the supremum-norm instead of the $L^\infty$-norm. In that case, $\cE^{\alpha,\infty}$ actually coincides with the H\"older space $\cC^\alpha(\R^{d+1})$: this can be deduced from the forthcoming wavelet characterisation of Proposition \ref{Prop:CharactDistrib} which coincides with the wavelet characterisation of $\cC^\alpha(\R^{d+1})$ stated, for instance, in~\cite[Section 6.10]{Meyer}. In the case $p<\infty$, our space $\cE^{\alpha,p}$ does not coincide with the usual Besov space $B^\alpha_{p,\infty}$ for our special treatment of the time variable: again, this can be seen by comparing the wavelet characterisations of these two spaces.\\
In order to deal with random distributions, it is more convenient to have a countable characterisation of the spaces $\cE^{\alpha,p}$. To that end, we rely on a wavelet analysis that we briefly summarise below; we refer to the works of Meyer~\cite{Meyer} and Daubechies~\cite{Ingrid} for more details on wavelet analysis.
\paragraph{Wavelet analysis.} For every $r>0$, there exists a compactly supported function $\phi\in\cC^r(\R)$ such that:
\begin{enumerate}
\item We have $\langle\phi(\cdot),\phi(\cdot-k)\rangle=\delta_{k,0}$ for every $k\in\Z$,
\item There exist $\tilde{a}_k,k\in\Z$ with only finitely many non-zero values, and such that $\phi(x)=\sum_{k\in\Z}\tilde{a}_k\phi(2x-k)$ for every $x\in\R$,
\item For every polynomial $P$ of degree at most $r$ and for every $x \in \R$,
\begin{equ}
\sum_{k\in\Z}\int P(y) \phi(y-k) dy \,\phi(x-k) = P(x) \;.
\end{equ}
\end{enumerate}
Given such a function $\phi$, we define for every $(t,x)\in\R^{d+1}$ the recentered and rescaled function $\phi_{t,x}^n$ as follows
\[ \phi_{t,x}^n(s,y)=2^{n}\phi\big(2^{2n}(s-t)\big)\prod_{i=1}^{d}2^{\frac{n}{2}}\phi\big(2^n(y_i-x_i)\big) \;.\]
Observe that this rescaling preserves the $L^2$-norm. We let $V_n$ be the subspace of $L^2(\R^{d+1})$ generated by $\{\phi_{t,x}^n:(t,x)\in\Lambda_n\}$ where
\[ \Lambda_n:=\big\{(2^{-2n}k_0,2^{-n}k_1,\ldots,2^{-n}k_d):k_i\in\Z\big\} \;.\]
Using the second property above, we deduce that
\begin{equs}\label{Eq:Waveletak}
\phi_{t,x}^n = \sum_k a_k \phi^{n+1}_{(t,x)_{n,k}}\;,\qquad (t,x)_{n,k} = (t,x)+k2^{-(n+1)}\;,
\end{equs}
where 
only finitely many of the $a_k$'s are non-zero, and for every $k\in\Z^{d+1}$
\begin{equ}
k2^{-(n+1)} = (k_0 2^{-2(n+1)},k_1 2^{-(n+1)},\ldots, k_d 2^{-(n+1)})\;.
\end{equ}
Using the third property above, we deduce that for every $n\geq 0$, $V_n$ contains all polynomials of scaled degree less or equal to $r$. 

Another important property of wavelets is the existence of a finite set $\Psi$ of compactly supported functions in $\cC^r$ such that, for every $n\geq 0$, the orthogonal complement of $V_n$ inside $V_{n+1}$ is given by the linear span of all the $\psi^n_x, x\in \Lambda_n, \psi\in\Psi$. Necessarily, by the third property above,
each of the functions $\psi\in\Psi$ annihilates all polynomials of $\s$-scaled degree less than or equal to $r$. Finally, for every $n\geq 0$
\[ \{\phi_{t,x}^n:(t,x)\in\Lambda_n\} \cup \{\psi_{t,x}^m:m\geq n, \psi\in\Psi, (t,x)\in\Lambda_m\} \;,\]
forms an orthonormal basis of $L^2(\R^{d+1})$. 

This wavelet analysis allows one to identify a countable collection of conditions that determines the regularity of a distribution. The next proposition is in the flavour of classical results on the characterisation of Besov spaces in terms of a wavelet analysis, we refer the reader to~\cite{FraJaw,Meyer} among other references.

\begin{proposition}\label{Prop:CharactDistrib}
Let $\alpha<0$, $p\in[1,\infty]$ and $r > |\alpha|$. Let $\xi$ be a distribution on $\R^{d+1}$. Set $a^{n,\psi}_{t,x} := \langle \xi,\psi_{t,x}^{n}\rangle$ for all $(t,x) \in \Lambda_n$, $n\geq 1$, $\psi\in\Psi$ as well as $b^0_{t,x} := \langle\xi,\phi_{t,x}\rangle$ for all $(t,x) \in \Lambda_0$. If $\xi \in \cE^{\alpha,p}$, then we have the bounds
\begin{equs}[e:characterisationEap]
\sup_{\psi\in\Psi}\sup_{n\in\N}\sup_{t\in 2^{-2n}\Z} \bigg(\sum_{x:(t,x)\in \Lambda_n} 2^{-nd} \Big|\frac{a^{n,\psi}_{t,x}}{2^{-\frac{n|\s|}{2}-n\alpha}}\Big|^p\bigg)^{\frac{1}{p}} &< \infty\;,\\
\sup_{t\in \Z}\bigg(\sum_{x:(x,t)\in\Lambda_0}\big|b^0_{t,x} \big|^p\bigg)^{\frac{1}{p}} &< \infty\;.
\end{equs}
Conversely, to any sequences $a^{n,\psi}_{t,x}$ and $b^0_{t,x}$ satisfying the bounds (\ref{e:characterisationEap}), one can associate a distribution $\xi \in \cE^{\alpha,p}$ by setting
\begin{equ}[e:defxiWavelets]
\xi := \sum_{\psi\in\Psi}\sum_{n\geq 0}\sum_{(t,x)\in\Lambda_n} a^{n,\psi}_{t,x} \psi^n_{t,x} + \sum_{(t,x)\in\Lambda_0} b^0_{t,x} \phi_{t,x}\;.
\end{equ}
\end{proposition}

%
%

\begin{remark}
As an immediate consequence of this result, we have a continuous embedding of $\cE^{\alpha,p}$ into $\cE^{\alpha-\frac{d}{p},\infty}$, for every $p\in[1,\infty)$.
\end{remark}
\begin{proof}
The case $p=\infty$ is covered by Proposition 3.20 in~\cite{Hairer2014}. Let us adapt the proof for the case 
$p\in[1,\infty)$. If $\xi\in\cE^{\alpha,p}$, then it is immediate to see that the bounds \eqref{e:characterisationEap} 
are satisfied, using the simple fact that for any $(s,y)$ lying in the parabolic hypercube of 
sidelength $2^{-n}$ centred around $(t,x) \in \Lambda_n$, the function $\psi_{t,x}^{n}$ is of the form 
$\eta^{\lambda}_{s,y}$ with $\lambda = 2^{-n}$, up to a constant multiplicative factor 
of the order $2^{-{n|\s|\over 2}}$. This allows in particular to turn the $L^p$ norm in space
into an $\ell^p$ norm at the expense of the corresponding volume factor.

Let us now prove the more difficult converse implication. For $\lambda \in (0,1]$, let 
$n_0\geq 0$ be the largest integer such that $2^{-n_0}\geq\lambda$. We need to show that the series
\begin{equs}
\sum_{\psi\in\Psi}\sum_{n\geq 0}\sum_{(s,y)\in\Lambda_n} a^{n,\psi}_{s,y} \langle \psi^n_{s,y},\eta^\lambda_{t,x}\rangle + \sum_{(s,y)\in\Lambda_0} b^0_{s,y} \langle \phi_{s,y},\eta^\lambda_{t,x}\rangle \;,
\end{equs}
converges for any test function $\eta\in\cB^r$ and any $\lambda \in (0,1]$, and that the bound of Definition \ref{Def:HolderLp} is fulfilled. Once this is established, it is simple to check that (\ref{e:defxiWavelets}) defines a distribution $\xi$ (necessarily in $\cE^{\alpha,p}$) and that the sequences $a^{n,\psi}_{t,x}$ and $b^0_{t,x}$ coincide with the coefficients of $\xi$ on the wavelet basis.

If $n\geq n_0$, we use the fact that $\psi$ kills polynomials of degree $r$ to get the bound
\begin{equs}
\sup_{\eta\in\cB^r}|\langle \psi^n_{s,y},\eta^\lambda_{t,x}\rangle| \lesssim 2^{-(n-n_0)(r+\frac{|\s|}{2}) + n_0\frac{|\s|}{2}}\;,
\end{equs}
uniformly over all the parameters. Observe that the left hand side actually vanishes as soon as $\|(t-s,x-y)\|_{\s}\ge C 2^{-n_0}$, for some positive constant $C$ that only depends on the size of the support of $\psi$. For a given $(t,x)\in\R^{d+1}$, there are at most $2^{2(n-n_0)}$ such $s$'s in $2^{-2n}\Z$, and $2^{d(n-n_0)}$ such $y$'s in $2^{-n}\Z^d$. Consequently, using Jensen's inequality at the third line we obtain
\begin{equs}
\bigg\|\sum_{(s,y)\in\Lambda_n} &\sup_{\eta\in\cB^r}\frac{\big|a^{n,\psi}_{s,y} \langle \psi^n_{s,y},\eta^\lambda_{t,x}\rangle\big|}{\lambda^\alpha}\bigg\|_{L^p(dx)}\\
&\lesssim \sup_{\substack{s\in 2^{-2n}\Z\\ |t-s| \leq C 2^{-2n_0}}}\bigg\|\sum_{\substack{y:(s,y)\in\Lambda_n\\ |x-y| \leq C 2^{-n_0}}} 
\frac{|a^{n,\psi}_{s,y}|}{\lambda^\alpha} 2^{-(n-n_0)(r+d)+n\frac{|\s|}{2}}\bigg\|_{L^p(dx)}\\
&\lesssim \sup_{s\in 2^{-2n}\Z} \Big(\sum_{y:(s,y)\in\Lambda_n}2^{-nd} \Big|\frac{a^{n,\psi}_{s,y}}{2^{-\frac{n|\s|}{2}-n\alpha}}\Big|^p \Big)^{\frac{1}{p}}2^{-(n-n_0)(r+\alpha)}\;,
\end{equs}
uniformly over all $t\in\R$ and all $n\geq n_0$. Therefore, since $r$ was chosen sufficiently large
so that $r+\alpha > 0$, the sum over $n\geq n_0$ converges.

On the other hand, for $n<n_0$, we have the bound
\begin{equs}
\sup_{\eta\in\cB^r}|\langle \psi^n_{s,y},\eta^\lambda_{t,x}\rangle| \lesssim 2^{n\frac{|\s|}{2}}\;,
\end{equs}
uniformly over all the parameters. Moreover, the left hand side vanishes as soon as $\|(t-s,x-y)\|_{\s} > C2^{-n}$. Consequently, only a finite number of $(s,y)\in\Lambda_n$ yield a non-zero contribution, uniformly over all $(t,x)\in\R^{d+1}$ and all $n<n_0$. An elementary computation using Jensen's inequality gives the bound
\begin{equs}
\bigg\|\sum_{(s,y)\in\Lambda_n} \sup_{\eta\in\cB^r} &\frac{\big|a^{n,\psi}_{s,y} \langle \psi^n_{s,y},\eta^\lambda_{t,x}\rangle\big|}{\lambda^\alpha}\bigg\|_{L^p(dx)}\\
&\lesssim \sup_{s\in 2^{-2n}\Z}\Big(\sum_{y:(s,y)\in\Lambda_n}2^{-nd} \Big|\frac{a^{n,\psi}_{s,y}}{2^{-n\frac{|\s|}{2}-n\alpha}}\Big|^p\Big)^{\frac{1}{p}} 2^{-(n-n_0)\alpha}\;,
\end{equs}
uniformly over all $n<n_0$ and all $t\in\R$. The sum over all $n<n_0$ of the last expression is therefore uniformly bounded in $n_0$ and $t$.
Finally, the contribution of the $\phi_{s,y}$'s is treated similarly as the case $n<n_0$.
\end{proof}
Given a regularity structure $(\cT,\cG)$ and a model $(\Pi,\Gamma)$, we now define a space of 
modelled distributions which mimics the space $\cE^{\alpha,p}$.
\begin{definition}
Let $\gamma >0$ and $p\in[1,\infty)$. The space $\cD^{\gamma,p}$ consists of those maps $f:\R^{d+1}\rightarrow\cT_{<\gamma}$ such that
\begin{equs}
\Big\||f(t,x)|_\zeta\Big\|_{L^p(\R^d,dx)}
&+ \bigg\|\int_{y\in B(x,\lambda)}\frac{|f(t,y)-\Gamma_{y,x}^t f(t,x)|_\zeta}{\lambda^{\gamma-\zeta}}\, \lambda^{-d} dy\bigg\|_{L^p(\R^d,dx)}\\
&+ \bigg\|\frac{|f(t,x)-\Gamma_{t,t-\lambda^2}^x f(t-\lambda^2,x)|_\zeta}{\lambda^{\gamma-\zeta}}\bigg\|_{L^p(\R^d,dx)} < \infty\;,
\end{equs}
uniformly over all $t\in\R$, all $\zeta \in \cA$ and all $\lambda\in (0,2]$. We denote by 
$\|f\|_{\gamma,p}$ the corresponding norm.
\end{definition}
For all $B\subset \R^{d+1}$ of the form $[s,t]\times B(x_0,L)$, we will use the notation $\| f \|_B$ to denote the supremum of the terms appearing in the $\cD^{\gamma,p}$-norm of $f$, but with the additional constraint that the time indices are restricted to $[s,t]$ and the $L^p(\R^d)$-norms are replaced by the $L^p$-norm on the ball $B(x_0,L)$.

\begin{remark}
Our spaces $\cD^{\gamma,p}$ are the $L^p$ counterparts of the space $\cD^{\gamma,\infty}=\cD^{\gamma}$ from 
\cite[Def.~3.1]{Hairer2014}. Notice also that, just as in the definition of $\cE^{\alpha,p}$, we treat space and time translations
separately: this will be useful in the weighted setting later on. 
\end{remark}

The definition of the space $\cD^{\gamma,p}$ depends implicitly on the underlying model through $\Gamma$. In order to compare two elements $f\in\cD^{\gamma,p}$ and $\bar{f}\in\bar{\cD}^{\gamma,p}$ associated to two models $(\Pi,\Gamma)$ and $(\bar{\Pi},\bar{\Gamma})$, we introduce $\|f;\bar{f}\|_{\gamma,p}$ as the supremum of
\begin{equs}
{}&\Big\||f(t,x)-\bar{f}(t,x)|_\zeta\Big\|_{L^p(dx)}\\
+&\Big\|\int_{y\in B(x,\lambda)}\frac{|f(t,y)-\bar{f}(t,y)-\Gamma_{y,x}^tf(t,x)+\bar{\Gamma}_{y,x}^t\bar{f}(t,x)|_\zeta}{\lambda^{\gamma-\zeta}}\,\lambda^{-d}dy\Big\|_{L^p(dx)}\\
+& \Big\|\frac{|f(t,x)-\bar{f}(t,x)-\Gamma_{t,t-\lambda^2}^xf(t-\lambda^2,x)+\bar{\Gamma}_{t,t-\lambda^2}^x\bar{f}(t-\lambda^2,x)|_\zeta}{\lambda^{\gamma-\zeta}}\Big\|_{L^p(dx)} \;,
\end{equs}
over all $t\in\R$, all $\zeta \in \cA$ and all $\lambda\in (0,2]$.

\medskip

The following result shows that these modelled distributions can actually be \textit{reconstructed} into genuine distributions. This is a modification of Theorem 5.12 in~\cite{Hairer2014}. For any function $g:\R^d\rightarrow\R$ and any $x_0\in\R^d$, we use the notation
\begin{equ}
\|g\|_{L^p_{x_0,1}} = \Big( \int_{x\in B(x_0,1)} |g(x)|^p dx\Big)^{\frac{1}{p}}\;.
\end{equ}
\begin{theorem}[Reconstruction]\label{Th:Reconstruction}
Let $(\cT,\cG,\cA)$ be a regularity structure. Let $\gamma > 0$, $p\in[1,\infty)$, $\alpha:=\min \cA < 0$, $r>|\alpha|$ and $(\Pi,\Gamma)$ be a model. There exists a unique continuous linear map $\cR:\cD^{\gamma,p}\rightarrow\cE^{\alpha,p}$ such that
\begin{equs}\label{Eq:BoundReconstructionp}
\bigg\| \sup_{\eta\in\cB^r}\big|\langle\cR f - \Pi_{t,x} f(t,x) , \eta_{t,x}^\lambda \rangle\big| \bigg\|_{L^p_{x_0,1}}  \lesssim \lambda^\gamma C_{t,x_0,\lambda}(\Pi,f)\;,
\end{equs}
uniformly over all $\lambda\in(0,1]$, all $(t,x_0)\in\R^{d+1}$, all $f\in\cD^{\gamma,p}$ and all admissible models $(\Pi,\Gamma)$. Here the proportionality constant can be given by
\begin{equ}[Eq:RefinedBoundReconstructionp]
C_{t,x_0,\lambda}(\Pi,f) = \sum_{2^{-n} \leq \lambda} \Big(\frac{2^{-n}}{\lambda}\Big)^{\gamma\wedge(r+\alpha)} \| \Pi \|_{B^n_{\lambda,t,x_0}}\big(1+\| \Gamma \|_{B^n_{\lambda,t,x_0}}\big) \|f \|_{B^n_{\lambda,t,x_0}}\;,
\end{equ}
with $B^n_{\lambda,t,x_0} = \big[t-2\lambda^2,t+\lambda^2-2^{-2n}\big] \times B(x_0,3)$.

If $(\bar{\Pi},\bar{\Gamma})$ is a second model for $\cT$ and if $\bar{\cR}$ is its associated reconstruction operator, then one has the bound
\begin{equs}[Eq:BoundReconstructionTwo]
\bigg\| \sup_{\eta\in\cB^r}\big|&\langle \cR f - \bar{\cR}\bar{f} - \Pi_{t,x} f(t,x) + \bar{\Pi}_{t,x}\bar{f}(t,x) , \eta_{t,x}^\lambda \rangle\big|^p \bigg\|_{L^p_{x_0,1}}\\
&\lesssim \lambda^\gamma C_{t,x_0,\lambda}(\Pi,\bar\Pi,f,\bar f)\;,
\end{equs}
uniformly over all $\lambda\in(0,1]$, all $f\in\cD^{\gamma,p}$, all $\bar{f}\in\bar{\cD}^{\gamma,p}$, all $(t,x_0)\in\R^{d+1}$ and all admissible models $(\Pi,\Gamma)$, $(\bar\Pi,\bar\Gamma)$. Here, the proportionality constant is obtained from (\ref{Eq:RefinedBoundReconstructionp}) by replacing $\| \Pi \|_{B^n_{\lambda,t,x_0}} \big(1+\| \Gamma \|_{B^n_{\lambda,t,x_0}}\big)\|f \|_{B^n_{\lambda,t,x_0}}$ by
\begin{equs}[Eq:BoundNormModels]
\| \Pi \|_{B^n}&(1+\| \Gamma \|_{B^n}) \|f;\bar f \|_{B^n}\\
&+\big(\| \Pi-\bar \Pi \|_{B^n}(1+\| \Gamma \|_{B^n})+\| \bar \Pi \|_{B^n}\| \Gamma-\bar\Gamma \|_{B^n}\big)\|\bar f \|_{B^n}\;,
\end{equs}
with $B^n=B^n_{\lambda,t,x_0}$ as defined above.
\end{theorem}
To prove this theorem, we adapt the arguments from \cite[Th 3.10]{Hairer2014}. In particular, we obtain $\cR f$ as the limit of a sequence $\cR_n f \in V_n$, where $V_n$ is the subspace of $L^2(\R^{d+1})$ defined by our wavelet analysis. Let us comment on the technical bound (\ref{Eq:RefinedBoundReconstructionp}). Its purpose is to provide a precise control on the time-locations of these values $f(s,y)$ needed to define $\langle\cR f , \eta_{t,x}^\lambda\rangle$. In the original proof of the reconstruction theorem \cite[Th 3.10]{Hairer2014}, these points were taken in a domain slightly larger than the support of the test function $\eta_{t,x}^\lambda$. In the setting with weights, this would only allow us to weigh $\langle\cR f , \eta_{t,x}^\lambda\rangle$ by a weight taken at a time slightly larger than the maximal time of the support of the test function. In our present approach, the values $f(s,y)$ used for the term coming from $\langle \cR_n f , \eta_{t,x}^{\lambda} \rangle$ will always be such that $s<t+\lambda^2-2^{-2n}$. In the setting with weights, this will allow us to weigh $\langle\cR f , \eta_{t,x}^\lambda\rangle$ by a weight taken at time $t+\lambda^2$. We refer to Figure \ref{Fig} for an illustration.

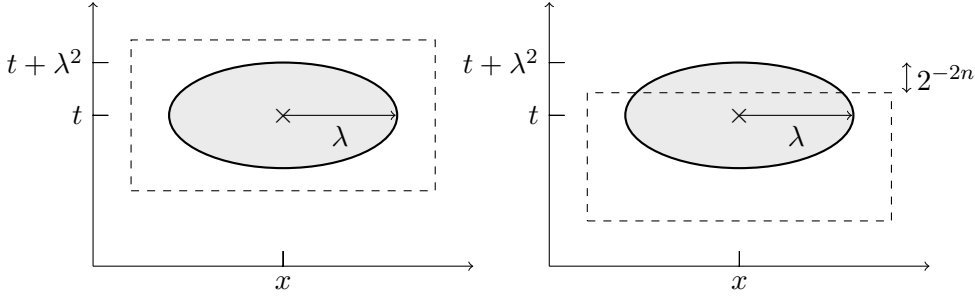
\begin{figure}
\begin{tikzpicture}

\draw[dashed] (-7.0,-1.0) rectangle (-3.0,1.0);
\draw[thick,fill=black!8] (-5.0,0) ellipse (1.5cm and 0.7cm);
\draw [->] (-5,0) node[] {$\times$} -- (-4.5,0) node[below right] {$\lambda$} -- (-3.52,0);
\draw [->] (-7.5,-2.0) -- (-7.5,0) -- (-7.3,0) -- (-7.5,0) node[left] {$t$} -- (-7.5,0.7)  -- (-7.3,0.7) -- (-7.5,0.7) node[left] {$t+\lambda^2$} -- (-7.5,1.5);
\draw [->] (-7.5,-2.0) -- (-5,-2.0) -- (-5,-1.8) -- (-5,-2) node[below] {$x$}  -- (-2.5,-2.0);

\draw[thick,fill=black!8] (1,0) ellipse (1.5cm and 0.7cm);
\draw[dashed] (-1,-1.4) rectangle (3,0.3);
\draw [->] (1,0) node[] {$\times$} -- (1.5,0) node[below right] {$\lambda$} -- (2.48,0);
\draw [->] (-1.5,-2.0) -- (-1.5,0)  -- (-1.3,0) -- (-1.5,0) node[left] {$t$} -- (-1.5,0.7) -- (-1.3,0.7) -- (-1.5,0.7) node[left] {$t+\lambda^2$} -- (-1.5,1.5);
\draw [->] (-1.5,-2.0) -- (1,-2.0) -- (1,-1.8) -- (1,-2) node[below] {$x$}  -- (3.5,-2.0);
\draw [<->] (3.2,0.3) -- (3.2,0.5) node[right] {$2^{-2n}$} -- (3.2,0.7);

\end{tikzpicture}
\caption{\label{Fig}Reconstruction theorem. On the left, the original approach and on the right, the approach presented in our proof. The shaded region depicts the support of a test function $\eta_{t,x}^\lambda$, the dashed box is the domain of the evaluations of the modelled distribution $f$ required to define $\langle \cR_n f , \eta_{t,x}^\lambda \rangle$.}
\end{figure}

The core of the proof rests on the following result. Recall the wavelet analysis introduced above. Let $f_n = \sum_{(t,x)\in\Lambda_n} A^n_{t,x} \varphi^n_{t,x}$ be a sequence of elements in $V_n$ and define $\delta A^n_{t,x} = \langle f_{n+1}-f_n , \varphi^n_{t,x}\rangle$. The following criterion for the convergence of the sequence $f_n$ is an adaptation of Theorem 3.23 in~\cite{Hairer2014}, which in turn can be viewed as a multidimensional generalisation of Gubinelli's sewing lemma~\cite{Gub04}. 
\begin{proposition}\label{Prop:Reconstruction}
Let $\alpha < 0$. Assume that there exists a constant $\|A\|$ such that
\begin{equs}[Eq:HypSewing]
\sup_{n\geq 0}\sup_{t\in 2^{-2n}\Z}\bigg(\sum_{x:(t,x)\in\Lambda_n} 2^{-nd} \Big|\frac{A_{t,x}^n}{2^{-n\frac{|\s|}{2}-n\alpha}}\Big|^p\bigg)^{\frac{1}{p}} &\leq \left\| A\right\|\;,\\
\sup_{n\geq 0}\sup_{t\in 2^{-2n}\Z} \bigg(\sum_{x:(t,x)\in\Lambda_n} 2^{-nd} \Big|\frac{\delta A_{t,x}^n}{2^{-n\frac{|\s|}{2}-n\gamma}}\Big|^p\bigg)^{\frac{1}{p}} &\leq \left\| A\right\| \;.
\end{equs}
Then, the sequence $f_n$ converges in $\cE^{\bar\alpha,p}$ for every $\bar\alpha < \alpha$ to a limit $f\in\cE^{\alpha,p}$. Moreover, the bounds
\begin{equs}\label{Eq:RateReconstr}
\|f-f_n\|_{\bar\alpha,p} \lesssim \|A\| 2^{-n(\alpha-\bar\alpha)} \;\;\;,\;\;\; \|\cP_n f - f_n\|_{\alpha,p} \lesssim \|A\| 2^{-n\gamma} \;,
\end{equs}
hold for $\bar\alpha \in (\alpha-\gamma,\alpha)$.
\end{proposition}
Here, $\cP_n$ denotes the orthogonal projection from $L^2(\R^{d+1})$ onto $V_n$. We also write $V_n^{\perp}$ for the orthogonal complement of $V_n$ in $V_{n+1}$. From the wavelet analysis, we know that this is obtained as the linear span of all the $\psi^n_{t,x}$ with $(t,x)\in\Lambda_n$ and $\psi\in\Psi$.
\begin{proof}
Let us write $f_{n+1}-f_n = g_n + \delta f_n$, where $g_n \in V_n$ and $\delta f_n \in V_n^{\perp}$. We bound separately the contributions of these two terms. By Proposition~\ref{Prop:CharactDistrib}, the $\cE^{\beta,p}$ norm 
is equivalent to the supremum over $n\geq 0$ of the $\cE^{\beta,p}$ norms of the projections onto 
$V_n^\perp$ and onto $V_0$. Therefore, the sequence $\sum_{n=0}^M \delta f_n$ converges in $\cE^{\bar{\alpha},p}$ 
as $M \to \infty$ to an element in $\cE^{\alpha,p}$ precisely if
\begin{equs}\label{Eq:CVdeltaf}
\lim_{n\rightarrow\infty} \| \delta f_n \|_{\bar\alpha,p} = 0\;,\quad \sup_{n\geq 0} \| \delta f_n \|_{\alpha,p} < \infty\;.
\end{equs} 
We have
\begin{equs}
\langle \delta f_n , \psi^n_{t,x} \rangle = \sum_{(s,y)\in\Lambda_{n+1}} A^{n+1}_{s,y} \langle \varphi_{s,y}^{n+1} , \psi^n_{t,x} \rangle \;.
\end{equs}
Observe that $|\langle \varphi_{s,y}^{n+1} , \psi^n_{t,x} \rangle| \lesssim 1$ uniformly over all $n\geq 0$, and that the inner product vanishes as soon as $\|(t-s,x-y)\|_{\s} > C 2^{-n}$ for some constant $C>0$ depending on the sizes of the support of $\varphi$ and $\psi$. Hence, for a given $(t,x)$, the number of $(s,y)\in\Lambda_{n+1}$ with a non-zero contribution is uniformly bounded in $n\geq 0$. Therefore, we have
\begin{equs}
\| \delta f_n \|_{\beta,p} &\lesssim \sup_{t\in 2^{-2n}\Z}\bigg( \sum_{x:(t,x)\in\Lambda_n} 2^{-nd} \Big( \sum_{\substack{(s,y)\in\Lambda_{n+1}\\\|(t-s,x-y)\|_{\s}\leq C 2^{-n}}} \frac{|A^{n+1}_{s,y}|}{2^{-n\frac{|\s|}{2}-n\beta}}\Big)^p\bigg)^{\frac{1}{p}}\\
&\lesssim \sup_{t\in 2^{-2n}\Z}\sum_{\substack{s\in 2^{-2(n+1)}\Z\\ |t-s|\leq C^2 2^{-2n}}}\bigg( \sum_{x:(t,x)\in\Lambda_n}\sum_{\substack{y:(s,y)\in\Lambda_{n+1}\\ |x-y|\leq C 2^{-n}}} 2^{-nd} \Big|  \frac{A^{n+1}_{s,y}}{2^{-n\frac{|\s|}{2}-n\beta}}\Big|^p\bigg)^{\frac{1}{p}}\\
&\lesssim \sup_{s\in 2^{-2(n+1)}\Z}2^{-n(\alpha-\beta)} \bigg( \sum_{y:(s,y)\in\Lambda_{n+1}} 2^{-(n+1)d} \Big|\frac{A^{n+1}_{s,y}}{2^{-n\frac{|\s|}{2}-n\alpha}}\Big|^p\bigg)^{\frac{1}{p}}\;,\\
\end{equs}
so that (\ref{Eq:CVdeltaf}) follows from (\ref{Eq:HypSewing}). Moreover, this yields the bound
\begin{equ}
\Big\| \sum_{n=m}^\infty\delta f_n \Big\|_{\bar\alpha,p} \lesssim \|A\| 2^{-m(\alpha-\bar\alpha)}\;.
\end{equ}
Let us now prove that the series of the $g_n$'s is also summable in $\cE^{\alpha,p}$. We have
\begin{equs}
\left\| \sum_{n=m}^{M} g_n \right\|_{\alpha,p} \lesssim \sum_{n=m}^M \sup_{N\geq 0} \| Q_N g_n \|_{\alpha,p} \vee \| \cP_0 g_n \|_{\alpha,p}\;,
\end{equs}
where $Q_N$ denotes the projection onto $V_N^\perp$ and $\cP_0$ the projection onto $V_0$. Since $g_n\in V_n$, we have
\begin{equs}
g_n = \sum_{(s,y)\in\Lambda_n} \langle g_n, \varphi_{s,y}^n\rangle \varphi_{s,y}^n = \sum_{(s,y)\in\Lambda_n} \delta A^n_{s,y} \varphi_{s,y}^n\;. 
\end{equs}
Whenever $N\geq n$, $Q_N g_n$ vanishes. On the other hand, we have $|\langle \varphi_{s,y}^n , \psi_{t,x}^N\rangle | \lesssim 2^{-(n-N)\frac{|\s|}{2}}$ uniformly over all $N < n$, and this inner product actually vanishes as soon as $\|(t-s,x-y)\|_{\s} > C 2^{-N}$. Consequently, using the triangle inequality on the sum over $s$ and Jensen's inequality on the sum over $y$ to pass from the third to the fourth line, we have
\begin{equs}
{}&\| Q_N g_n \|_{\alpha,p}\\
&\lesssim \sup_{t\in 2^{-2N}\Z} \bigg( \sum_{x:(t,x)\in\Lambda_N} 2^{-Nd} \Big( \sum_{(s,y)\in\Lambda_n} \frac{|\delta A^n_{s,y}| |\langle \varphi^n_{s,y},\psi^N_{t,x}\rangle|}{2^{-N\frac{|\s|}{2}-N\alpha}}\Big)^p \bigg)^{\frac{1}{p}}\\
&\lesssim \sup_{t\in 2^{-2N}\Z} \bigg( \sum_{x:(t,x)\in\Lambda_N} 2^{-Nd} \Big( \!\!\!\sum_{\substack{(s,y)\in\Lambda_n\\\|(t-s,x-y)\|_{\s} \leq C 2^{-N}}}\!\!\!2^{-(n-N)|\s|} \frac{|\delta A^n_{s,y}|}{2^{-n\frac{|\s|}{2}-N\alpha}}\Big)^p \bigg)^{\frac{1}{p}} \\
&\lesssim \sup_{t\in 2^{-2N}\Z}\sum_{\substack{s\in 2^{-2n}\Z\\ |t-s| \leq C^2 2^{-2N}}}2^{-2(n-N)} \bigg( \sum_{x:(t,x)\in\Lambda_N}\sum_{\substack{y:(s,y)\in\Lambda_n\\ |x-y| \leq C 2^{-N}}}\!\!\!2^{-nd} \Big| \frac{\delta A^n_{s,y}}{2^{-n\frac{|\s|}{2}-N\alpha}}\Big|^p \bigg)^{\frac{1}{p}} \\
&\lesssim \sup_{s\in 2^{-2n}\Z} \Big( \sum_{y:(s,y)\in\Lambda_n}2^{-nd} \Big|\frac{\delta A^n_{s,y}}{2^{-n\frac{|\s|}{2}-n\gamma}}\Big|^p\Big)^{\frac{1}{p}} 2^{-n\gamma}\;,
\end{equs}
uniformly over all $n>N\geq 0$. The calculation for $\cP_0 g_n$ is very similar. Consequently, $\| \sum_{n=m}^{\infty} g_n \|_{\alpha,p} \lesssim \|A\| 2^{-m\gamma}$ and the asserted convergence is proved. Moreover, the bounds (\ref{Eq:RateReconstr}) follow immediately by keeping track of constants.
\end{proof}
We now proceed to the proof of the reconstruction theorem. Even though the general method of proof is quite similar to that of Theorem 3.10 in~\cite{Hairer2014}, a specific work is needed here in order to get the refined bound (\ref{Eq:BoundReconstructionp}).
\begin{proof}[of Theorem \ref{Th:Reconstruction}]
Set
\begin{equ}
M = (\mbox{diam supp }\varphi) \vee (\max\{ \mbox{diam supp } \psi: \psi\in\Psi\}) \vee (\max\{|k|: a_k \ne 0\})\;.
\end{equ}
Let us introduce the following notation: for all $t\in\R$, we let $t^{\downarrow n} := t - C 2^{-2n}$ where $C = 7 M^2 +1$. Recall the notation $x_{n,k}$ and $t_{n,k}$ introduced above (\ref{Eq:Waveletak}). For all $n\geq 0$, we define
\begin{equ}
\cR_n f := \sum_{(t,x)\in\Lambda_n} A^n_{t,x} \varphi^n_{t,x} \;,
\end{equ}
where, for all $(t,x)\in \R^{d+1}$
\begin{equ}
A^n_{t,x} := \int_{y\in B(x,2^{-n})} 2^{nd} \langle \Pi_{t^{\downarrow n},y} f(t^{\downarrow n},y) , \varphi^n_{t,x} \rangle dy \;,
\end{equ}
with $\scal{\cdot,\cdot}$ denoting the pairing between distributions and test functions.
One can write
\begin{equs}
\delta A^n_{t,x} = \sum_{k\in\Z^{d+1}} a_k \Big( &\int_{v\in B(x_{n,k},2^{-(n+1)})} 2^{(n+1)d} \big\langle \Pi_{t^{\downarrow n+1}_{n,k} , v} f(t^{\downarrow n+1}_{n,k} , v) , \varphi^{n+1}_{t_{n,k},x_{n,k}} \big\rangle dv\\
&- \int_{u\in B(x,2^{-n})} 2^{nd} \big\langle \Pi_{t^{\downarrow n},u} f(t^{\downarrow n},u) , \varphi^{n+1}_{t_{n,k},x_{n,k}} \big\rangle du\Big)\;.
\end{equs}
Observe that any two points $v$ and $u$ appearing in the integral above are at distance at most $(M+3)2^{-(n+1)}$
from each other. A simple calculation thus shows that
\begin{equ}\label{Eq:BounddeltaAn}
|\delta A^n_{t,x}| \lesssim \sum_{\substack{k\in\Z^{d+1}\\a_k\ne 0}}\sum_{\zeta\in\cA} \int_{u \in B(x,2^{-n})} 2^{n(d-\zeta-\frac{|\s|}{2})} F_{\zeta}^n(t^{\downarrow n},t^{\downarrow n+1}_{n,k},u)\,du\;,
\end{equ}
where the quantity $F_{\zeta}^n$ is given by
\begin{equs}
F_{\zeta}^n(t,s,u) &= \| \Pi \|_{su}|f(s,u)-\Gamma_{s,t}^u f(t,u)|_\zeta\\
&+ \int_{v\in B(u,(M+3) 2^{-(n+1)})} 2^{nd} \| \Pi \|_{sv}|f(s,v)-\Gamma_{v,u}^{s} f(s,u)|_\zeta dv\;.
\end{equs}
At this stage, it is simple to check that the conditions of Proposition \ref{Prop:Reconstruction} are satisfied, so that $\cR$ can be defined as the limit of $\cR_n$ as $n\rightarrow\infty$.

Let us now establish (\ref{Eq:BoundReconstructionp}). For every $\lambda \in (0,1]$, we let $n_0$ be the smallest integer such that $2^{-n_0} \leq \lambda$. Then, we define $n_1$ as the smallest integer such that
\begin{equ}\label{Eq:Defn0n1}
2^{-n_0} \geq 6M 2^{-n_1}\;,\quad 2^{-2n_0} \geq (7M^2+C) 2^{-2n_1}\;.
\end{equ}
Then, we write
\begin{equs}
\cR f - \Pi_{t,x}f(t,x) &= \big(\cR_{n_1} f - \cP_{n_1} \Pi_{t,x} f(t,x) \big) \label{Eq:SumRn}\\
&+ \sum_{n\geq n_1} \cR_{n+1} f - \cR_n f - \big( \cP_{n+1} - \cP_n \big) \Pi_{t,x} f(t,x)\;,
\end{equs}
where $\cP_n$ is the orthogonal projection onto $V_n$. We bound the terms on the right hand side separately. To that end, we introduce the set
\begin{equ}
\Lambda_{n}^{t,x,\lambda} := \big\{ (s,y) \in \Lambda_n: |t-s| \leq \lambda^2 + 7M^2 2^{-2n}, |x-y| \leq \lambda + 5M 2^{-n}\big\}\;.
\end{equ}
We claim that
\begin{equs}
\bigg\| \sum_{(s,y)\in \Lambda_{n}^{t,x,\lambda}} &\Big|A^n_{s,y} - \langle \Pi_{t,x} f(t,x) , \varphi^n_{s,y} \rangle\Big| \bigg\|_{L^p_{x_0,1}} \label{Eq:BoundDegueu}\\
&\lesssim \| \Pi \|_{B^n_{\lambda,t,x_0}}(1+\| \Gamma \|_{B^n_{\lambda,t,x_0}}) \|f \|_{B^n_{\lambda,t,x_0}} \sum_{\zeta\in\cA} \lambda^{|\s|+\gamma-\zeta} 2^{-n(\zeta-\frac{|\s|}{2})}\;,
\end{equs}
holds uniformly over all $(t,x_0)\in\R^{d+1}$, all $\lambda\in (0,1]$ and all $n\geq n_1$. We postpone the proof of (\ref{Eq:BoundDegueu}), and proceed to bounding the terms appearing in (\ref{Eq:SumRn}). The first term on the right hand side of (\ref{Eq:SumRn}) yields the following contribution:
\begin{equs}
\langle \cR_{n_1} f - \cP_{n_1} \Pi_{t,x} f(t,x) , \eta^\lambda_{t,x} \rangle = \sum_{(s,y)\in \Lambda_{n_1}}\big( A^{n_1}_{s,y} - \langle \Pi_{t,x} f(t,x) , \varphi^{n_1}_{s,y} \rangle \big) \langle \varphi^{n_1}_{s,y} , \eta^\lambda_{t,x} \rangle\;.
\end{equs}
We have $|\langle \varphi^{n_1}_{s,y} , \eta^\lambda_{t,x} \rangle| \lesssim 2^{-n_1\frac{|\s|}{2}}\lambda^{-|\s|}$ uniformly over all the parameters, and the inner product vanishes as soon as $(s,y) \notin \Lambda_{n_1}^{t,x,\lambda}$. Therefore, using (\ref{Eq:BoundDegueu}) we obtain that 
\begin{equs}
\bigg\| \sup_{\eta\in\cB^r}\Big|&\langle \cR_{n_1} f - \cP_{n_1} \Pi_{t,x} f(t,x) , \eta^\lambda_{t,x} \rangle\Big| \bigg\|_{L^p_{x_0,1}}\\
&\lesssim \| \Pi \|_{B^{n_1}_{\lambda,t,x_0}}(1+\| \Gamma \|_{B^{n_1}_{\lambda,t,x_0}}) \|f \|_{B^{n_1}_{\lambda,t,x_0}} \lambda^\gamma\;,
\end{equs}
as required. We turn to the second term on the right hand side of (\ref{Eq:SumRn}). As before, we write
\begin{equs}
\cR_{n+1} f - \cR_n f = \delta_n f + g_n \;,
\end{equs}
with $\delta_n f\in V_n^\perp$ and $g_n \in V_n$. We then have
\begin{equs}
\langle &\delta_n f - \big( \cP_{n+1} - \cP_n \big) \Pi_{t,x} f(t,x) , \eta^\lambda_{t,x} \rangle\\
&= \sum_{(s,y)\in\Lambda_{n+1}} \sum_{(r,u)\in\Lambda_n}  \Big( A^{n+1}_{s,y} - \langle \Pi_{t,x} f(t,x) , \varphi^{n+1}_{s,y} \rangle \Big) \langle \varphi^{n+1}_{s,y} , \psi^n_{r,u} \rangle \langle \psi^n_{r,u} , \eta^\lambda_{t,x} \rangle\;.
\end{equs}
Observe that $|\langle \varphi^{n+1}_{s,y} , \psi^n_{r,u} \rangle| \lesssim 1$ and 
$|\langle \psi^n_{r,u} , \eta^\lambda_{t,x} \rangle| \lesssim 2^{-n(r+\frac{|\s|}{2})} \lambda^{-(r+|\s|)}$, 
uniformly over all the parameters. For every given $(s,y)$, the first inner product vanishes except 
for those finitely many space-time coordinates $(r,u) \in \Lambda_n$ such that $|r-s|\leq 5 M^2 2^{-2(n+1)}$ 
and $|u-y|\leq 3 M 2^{-(n+1)}$. Furthermore, the second inner product vanishes whenever 
$|r-t| > \lambda^2+M^2 2^{-2n}$ or $|u-x| > \lambda + M 2^{-n}$. Therefore, we have
\begin{equs}
{}&|\langle \delta_n f - \big( \cP_{n+1} - \cP_n \big) \Pi_{t,x} f(t,x) , \eta^\lambda_{t,x} \rangle|\\
&\lesssim \sum_{(s,y)\in\Lambda_{n+1}^{t,x,\lambda}} \Big| A^{n+1}_{s,y} - \langle \Pi_{t,x} f(t,x) , \varphi^{n+1}_{s,y} \rangle \Big| 2^{-n(r+\frac{|\s|}{2})} \lambda^{-(r+|\s|)}\;,
\end{equs}
uniformly over all the parameters. Using (\ref{Eq:BoundDegueu}), it is then easy to get
\begin{equs}
\bigg\| \sup_{\eta\in\cB^r}\Big|&\langle \delta_n f - \big( \cP_{n+1} - \cP_n \big) \Pi_{t,x} f(t,x) , \eta^\lambda_{t,x} \rangle\Big| \bigg\|_{L^p_{x_0,1}}\\
&\lesssim \| \Pi \|_{B^{n+1}_{\lambda,t,x_0}}(1+\| \Gamma \|_{B^{n+1}_{\lambda,t,x_0}}) \|f \|_{B^{n+1}_{\lambda,t,x_0}} \Big(\frac{2^{-(n+1)}}{\lambda}\Big)^{r+\alpha}\lambda^\gamma\;,
\end{equs}
as required. Finally, we treat the contribution of $g_n=\sum_{(s,y)\in\Lambda_n} \delta A_{s,y}^n \varphi_{s,y}^n$:
\begin{equs}
{}&\bigg\| \sup_{\eta\in\cB^r}|\langle g_n , \eta^\lambda_{t,x} \rangle|\bigg\|_{L^p_{x_0,1}} \lesssim \bigg\|\sum_{\substack{(s,y)\in\Lambda_{n}:|s-t|\leq \lambda^2 + M^2 2^{-2n}\\ |y-x|\leq \lambda+M 2^{-n}}} \big|\delta A^n_{s,y}\big| 2^{-n\frac{|\s|}{2}}\lambda^{-|\s|} \bigg\|_{L^p_{x_0,1}}\;.
\end{equs}
For all $s$ in the sum above and for all $k\in\Z^{d+1}$ such that $a_k\ne 0$, $s_{n,k}^{\downarrow n+1}$ belongs to $[t-\lambda^2-(5M^2+C)2^{-2(n+1)} , t+\lambda^2+(5M^2-C)2^{-2(n+1)}]$, which is a subset of $[t-2\lambda^2,t+\lambda^2-2^{-(n+1)}]$ thanks to (\ref{Eq:Defn0n1}) and the definition of $C$. By (\ref{Eq:BounddeltaAn}), a simple calculation using Jensen's inequality yields
\begin{equs}
{}\bigg\| \sup_{\eta\in\cB^r}\Big|\langle g_n , \eta^\lambda_{t,x} \rangle\Big| \bigg\|_{L^p_{x_0,1}}\lesssim \| \Pi \|_{B^n_{\lambda,t,x_0}} \|f \|_{B^n_{\lambda,t,x_0}}2^{-n\gamma}\;,
\end{equs}
so that the asserted bound follows.

We are now left with the proof of (\ref{Eq:BoundDegueu}). We split $A^n_{s,y} - \langle \Pi_{t,x} f(t,x), \varphi^n_{s,y} \rangle$ into the sum of
\begin{equ}
I_n(t,x,s,y) = \int_{u\in B(y,2^{-n})} 2^{nd} \langle \Pi_{s^{\downarrow n},u} \big(f(s^{\downarrow n},u) - \Gamma_{u,x}^{s^{\downarrow n}} f(s^{\downarrow n},x)\big) , \varphi^n_{s,y} \rangle du\;,
\end{equ}
and
\begin{equ}
J_n(t,x,s,y) = \langle \Pi_{s^{\downarrow n},y}\Gamma_{y,x}^{s^{\downarrow n}} \big(f(s^{\downarrow n},x) - \Gamma_{s^{\downarrow n},t}^x f(t,x)\big) , \varphi^n_{s,y} \rangle\;.
\end{equ}
We start with $|I_n(t,x,s,y)|$, which can be bounded by
\begin{equ}
\sum_{\zeta\in\cA}\int_{u\in B(y,2^{-n})} 2^{n(d-\zeta-\frac{|\s|}{2})} \|\Pi\|_{s^{\downarrow n},u} \big|f(s^{\downarrow n},u) - \Gamma_{u,x}^{s^{\downarrow n}} f(s^{\downarrow n},x)\big|_\zeta du\;.
\end{equ}
For all $(s,y)\in\Lambda_n^{t,x,\lambda}$, we have $|y-x|\leq \lambda+5M 2^{-n}$ so that using (\ref{Eq:Defn0n1}), we can bound the integral over all $u\in B(y,2^{-n})$ by the same integral over all $u\in B(x,2\lambda)$. This yields
\begin{equs}
{}&\bigg\| \sum_{(s,y)\in \Lambda_{n}^{t,x,\lambda}} |I_n(t,x,s,y)|\bigg\|_{L^p_{x_0,1}}\lesssim \sum_{\substack{s\in 2^{-2n}\Z\\|s-t| \leq \lambda^2 +7M^2 2^{-2n}}} \bigg\| \sum_{\zeta\in\cA}\int_{u\in B(x,2\lambda)} 2^{n(d-\zeta-\frac{|\s|}{2})}\\
&\qquad\times\|\Pi\|_{s^{\downarrow n},u}\big|f(s^{\downarrow n},u) - \Gamma_{u,x}^{s^{\downarrow n}} f(s^{\downarrow n},x)\big|_\zeta du\bigg\|_{L^p_{x_0,1}}\\
&\lesssim \| \Pi \|_{B^n_{\lambda,t,x_0}} \|f \|_{B^n_{\lambda,t,x_0}} \sum_{\zeta\in\cA} \lambda^{|\s|+\gamma-\zeta} 2^{-n(\zeta-\frac{|\s|}{2})}\;,
\end{equs}
as required. Notice that we have used the fact that the sum over $s$ at the second line contains at most $(\lambda 2^{n})^2$ elements, and that for all these $s$, we have $s^{\downarrow n} \in [t-2\lambda^2,t+\lambda-2^{-2n}]$ thanks to (\ref{Eq:Defn0n1}) and the definition of $C$.

To bound $|J_n(t,x,s,y)|$, we distinguish two cases. If $s^{\downarrow n} > t$, then it can be bounded by
\begin{equs}
{}&\lesssim \sum_{\substack{\zeta,\beta \in\cA\\\zeta\geq \beta}} \|\Pi\|_{s^{\downarrow n},y}\|\Gamma\|_{s^{\downarrow n}y,s^{\downarrow n}x} |x-y|^{\zeta-\beta} \big|f(s^{\downarrow n},x) - \Gamma_{s^{\downarrow n},t}^x f(t,x)\big|_\zeta \,2^{-n(\beta + \frac{|\s|}{2})}\\
&\lesssim \sum_{\zeta\geq\beta} \|\Pi\|_{s^{\downarrow n},y} \|\Gamma\|_{s^{\downarrow n}y,s^{\downarrow n}x}\frac{\big|f(s^{\downarrow n},x) - \Gamma_{s^{\downarrow n},t}^x f(t,x)\big|_\zeta}{\lambda^{\gamma-\zeta}} \,\lambda^{\gamma-\beta} 2^{-n(\beta + \frac{|\s|}{2})}\;.
\end{equs}
On the other hand, if $s^{\downarrow n} < t$, then we write 
\begin{equ}
J_n(t,x,s,y) = -\langle \Pi_{s^{\downarrow n},y}\Gamma_{s^{\downarrow n}y,tx} \big(f(t,x) - \Gamma_{t,s^{\downarrow n}}^x f(s^{\downarrow n},x)\big) , \varphi^n_{s,y} \rangle\;,
\end{equ}
and, for all $(s,y)\in\Lambda_n^{t,x,\lambda}$, we bound $|J_n(t,x,s,y)|$ by
\begin{equs}
{}&\lesssim \sum_{\substack{\zeta,\beta \in\cA\\\zeta\geq \beta}} \|\Pi\|_{s^{\downarrow n},y}\|\Gamma\|_{s^{\downarrow n}y,tx} \lambda^{\zeta-\beta} \big|f(t,x) - \Gamma_{t,s^{\downarrow n}}^x f(s^{\downarrow n},x)\big|_\zeta \,2^{-n(\beta + \frac{|\s|}{2})}\\
&\lesssim \sum_{\zeta\geq\beta} \|\Pi\|_{s^{\downarrow n},y}\|\Gamma\|_{s^{\downarrow n}y,tx} \frac{\big|f(t,x) - \Gamma_{t,s^{\downarrow n}}^x f(s^{\downarrow n},x)\big|_\zeta}{\lambda^{\gamma-\zeta}} \,\lambda^{\gamma-\beta} 2^{-n(\beta + \frac{|\s|}{2})}\;.
\end{equs}
In both cases, we deduce that
\begin{equs}
{}&\bigg\| \sum_{(s,y)\in \Lambda_{n}^{t,x,\lambda}} |J_n(t,x,s,y)|\bigg\|_{L^p_{x_0,1}}\\
&\lesssim \| \Pi \|_{B^n_{\lambda,t,x_0}}\| \Gamma \|_{B^n_{\lambda,t,x_0}} \|f \|_{B^n_{\lambda,t,x_0}} \sum_{\zeta\in\cA} \lambda^{|\s|+\gamma-\zeta} 2^{-n(\zeta-\frac{|\s|}{2})}\;,
\end{equs}
This ends the proof.

The uniqueness of the reconstruction follows from the same argument as in~\cite{Hairer2014}, but for completeness, we recall it briefly. Assume that $\xi_1$ and $\xi_2$ are two candidates for $\cR f$ that both satisfy (\ref{Eq:BoundReconstructionp}). Let $\psi$ be a compactly supported, smooth function on $\R^{d+1}$ and let $\eta\in\cB^r$ be even and integrating to $1$. We set
\[ \psi_\lambda(s,y) = \langle \eta^\lambda_{s,y} , \psi\rangle = \int \psi(t,x) \eta^\lambda_{t,x}(s,y) dt\,dx \;.\]
Then, we have
\[ \langle \xi_1-\xi_2 , \psi_\lambda \rangle = \int \psi(t,x) \langle \xi_1-\xi_2 , \eta^\lambda_{t,x}\rangle dt\,dx \;.\]
We obtain
\begin{equ}
|\langle \xi_1-\xi_2 , \psi_\lambda \rangle| \lesssim \|\psi\|_\infty \sup_t \Big\| \langle \xi_1-\xi_2 , \eta^\lambda_{t,x}\rangle \Big\|_{L^p(dx)}\lesssim \|\psi\|_\infty \lambda^\gamma \;,
\end{equ}
so that $\langle \xi_1-\xi_2 , \psi_\lambda \rangle$ goes to $0$ as $\lambda\downarrow 0$. Since $\psi_\lambda$ converges to $\psi$ in the $\cC^\infty$ topology, one has $\langle \xi_1-\xi_2 , \psi_\lambda \rangle \rightarrow \langle \xi_1-\xi_2 , \psi \rangle$. Hence $\xi_1=\xi_2$ and the uniqueness follows.

To complete the proof of the theorem, it remains to consider the case of two models $(\Pi,\Gamma)$ and $(\bar\Pi,\bar\Gamma)$. The reconstruction theorem applies to both $f$ and $\bar f$ separately, using the sequences $\cR_n f$ and $\bar{\cR}_n\bar f$ associated to each of them. Then, we observe that $|\delta A^n_{t,x} - \delta \bar{A}^n_{t,x}|$ satisfies the bound (\ref{Eq:BounddeltaAn}) with $F_{\zeta}^n(t,s,u)$ replaced by
\begin{equs}
{}&\tilde{F}_{\zeta}^n(t,s,u) = \| \Pi \|_{su}|f(s,u)-\bar f(s,u)-\Gamma_{s,t}^u f(t,u)+\bar\Gamma_{s,t}^u \bar f(t,u)|_\zeta\\
&+ \int_{B(u,(M+3) 2^{-(n+1)})}\!\!\! 2^{nd} \| \Pi \|_{sv}|f(s,v)-\bar f(s,v)-\Gamma_{v,u}^{s} f(s,u)+\bar\Gamma_{v,u}^{s} \bar f(s,u)|_\zeta dv \\
&+ \| \Pi-\bar\Pi \|_{su}|\bar f(s,u)-\bar \Gamma_{s,t}^u \bar f(t,u)|_\zeta\\
&+ \int_{B(u,(M+3) 2^{-(n+1)})}\!\!\! 2^{nd} \| \Pi-\bar\Pi \|_{sv}|\bar f(s,v)-\bar\Gamma_{v,u}^{s} \bar f(s,u)|_\zeta dv\;.
\end{equs}
Furthermore, in this context, (\ref{Eq:BoundDegueu}) becomes
\begin{equs}[Eq:BoundDegueuTwoModels]
\bigg\| \sum_{(s,y)\in \Lambda_{n}^{t,x,\lambda}} &\big|A^n_{s,y}-\bar{A}^n_{s,y} - \langle \Pi_{t,x} f(t,x)- \bar\Pi_{t,x} \bar{f}(t,x) , \varphi^n_{s,y} \rangle\big| \bigg\|_{L^p_{x_0,1}}\\
&\lesssim K^{n}_{t,x_0,\lambda} \sum_{\zeta\in\cA} \lambda^{|\s|+\gamma-\zeta} 2^{-n(\zeta-\frac{|\s|}{2})}\;,
\end{equs}
where $K^n_{t,x_0,\lambda}$ is given by (\ref{Eq:BoundNormModels}). The proof of (\ref{Eq:BoundDegueuTwoModels}) follows from the same arguments as above \textit{mutatis mutandis}. This being given, the proof of (\ref{Eq:BoundReconstructionTwo}) follows from exactly the same arguments as above.
\end{proof}

\section{Weighted spaces}\label{sec:Weights}

We would like to deal with white noise as the elementary input in our regularity 
structure, but unfortunately white noise does not live in any of the spaces $\cC^{\alpha}$ due to the unboundedness of the underlying space. In order to circumvent this problem, we choose to consider weighted versions of the previously mentioned spaces.
We first define the class of functions that have good enough properties to be used as weights.
\begin{definition}\label{Def:Weights}
A function $w:\R^d\rightarrow\R_+$ is a weight if there exists $C>0$ such that for all $x,y\in\R^d$ with $|x-y| \leq 1$
\[ \frac{1}{C} \leq \frac{w(x)}{w(y)} \leq C \;.\]
\end{definition}
All the weights considered in this article are built from two elementary families:
\[ \p_a (x) := (1+|x|)^a \;,\quad \e_\ell(x):= e^{\ell(1+|x|)} \;,\]
with $a,\ell\in \R$. It is easy to verify that these are indeed weights. 
We also observe that the constant $C$ can be taken uniformly over all $a$ and $\ell$ in compact sets of $\R$.
Given a weight $w$, we let $\cC^\alpha_w(\R^{d+1})$ be the set of distributions $f$ on $\R^{d+1}$ such that
\begin{equ}
\sup_{\lambda \in (0,1]} \sup_{(t,x)\in\R^{d+1}} \sup_{\eta \in \cB^r(\R^{d+1})} \frac{|\langle f , \eta^\lambda_{t,x} \rangle|}{w(x) \, \lambda^\alpha}  < \infty\;.
\end{equ}

\begin{remark}
Our setting may seem surprising since our weights are in space and not in space-time; the reason for this 
choice is twofold. First, the solution map for the SPDEs only needs to be defined on (arbitrary) 
bounded intervals of time, so that it suffices to characterise the regularity of the white noise on 
$(0,T] \times \R^d$: therefore, only the unboundedness of the space variable matters. Second, 
and this is more serious, we aim at 
using the exponential weights $\e_{\ell+t}$ for the solution, and it happens that they are \textit{not} 
space-time weights since $e^{t(1+|x|)}/e^{s(1+|y|)}$ is not uniformly bounded from above and below, when 
$(t,x)$ and $(s,y)$ are only constrained to be at distance at most $1$ from one another.
\end{remark}

We now characterise the regularity of white noise. Let us start with the case of the space-time white noise, which is the driving noise for (SHE). Let $\chi_{T}:\R\rightarrow\R$ be a compactly supported smooth function, which is equal to $1$ on $(-2T,2T)$, and let $\xi$ be a white noise on $\R^{d+1}$. Let $\rho :\R^{d+1} \to \R$ be a compactly supported, even, smooth function that integrates to one. We set $\rho_\epsilon(t,x) = \epsilon^{-|\s|}\rho(t\epsilon^{-2},x\epsilon^{-1})$, and we define the mollified noise $\xi_\epsilon = \rho_\epsilon * \xi$.

\begin{lemma}\label{Lemma:RegXi}
Fix $a >0$, set $\wPi(x):= (1+|x|)^a$, $x\in\R^d$, and let $\alpha < -|\s|/2$. Then, for any arbitrary $T>0$, $\xi\cdot\chi_T$ admits a modification that belongs almost surely to $\cC^{\alpha}_{\wPi}$, and there exists $\delta > 0$ such that
\begin{equation*}
\bbE \|\xi_\epsilon\cdot\chi_T - \xi\cdot\chi_T\|_{\alpha,\wPi} \lesssim \epsilon^{\delta}\;,
\end{equation*}
uniformly over all $\epsilon\in(0,1]$.
\end{lemma}
\noindent Observe that $a$ can be taken as small as desired. In the case of (PAM), the white noise 
is only in space and an immediate adaptation of the proof shows that it admits a modification in 
$\cC^{\alpha}_{\wPi}$ for any $\alpha < -d/2$.
\begin{proof}
From Proposition \ref{Prop:CharactDistrib}, it suffices to show that almost surely
\begin{equation*}
\sup_{n\geq 0}\sup_{\psi \in \Psi}\sup_{(t,x)\in\Lambda_n}\frac{|\left\langle\xi\cdot\chi_T,\psi_{t,x}^{n}\right\rangle|}{\wPi(x)2^{-n\frac{|\s|}{2}-n\alpha}} < \infty  \;,\quad \sup_{(t,x)\in\Lambda_0}\frac{|\left\langle\xi\cdot\chi_T, \varphi_{t,x}\right\rangle|}{\wPi(x)} < \infty \;.
\end{equation*}
We only treat the first bound, since the second is similar. For any $p\geq 1$, we write
\begin{equs}
\bbE\bigg[ \sup_{n\geq 0}\sup_{\psi \in \Psi}&\sup_{(t,x)\in\Lambda_n}\bigg(\frac{|\left\langle\xi\cdot\chi_T,\psi_{t,x}^{n}\right\rangle|}{\wPi(x)2^{-n\frac{|\s|}{2}-n\alpha}} \bigg)^{2p}\bigg]\\
&\lesssim \sum_{n\geq 0}\sum_{\psi \in \Psi}\sum_{(t,x)\in\Lambda_n}\bigg(\frac{\bbE \left\langle\xi\cdot\chi_T,\psi_{t,x}^n\right\rangle^2}{\wPi(x)^{2}2^{-|\s|n-2n\alpha}}\bigg)^p
\end{equs}
where we have used the equivalence of moments of Gaussian random variables. Recall that the $L^2$ norm of $\psi_{t,x}^{n}$ is $1$, that the cardinality of the restriction of $\Lambda_n$ to the unit ($\s$-scaled parabolic) ball of $\R^{d+1}$ is of order $2^{|\s|n}$, and that $\Psi$ is a finite set. Recall also that $\chi_T$ is compactly supported. Thus we obtain that the last term is of order
\begin{equs}
\sum_{x\in\Z^{d}} \wPi(x)^{-2p}\sum_{n\geq 0} 2^{|\s|n(p+1)+2\alpha np} \;.
\end{equs}
Taking $p$ large enough, the sums over $n$ and $x$ converge. This shows that $\xi\cdot\chi_T$ admits a modification that almost surely belongs to $\cC^{\alpha}_{\wPi}$. We turn to $\|(\xi_\epsilon- \xi)\chi_T\|_{\alpha,\wPi}$. The computation is very similar, the only difference rests on the term
\begin{equs}
\bbE \left\langle(\xi-\xi_\epsilon)\chi_T,\psi_{t,x}^{n}\right\rangle^2 &= \|\psi_{t,x}^n\chi_T-\rho_\epsilon*\big(\psi_{t,x}^n\chi_T\big)\|^2_{L^2} \;.
\end{equs}
When $t\notin (-2T-\epsilon,2T+\epsilon)$, this term vanishes. Otherwise, it can be bounded by a term of order $1\wedge (\epsilon^2 2^{2n})$ uniformly over all $\epsilon\in(0,1]$, all $n\geq 0$ and all $(t,x)\in\R^{d+1}$. We obtain
\begin{equs}
\bbE \bigg[\sup_{n\geq 0}\sup_{\psi \in \Psi}&\sup_{x\in\Lambda_n}\Big( \frac{|\left\langle(\xi-\xi_\epsilon)\chi_T,\psi_x^{n}\right\rangle|}{\wPi(x)2^{-n\frac{|\s|}{2}-n\alpha}} \Big)^{2p}\bigg]\\
&\lesssim \sum_{x\in\Z^d}\sum_{n\geq 0}\frac{2^{n(|\s|+2p\alpha+|\s|p)}(1 \wedge \epsilon^{2p} 2^{2np})}{\wPi(x)^{2p}} \;,
\end{equs}
so that for $\alpha<-|\s|/2$ and $p$ large enough, the previous calculation yields the bound $\bbE \|\xi_\epsilon - \xi\|_{\alpha,p,\wPi} \lesssim (\epsilon|\log \epsilon|^{\frac{1}{2p}}) \vee \epsilon^{-\alpha-\frac{|\s|}{2}(1+\frac{1}{p})}$ uniformly over all $\epsilon \in (0,1]$.
\end{proof}

Given a weight $\wPi$ on $\R^d$, we define weighted versions of the seminorm on the model. For any subset $B\subset \R^{d+1}$, we set
\begin{equs}
\$ \Pi\$_B := \sup_{z\in B} \frac{\left\| \Pi \right\|_z}{\wPi(x)}\;,\qquad\$ \Gamma\$_B := \sup_{\substack{z,z'\in B\\\|z-z'\|_{\s}\leq 1}} \frac{\left\| \Gamma \right\|_{z,z'}}{\wPi(x)}\;,
\end{equs}
where $x$ is the space component of $z$ in the above expressions. We are now in a position to introduce the natural model associated to the mollified noise.

\begin{lemma}\label{lem:normModel}
Set $\wPi(x) = (1+|x|)^a$ for a given $a > 0$. Then, for any set $B$ of the form $[0,T]\times\R^d$ the seminorms $\$ \Pi^{(\eps)} \$_B$ and $\$ \Gamma^{(\eps)} \$_B$ are almost surely finite.
\end{lemma}

These bounds are \textit{not} uniform over $\eps \in (0,1]$: this is the reason why a renormalisation at the level of the model is required.

\begin{proof}
Let $B = [0,T]\times\R^d$ for a given $T > 0$. First, we observe that the required bound on $\Pi^{(\eps)}_z$ holds for polynomials, and also for $\Xi$ by Lemma \ref{Lemma:RegXi} since $\langle \xi_\epsilon , \eta_z\rangle$ coincides with $\langle \xi_\epsilon \cdot\chi_T, \eta_z \rangle$ for all test functions $\eta \in \cB^r(\R^{d+1})$ and all $z\in B$. Then, the key observation is that all the elements in the regularity structure are built from polynomials and $\Xi$ by multiplication and/or application of $\cI$. Additionally, for every $\|z-z'\|_{\s} \leq 1$, the definitions of $\Pi^{(\eps)}_{z} \cI \tau(z')$ and $\Pi^{(\eps)}_{z}\tau\bar\tau(z')$ only involve the values of $\Pi^{(\eps)}_{z} \tau(\cdot)$ and $\Pi^{(\eps)}_{z} \bar\tau(\cdot)$ in a neighourhood of $z$, so that, for bounding these terms, the definition of a weight allows one to disregard the precise location at which the evaluation is taken. Since the regularity structure has finitely many elements, a simple recursion shows that the analytical bound on $\Pi^{(\eps)}_z$ holds with the weight $\wPi(x)^n$ for some $n\geq 1$, instead of $\wPi(x)$. Given the expression of $\wPi(x)$, it suffices to decrease $a$ accordingly in order to get the required statement. Regarding the analytical bound on $\Gamma^{(\eps)}_{z,z'}$, the proof follows from very similar arguments, using the proof of \cite[Prop~8.27]{Hairer2014} and the bound of \cite[Lemma 5.21]{Hairer2014}.
\end{proof}
\begin{notation}
From now on, the seminorm on the model will always be taken with the set $B=[0,T]\times\R^d$ and the maximal $T$ will always be clear from the context. Therefore, we will omit the subscript $B$ on this seminorm for simplicity.
\end{notation}

Let us now introduce weighted spaces of modelled distributions. For similar reasons as for the model, 
we add weights at infinity in the spaces $\cD^{\gamma,p}$. Moreover, to allow for an irregular initial 
condition, we also weigh these spaces near time $0$. For every $\zeta\in \cA$ and $t\in\R$, we consider 
two collections of weights on $\R^d$, $\wun_t(\cdot,\zeta)$ and $\wde_t(\cdot,\zeta)$. We set
\begin{equ}[e:defww]
\w_t(x) := \inf_{\zeta\in\cA} \inf_{i\in\{1,2\}} \wi_t(x,\zeta) \;,
\end{equ}
and make the following assumption.
\begin{assumption}[Weights and reconstruction]\label{AssumptionWRecons}
All the weights $\wi_t(x,\zeta)$ are increasing functions of time. Furthermore, there exists $c > 0$ such that, for any time $T > 0$, there exists $K > 0$ such that
\begin{equs}
K^{-1} \leq \sup_{x,y\in\R^d:|x-y|\leq 1}\frac{\wi_t(x,\zeta)}{\wi_t(y,\zeta)} &\leq K \label{CondP0}\tag{W-0}\;,\\
\sup_{x\in\R^d} \frac{(\wPi(x))^2\wi_s(x,\zeta)}{\w_t(x)} &\leq K (t-s)^{-\frac{c}{2}} \label{CondP1}\tag{W-1}\;,
\end{equs}
uniformly over all $s < t \in (-\infty,T]$, all $i\in\{1,2\}$ and all $\zeta\in\cA$.
\end{assumption}

From now on, we take $L^p=L^p(\R^d,dx)$ and, by convention, the integration variable is always $x$,
so that for example $\|f(x,y)\|_{L^p}$ really means $\|f(\cdot,y)\|_{L^p}$.
\begin{definition}
Let $\eta,\gamma \in \R$ and $p\in[1,\infty)$. We define $\ccD^{\gamma,\eta,p}_{T,\rw}$ as the set of maps $f:(0,T]\times\R^d\rightarrow \cT_{<\gamma}$ such that
\begin{equs}[Eq:NormccD]
\bigg\|\frac{|f_t(x)|_\zeta}{\wun_t(x,\zeta)}\bigg\|_{L^p} &\lesssim t^{\frac{(\eta-\zeta)\wedge 0}{2}}\;,\\
\bigg\|\int_{y\in B(x,\lambda)}\lambda^{-d}\frac{|f_t(y)-\Gamma_{y,x}^tf_t(x)|_\zeta}{\wde_t(x,\zeta)\, \lambda^{\gamma-\zeta}}dy\bigg\|_{L^p} &\lesssim t^{\frac{\eta-\gamma}{2}} \;,\\
\bigg\|\frac{|f(t,x)-\Gamma_{t,t-\lambda^2}^x f(t-\lambda^2,x)|_\zeta}{\wun_{t}(x,\zeta)\, \lambda^{\gamma-\zeta}}\bigg\|_{L^p} &\lesssim t^{\frac{\eta-\gamma}{2}}\;,
\end{equs}
uniformly over all $\lambda\in(0,2]$, all $t\in (2\lambda^2,T]$, and all $\zeta\in \cA$. If $f$ takes values in $\cT(\cU)$, resp. $\cT(\cF)$, we say that $f$ belongs to $\ccD^{\gamma,\eta,p}_{T,\rw}(\cU)$, resp. $\ccD^{\gamma,\eta,p}_{T,\rw}(\cF)$. Finally, we let $\$f\$$ denote the corresponding norm.
\end{definition}
As we did in the previous subsection, we need to be able to compare two modelled distributions $f$ and $\bar f$ associated to two different models $(\Pi,\Gamma)$ and $(\bar\Pi,\bar\Gamma)$. To that end, we define $\$f;\bar{f}\$$ as the supremum of
\begin{equs}
{}&\bigg\|\frac{|f(t,x)-\bar{f}(t,x)|_\zeta}{t^{\frac{(\eta-\zeta)\wedge 0}{2}}\,\wun_t(x,\zeta)}\bigg\|_{L^p(dx)}\\
&+ \bigg\|\int_{y\in B(x,\lambda)}\frac{|f(t,y)-\bar{f}(t,y)-\Gamma_{y,x}^tf(t,x)+\bar{\Gamma}_{y,x}\bar{f}(t,x)|_\zeta}{t^{\frac{\eta-\gamma}{2}}\,\wde_t(x,\zeta)\,\lambda^{\gamma-\zeta}}dy\bigg\|_{L^p(dx)}\\
&+ \bigg\|\frac{|f(t,x)-\bar{f}(t,x)-\Gamma_{t,t-\lambda^2}^x f(t-\lambda^2,x)+\bar{\Gamma}_{t,t-\lambda^2}\bar{f}(t-\lambda^2,x)|_\zeta}{t^{\frac{\eta-\gamma}{2}}\,\wun_{t}(x,\zeta)\,\lambda^{\gamma-\zeta}}\bigg\|_{L^p(dx)} \;,
\end{equs}
over all $\lambda\in(0,2]$, all $t\in (2\lambda^2,T]$ and all $\zeta\in \cA$.

Observe that the space $\ccD^{\gamma,\eta,p}_{T,\rw}$ is actually locally identical to $\cD^{\gamma,p}$ so that, for any test function $\eta^\lambda_{t,x}$ supported away from the negative times, we can use 
Theorem~\ref{Th:Reconstruction} and define a local reconstruction operator $\langle \tilde{\cR} f, \eta^\lambda_{t,x} \rangle$. The next theorem shows that there is a canonical distribution $\cR f$ that coincides with $\tilde{\cR} f$ everywhere. First, let us define a weighted version of the space $\cE^{\alpha,p}$.
\begin{definition}\label{Def:DistribSpace}
Let $\alpha < 0$, $p\in [1,\infty)$ and $T>0$. We let $\ccE^{\alpha,p}_{\rw,T}$ be the space of distributions $f$ on $(-\infty,T)\times\R^{d}$ such that
\begin{equ}\label{Eq:BoundDistribSpace}
\sup_{\lambda \in (0,1]} \sup_{t\in (-\infty,T-\lambda^2)} \bigg\| \sup_{\eta \in \cB^r(\R^{d+1})} \frac{|\langle f , \eta^\lambda_{t,x} \rangle|}{\lambda^\alpha \rw_{t+\lambda^2}(x)} \bigg\|_{L^p(dx)} < \infty\;,
\end{equ}
where the weights $\rw_t$ were defined in \eqref{e:defww}.
\end{definition}
We start with the following extension result.
\begin{proposition}\label{Prop:UniqDistr}
Let $\alpha\in (-2,0)$, $p\in [1,\infty]$ and $T>0$. Let $f$ be a distribution on the set of all $\eta\in\cC^r(\R^{d+1})$ whose support does not intersect the hyperplane $\{t=0\}$. Assume that $f$ satisfies the bound (\ref{Eq:BoundDistribSpace}) with the second supremum restricted to all $t\in (-\infty,T-\lambda^2) \backslash [-3\lambda^2,3\lambda^2]$. Then, $f$ can be uniquely extended into an element of $\ccE^{\alpha,p}_{\rw,T}$.
\end{proposition}
\begin{proof}
The proof is divided into three steps. First, we show uniqueness of the extension. Then, we build the extension but with a non-optimal weight. Finally, we show that the weight can actually be improved. From now on, we let $\chi:\R\rightarrow\R$ be a compactly supported, smooth function such that $\supp\chi \subset [5,\infty)$ and $\sum_{n\in\Z}\chi(2^{2n}s)=1$ for all $s\in(0,\infty)$. We also let $\tilde{\chi}:\R\rightarrow\R$ be a smooth function such that $\supp\tilde\chi \subset [-1,1]$ and $\sum_{k\in\Z}\tilde\chi(x-k)=1$ for all $x\in\R$.

\smallskip

\noindent\textit{Step 1: uniqueness.} Let  For every $n_0\geq 1$, we set $\nu_{n_0}(t) = \sum_{n\leq n_0}\big(\chi (2^{2n}t)+\chi (-2^{2n}t)\big)$. Observe that this function vanishes in $[-5\cdot 2^{-2n_0},5\cdot 2^{-2n_0}]$. We claim that for any $f\in\ccE^{\alpha,p}_{\rw,T}$ and $n_0$ large enough, we have
\begin{equ}\label{Eq:ClaimUniq}
\big| \langle f , \phi_{t,x} (1-\nu_{n_0}) \rangle \big| \lesssim 2^{-n_0(2+\alpha)} \rw_{T}(x)\;,
\end{equ}
uniformly over all $\phi\in\cB^r(\R^{d+1})$ and all $(t,x)\in\R^{d+1}$. Since $2+\alpha > 0$, this claim shows that the knowledge of $f$ away from the hyperplane $\{t=0\}$ is sufficient to characterise $f$. The uniqueness of the statement is then immediate. We now prove the claim. We use the following partition of unity:
\begin{equ}
\sum_{(s,y)\in\Lambda_{n_0}}\psi_{n_0,s,y}(z) = 1\;, \quad \psi_{n_0,s,y}(z)=\tilde{\chi}(2^{2n_0}(z_0-s))\prod_{i=1}^d \tilde{\chi}(2^{n_0}(z_i-y_i))\;.
\end{equ}
Since $(1-\nu_{n_0})$ is supported in some centred interval of length of order $2^{-2n_0}$, we deduce that there exists $C>0$ such that $\phi_{t,x}(1-\nu_{n_0})\psi_{n_0,s,y}$ is identically zero as soon as $|y-x| > C$ and $|s| > C 2^{-2n_0}$, uniformly over all $\phi\in\cB^r$, all $(t,x)\in\R^{d+1}$, all $n_0\geq 0$ and $(s,y)\in\Lambda_{n_0}$. Then, for any $\phi \in\cB^r(\R^{d+1})$ and any $(t,x)\in\R^{d+1}$, we have
\begin{equs}[Eq:tildePhi]
\langle f , \phi_{t,x} (1-\nu_{n_0}) \rangle =  \sum_{(s,y)\in\Lambda_{n_0}} \langle f,\phi_{t,x}(1-\nu_{n_0})\psi_{n_0,s,y}\rangle \;.
\end{equs}
Recall that $|\s| = 2+d$. For all $z\in B(y,2^{-n_0})$, the function $2^{n_0|\s|} \phi_{t,x}(1-\nu_{n_0})\psi_{n_0,s,y}$ can be written as $\eta_{s,z}^{2^{-n_0}}$, for some $\eta\in\cB^r$, up to some factor $C$, where $|C|$ is uniformly bounded over all $\phi\in\cB^r$, all $n_0\geq 0$, all $(s,y)\in\Lambda_{n_0}$ and all $z\in B(y,2^{-n_0})$. Using Jensen's inequality, we thus get
\begin{equs}
{}&\Big|\sum_{(s,y)\in\Lambda_{n_0}} \langle f,\phi_{t,x}(1-\nu_{n_0})\psi_{n_0,s,y}\rangle\Big|\\
&\lesssim \sup_{\substack{s\in 2^{-2n_0}\Z\\|s| \leq C 2^{-2n_0}}} \sum_{\substack{y:(s,y)\in\Lambda_{n_0}\\|y-x| \leq C}} 2^{-n_0(2+d+\alpha)} \frac{\big|\langle f , 2^{n_0|\s|} \phi_{t,x}(1-\nu_{n_0})\psi_{n_0,s,y} \rangle\big|}{2^{-n_0\alpha}}\\
&\lesssim \sup_{\substack{s\in 2^{-2n_0}\Z\\|s| \leq C 2^{-2n_0}}} \sum_{\substack{y:(s,y)\in\Lambda_{n_0}\\|y-x| \leq C}}\int_{z\in B(y,2^{-n_0})} 2^{-n_0(2+\alpha)} \frac{\big|\langle f , 2^{n_0|\s|} \phi_{t,x}(1-\nu_{n_0})\psi_{n_0,s,y} \rangle\big|}{2^{-n_0\alpha}}dz\\
&\lesssim 2^{-n_0(2+\alpha)} \rw_T(x) \sup_{\substack{s\in \R\\|s| \leq C 2^{-2n_0}}}  \bigg(\sum_{\substack{y:(s,y)\in\Lambda_{n_0}\\|y-x| \leq C}} \int_{z\in B(y,2^{-n_0})} \sup_{\eta\in\cB^r}\Big|\frac{\langle f , \eta^{2^{-n_0}}_{s,z} \rangle}{\rw_T(x)2^{-n_0\alpha}}\Big|^p dz\bigg)^{\frac{1}{p}}\\
&\lesssim 2^{-n_0(2+\alpha)} \rw_T(x) \sup_{\substack{s\in \R\\|s| \leq C 2^{-2n_0}}}  \bigg(\int_{z\in B(x,C')} \sup_{\eta\in\cB^r} \Big|\frac{\langle f , \eta^{2^{-n_0}}_{s,z} \rangle}{\rw_T(x)2^{-n_0\alpha}}\Big|^p dz\bigg)^{\frac{1}{p}}\;,
\end{equs}
uniformly over all $\phi\in\cB^r$, all $n_0\geq 0$ and all $(t,x)\in\R^{d+1}$. For all $n_0$ such that $(C+1) 2^{-2n_0} < T$, the term on the right hand side is bounded by (\ref{Eq:BoundDistribSpace}), thus concluding the proof of the claim.

\smallskip

\noindent\textit{Step 2: existence.} Let us now consider a distribution $f$ as in the statement, and let us construct its extension. We use the following partition of the complement of the hyperplane $\{t=0\}$
\begin{equ}\label{Eq:Partition1}
\sum_{n\in\Z} \big(\chi(2^{2n} z_0)+\chi(-2^{2n} z_0)\big) \sum_{(s,y)\in\Lambda_n}\tilde{\chi}\big(2^{2n}(z_0-s)\big)\prod_{i=1}^d\tilde{\chi}\big(2^{n}(z_i-y_i)\big) = 1\;,
\end{equ}
for all $z\in\R^{d+1}$ with $z_0\ne 0$. Then, for all $n\in\Z$ and all $(s,y)\in\Lambda_n$, we set
\begin{equ}\label{Eq:DefPsiPart0}
\psi_{n,s,y}(z) = \big(\chi(2^{2n} z_0)+\chi(-2^{2n} z_0)\big)\tilde{\chi}\big(2^{2n}(z_0-s)\big)\prod_{i=1}^d\tilde{\chi}\big(2^{n}(z_i-y_i)\big) \;.
\end{equ}
We need to define $\langle f, \eta^{\lambda}_{t,x}\rangle$ for all those $\eta\in\cB^r$ and $(t,x)\in\R^{d+1}$ such that $t\in[-3\lambda^2,3\lambda^2]$. The uniqueness part of the statement shows that $f$ should not have any contribution on the hyperplane $\{t=0\}$. This suggests to set
\begin{equs}
\langle f, \eta^\lambda_{t,x} \rangle := \sum_{2^{-n} < \lambda}\sum_{(s,y)\in \Lambda_n} \langle  f , \eta^\lambda_{t,x} \psi_{n,s,y} \rangle\;. \label{Eq:DefExtf}
\end{equs}
Notice that we restricted the sum to those $n$ such that $2^{-n} < \lambda$, since otherwise the product $\eta^\lambda_{t,x} \psi_{n,s,y}$ is identically zero. We only need to check that the right hand side makes sense. First, we notice that for any given $n$, the sum over $s$ in (\ref{Eq:DefExtf}) can be restricted to the set
\begin{equs}
\cS^{t,\lambda}_n = \Big\{s\in 2^{-2n}\Z:\quad &s\in [t-\lambda^2-2^{-2n},t+\lambda^2+2^{-2n}]\;,\\
& B(s,2^{-2n}) \cap \mbox{ supp }\big(\chi(2^{2n}\cdot)+\chi(-2^{2n}\cdot)\big) \ne \emptyset \Big\}\;.
\end{equs}
The cardinality of this set is uniformly bounded in $n\geq 0$. Then, for every $n\geq 0$ such that $2^{-n} < \lambda$, we write
\begin{equs}
{}&\bigg\| \sup_{\eta \in \cB^r}\Big| \sum_{s\in\cS_n^{t,\lambda}}\sum_{y\in 2^{-n}\Z^d} \langle f , \eta^\lambda_{t,x} \psi_{n,s,y} \rangle \Big| \bigg\|_{L^p_{x_0,1}}\\
&\lesssim \sup_{s\in \cS_n^{t,\lambda}}\bigg\| \sup_{\eta \in \cB^r} \sum_{\substack{y\in 2^{-n}\Z^d\\|y-x|\leq \lambda+C2^{-n}}} \int_{u\in B(y,2^{-n})} 2^{n|\s|} \big|\langle f , \eta^\lambda_{t,x} \psi_{n,s,y} \rangle \big|du \bigg\|_{L^p_{x_0,1}}\;,
\end{equs}
where $C>0$ depends on the size of the support of $\psi$, and where we have artificially added the integral over $u$ at the second line. At this point, we use Jensen's inequality, the bound (\ref{Eq:BoundDistribSpace}), and the fact that the function $\eta^\lambda_{t,x}\psi_{n,s,y}$ can be written $C'\big(\lambda2^{n}\big)^{-|\s|} \phi^{2^{-n}}_{s,u}$ for some function $\phi\in\cB^r$ and some constant $C'$, where $|C'|$ is bounded uniformly over all $t,x,s,y,u,n$ as above. This yields
\begin{equs}
{}&\lesssim \sup_{s\in\cS_n^{t,\lambda}}2^{-2n}\lambda^{-2}\Big( \int_{u\in B(x_0,3)} \sup_{\phi \in \cB^r}\Big| \langle f , \phi^{2^{-n}}_{s,u} \rangle\Big|^p du \Big)^{\frac{1}{p}}\\
&\lesssim \lambda^{-2} 2^{-n(2+\alpha)}\rw_{t+3\lambda^2}(x_0)\;, 
\end{equs}
uniformly over all $\lambda\in (0,1]$, all $t\leq \lambda^2$, all $x_0\in\R^d$ and all $n\in\Z$ such that $2^{-n} <\lambda$. To get the last bound, we used the fact that for all $s\in\cS_n^{t,\lambda}$, we have $s>3\cdot 2^{-2n}$ and $s<t+2\lambda^2$. Using the assumption $\alpha > -2$, we deduce that
\begin{equs}
{}&\sum_{2^{-n} <\lambda}\bigg\| \sup_{\eta \in \cB^r}\Big| \sum_{s\in\cS_n^{t,\lambda}}\sum_{y\in 2^{-n}\Z^d} \langle f , \eta^\lambda_{t,x} \psi_{n,s,y} \rangle \Big| \bigg\|_{L^p_{x_0,1}}\lesssim \lambda^{\alpha}\rw_{t+3\lambda^2}(x_0)\;, 
\end{equs}
uniformly over all the parameters. Therefore, we have extended $f$ into a genuine distribution over $\R^{d+1}$, with the right regularity index but with a slightly worse weight than desired.

\smallskip

\noindent\textit{Step 3: optimal bound.} We now show that the weight in the last bound can be replaced by $\rw_{t+\lambda^2}(x_0)$ as required. To that end, we refine the mesh of our partition of unity near the maximal time of the support of the test function. We fix $t,x,\lambda$ and assume that $t \leq 3\lambda^2$. We have:
\begin{equ}\label{Eq:Partition2}
\sum_{n\in\Z} \chi\big(2^{2n}(t+\lambda^2-z_0)\big) \!\!\sum_{(s,y)\in\Lambda_n}\tilde{\chi}\big(2^{2n}(z_0-s)\big)\tilde{\chi}\big(2^{n}(z_1-y_1)\big)\ldots\tilde{\chi}\big(2^{n}(z_d-y_d)\big) = 1\;,
\end{equ}
for all $z\in (-\infty,t+\lambda^2)\times\R^d$. Taking the product of (\ref{Eq:Partition1}) and (\ref{Eq:Partition2}), we deduce the existence of a set $\cS_n^{t,\lambda} \subset \R$ and a collection of smooth functions $\psi_{n,s,y}$, compactly supported in $B\big((s,y), 2^{-n}\big)$, indexed by $(s,y)\in \cS_n^{t,\lambda}\times \big(2^{-n}\Z^d\big)$, such that:\begin{enumerate}
\item For all $z\in\R^{d+1}$ such that $z_0\in (-\infty,0)\cup(0,t+\lambda^2)$,
\begin{equ}
\sum_{2^{-n} <\lambda} \sum_{s\in\cS_n^{t,\lambda}}\sum_{y\in 2^{-n}\Z^d} \psi_{n,s,y}(z) = 1\;.
\end{equ}
\item The number of elements of $\cS_n^{t,\lambda}$ is bounded uniformly over all 
$n\in \Z$, and $\cS_n^{t,\lambda} \subset (-\infty,-4\cdot 2^{-2n}] \cup [4\cdot 2^{-2n},t+\lambda^2-4\cdot 2^{-2n}]$,
\item For all $k\in\N^{d+1}$ with $|k|\leq r$, we have $|D^k \psi_{n,s,y}| \lesssim 2^{n|k|}$ uniformly over all $n\in\Z$ and all $(s,y)\in\cS_n^{t,\lambda}\times \big(2^{-n}\Z^d\big)$.
\end{enumerate}
This allows us to write
\begin{equ}\label{Eq:EtaUnity}
\eta^\lambda_{t,x}(z) = \sum_{2^{-n} < \lambda} \sum_{s\in\cS_n^{t,\lambda}}\sum_{y\in 2^{-n}\Z^d} \eta^\lambda_{t,x}(z)\psi_{n,s,y}(z)\;,
\end{equ}
for all $z\in\R^{d+1}$ with $z_0\ne 0$. In the sum over $y$, the number of elements with a non-zero contribution is of order at most $\big(\lambda 2^{n}\big)^d$. From Step 1, we know that the following equality holds
\begin{equs}
\langle f, \eta^\lambda_{t,x} \rangle = \sum_{2^{-n} < \lambda}\sum_{s\in\cS_n^{t,\lambda}}\sum_{y\in 2^{-n}\Z^d} \langle f , \eta^\lambda_{t,x} \psi_{n,s,y} \rangle\;. \label{Eq:DefRfSmallt2}
\end{equs}
Then, we can apply the calculations made in Step 2, the only difference comes from the set $\cS_n^{t,\lambda}$ whose elements are at distance at least $4\cdot 2^{-2n}$ from $t+\lambda^2$. This ensures the required weight.
\end{proof}
\begin{theorem}[Reconstruction with weights]\label{Th:ReconstructionWeight}
Let $(\cT,\cG,\cA)$ be a regularity structure. Let $\gamma > 0$, $p\in[1,\infty)$, $\alpha:=\min \cA$, $r>|\alpha|$ 
and $(\Pi,\Gamma)$ be a model with the weight $\wPi(x)=(1+|x|)^{\frac{c}{2}},x\in\R^d$. In addition to Assumption \ref{AssumptionWRecons} on the weights, we require that 
$\alpha'=\eta\wedge\alpha-c > -2$ and $\gamma-c > 0$. Then, there exists a unique continuous linear map 
$\cR:\ccD^{\gamma,\eta,p}_{\rw,T}\rightarrow\ccE^{\alpha',p}_{\rw,T}$ such that 
$\langle \cR f, \eta \rangle = 0$ whenever $\eta$ is supported in $(-\infty,0)\times \R^d$, and
\begin{equ}[Eq:BoundReconstructionpLocal]
\bigg\| \sup_{\eta\in\cB^r}\big|\langle\cR f - \Pi_{t,x} f(t,x) , \eta_{t,x}^\lambda \rangle\big| \bigg\|_{L^p_{x_0,1}}\lesssim C \lambda^{\gamma-c}t^{\frac{\eta-\gamma}{2}}\rw_{t+\lambda^2}(x_0) \;,
\end{equ}
uniformly over all $\lambda\in(0,1]$, all $x_0\in\R^d$, all $t\in[3\lambda^2,T-\lambda^2]$, all $f\in\ccD^{\gamma,\eta,p}_{\rw,T}$ and all admissible models $(\Pi,\Gamma)$. Here $C:= \$\Pi\$ (1+\$\Gamma\$)\$f\$$. Furthermore, we have the bound
\begin{equ}[Eq:BoundReconstructionWeight]
\bigg\| \sup_{\eta\in\cB^r}\big|\langle\cR f, \eta_{t,x}^\lambda \rangle\big|\bigg\|_{L^p_{x_0,1}} \lesssim C\lambda^{\alpha\wedge\eta-c}\rw_{t+\lambda^2}(x_0) \;,
\end{equ}
uniformly over all $\lambda\in(0,1]$, all $x_0\in\R^d$, all $t\in(0,T-\lambda^2]$ and all $f\in\ccD^{\gamma,\eta,p}_{\rw,T}$.

If $(\bar{\Pi},\bar{\Gamma})$ is a second model for $\cT$ and if $\bar{\cR}$ is its associated reconstruction operator, then we set
\begin{equ}
\tilde{C}:=\$\Pi\$(1+\$\Gamma\$)\$f;\bar{f}\$+\$\Pi-\bar{\Pi}\$(1+\$\Gamma\$)\$\bar f\$+\$\bar{\Pi}\$\$\Gamma-\bar\Gamma\$\$\bar f\$\;,
\end{equ}
and we have the bound
\begin{equs}[Eq:BoundReconstTwoModelsLoc]
\bigg\| \sup_{\eta\in\cB^r}\big|&\langle \cR f - \bar{\cR}\bar{f} - \Pi_{t,x} f(t,x) + \bar{\Pi}_{t,x}\bar{f}(t,x) , \eta_{t,x}^\lambda \rangle\big|\bigg\|_{L^p_{x_0,1}}\\
&\lesssim \tilde{C}\lambda^{\gamma-c} t^{\frac{\eta-\gamma}{2}}\rw_{t+\lambda^2}(x_0)\;,
\end{equs}
uniformly over all $\lambda\in(0,1]$, all $x_0\in\R^d$, all $t\in(3\lambda^2,T-\lambda^2)$, all $f\in\cD^{\gamma,\eta,p}_{\rw,T}$, all $\bar{f}\in\bar{\ccD}^{\gamma,\eta,p}_{\rw,T}$ and all admissible models $(\Pi,\Gamma)$, $(\bar\Pi,\bar\Gamma)$.
We also have
\begin{equ}[Eq:BoundReconstTwoModels]
\bigg\| \sup_{\eta\in\cB^r}\big|\langle \cR f - \bar{\cR}\bar{f}, \eta_{t,x}^\lambda \rangle\big|\bigg\|_{L^p_{x_0,1}}\lesssim \tilde{C}\lambda^{\alpha\wedge\eta-c}\rw_{t+\lambda^2}(x_0)\;,
\end{equ}
uniformly over the same parameters.
\end{theorem}
Notice that in these statements we lose a factor $\lambda^{-c}$ compared to what one 
would have expected: this is the price we pay for having added weights to our spaces and requiring uniformity in 
space. However, we will see in the sequel that we can choose the constant $c$ as small as we want.
\begin{proof}
We only need to show that there is a unique distribution $\cR f$, on the set of all test functions whose support does not intersect the hyperplane $\{t=0\}$, that fulfills the requirements of the theorem for these test functions. Then, Proposition \ref{Prop:UniqDistr} yields the desired result.\\
First, we set $\langle \cR f , \eta \rangle:=0$ for every $\eta\in\cB^r$ which is supported in the half-space $\{t<0\}$. Second, let $\lambda \in (0,1]$, $x\in\R^d$ and $t\in[3\lambda^2,T-\lambda^2]$. By a simple localisation argument, one can build an element $\tilde{f}\in\cD^{\gamma,p}$ such that $\tilde{f}$ coincides with $f$ in $[t-2\lambda^2,t+\lambda^2]\times B(x,3)$ and vanishes outside $[t-3\lambda^2,t+2\lambda^2]\times B(x,4)$. Indeed, it suffices to lift into the polynomial regularity structure a smooth function equal to $1$ on $[t-2\lambda^2,t+\lambda^2]\times B(x,3)$, and vanishing outside $[t-3\lambda^2,t+2\lambda^2]\times B(x,4)$, and to define $\tilde{f}$ as the product of $f$ with this smooth function (this may require to extend our original regularity structure with the polynomials, and to define the canonical product between elements in the regularity structure and polynomials).

Using the reconstruction theorem in $\cD^{\gamma,p}$, we set $\langle \cR f , \eta^\lambda_{t,x} \rangle := \langle \cR \tilde{f} , \eta^\lambda_{t,x} \rangle$. We now show (\ref{Eq:BoundReconstructionpLocal}). Recall the definition of $B^n=B^n_{\lambda,t,x_0}$ from Theorem \ref{Th:Reconstruction}. Notice that
\begin{equ}
\$\Pi \$_{B^n}\big(1+\$ \Gamma \$_{B^n}\big)\$f \$_{B^n} \lesssim t^{\frac{\eta-\gamma}{2}}\wPi(x_0)^2 \sup_{\zeta}\sup_{i\in\{1,2\}}\wi_{t+\lambda^2-2^{-2n}}(x_0,\zeta)\;,
\end{equ}
uniformly over all $\lambda\in(0,1]$, all $x_0\in\R^d$, all $t\in[3\lambda^2,T-\lambda^2]$, all $f\in\ccD^{\gamma,\eta,p}_{\rw,T}$ and all $n\geq 0$. Using (\ref{CondP1}), we deduce that the right hand side is actually bounded by a term of order $t^{\frac{\eta-\gamma}{2}} \rw_{t+\lambda^2}(x_0) 2^{nc}$ uniformly over all the parameters. Therefore, by (\ref{Eq:BoundReconstructionp}), we deduce that (\ref{Eq:BoundReconstructionpLocal}) holds.\\
This determines the value of $\langle \cR f , \phi \rangle$, for any test function $\phi$ whose support does not intersect the hyperplane $\{t=0\}$. Indeed, any such function can be split into a \textit{finite} sum of functions of the form $\eta^\lambda_{t,x}$, with $t\geq 3\lambda^2$, on which $\cR f$ has already been constructed. It is then simple to check that $\cR f$ is a well-defined distribution on the set of test functions whose support does not intersect the hyperplane $\{t=0\}$. We can apply Proposition \ref{Prop:UniqDistr}, and the statement of the theorem follows.\\
The case of two models is handled similarly, using the bound (\ref{Eq:BoundReconstructionTwo}) from the reconstruction theorem in $\cD^{\gamma,p}$, thus
concluding the proof.
\end{proof}

\section{Convolution with the heat kernel}\label{sec:Schauder}
The goal of this section is to define an operator that plays the role of the convolution with the heat kernel, but at the level of modelled distributions. This will be carried out separately for the singular part $P_+$ and the smooth part $P_-$ of the heat kernel, as defined in Lemma \ref{Lemma:Kernel}. Although such an abstract operator was defined in Section 5 of~\cite{Hairer2014}, the fact that we have incorporated weights in our spaces imposes some additional constraints on this map. The main difficulty comes with the singular part of the kernel $P_+$, which is handled in Theorem \ref{Th:Integration}. The smooth part is simpler, and is addressed in Proposition \ref{Prop:IntegrationSmooth}. We end this section with the treatment of the initial condition.

From now on, we take the following values for the parameters:
\begin{equs}
\alpha = -\frac{3}{2} - \kappa \;,\qquad \eta = -\frac{1}{2}+3\kappa\;,\qquad \gamma = \frac{3}{2}+2\kappa \;.\end{equs}
They fulfill the requirements that $\gamma > -\alpha$ and $\eta-\gamma>-2$. Recall that $\alpha$ is the regularity of the noise, $\eta$ is the regularity of the initial condition and $\gamma$ is the upper bound of the homogeneities involved in the regularity structure.

We also consider, for all $t\in \R$ and all $\zeta\in\cA$, two collections of weights $\wun_t(\cdot,\zeta)$ and $\wde_t(\cdot,\zeta)$ on $\R^d$. Observe that it is meaningful to write $\wi_t(\cdot,\tau)$ to denote $\wi_t(\cdot,|\tau|)$ for any $\tau\in\cT$.

\begin{assumption}[Weights and convolution]\label{AssumptionW}
Let $c>0$ and $\gamma' > 0$. In addition to Assumption \ref{AssumptionWRecons}, we impose that:
\begin{equs}
\wi_t(x,\tau) &\leq \wi_t(x,\cI(\tau\Xi))\;,\label{CondP2}\tag{W-2}\\[2pt]
\wPi(x)\wi_t(x,\tau\Xi)&\leq \wi_t(x,X^k)\;,\qquad \mbox{whenever } |\tau|+\alpha\leq |k|-2\;,\qquad\label{CondP3}\tag{W-3}\\[2pt]
\wPi(x)\wun_t(x,\tau\Xi)&\leq \wde_t(x,X^k) \;,\label{CondP4}\tag{W-4}\\[2pt]
\wi_t(x,\tau\Xi) &= \wi_t(x,\tau)\;,\label{CondP5}\tag{W-5}
\end{equs}
for all $x\in \R^d$, all $s<t \in (-\infty,T]$, all $\tau\in \cU_{<\gamma'}$, all $k\in\N^{d+1}$ such that $|k| < \gamma'$ and all $i\in\{1,2\}$.
\end{assumption}
\noindent Take $\gamma'=\gamma+\alpha+2-c$ with $c \in (0,\frac{\kappa}{2})$. Here is a possible choice of weights satisfying Assumption \ref{AssumptionW}:
\begin{equs}[Eq:Weights]
\wPi(x) &:= (1+|x|)^{\frac{c}{28}(1-2\kappa)}\;,\\
\wun_t(x,\zeta) &:= (1+|x|)^{\frac{c}{14} \zeta} \, e^{t (1+|x|)} \,e^{\ell (1+|x|)} \;,\\
\wde_t(x,\zeta) &:= (1+|x|)^{\frac{c}{14}(\zeta+3)} \, e^{t (1+|x|)}\,e^{\ell (1+|x|)} \;,
\end{equs}
where $\zeta\in\cA_{<\gamma'}(\cU)$ and $\ell$ is a constant which will allow us to consider an initial condition in a weighted space.

\begin{lemma}
Suppose that $u\in\ccD^{\gamma,\eta,p}_{\w,T}(\cU)$. Then, the map $f=u\cdot\Xi$ belongs to the space $\ccD^{\gamma+\alpha,\eta+\alpha,p}_{\w,T}(\cF)$.
\end{lemma}
\begin{proof}
By construction, we have $\Gamma_{z,z'} (\tau\Xi) = (\Gamma_{z,z'}\tau) \Xi$ for all $\tau\in\cU$ and all $z,z'\in\R^{d+1}$, 
so that $|f(z)-\Gamma_{z,z'} f(z')|_{\zeta} = |u(z)-\Gamma_{z,z'} u(z')|_{\zeta-\alpha}$ for all $\zeta\in\cA(\cF)$. Using (\ref{CondP5}), it is then immediate to check the statement.
\end{proof}

\subsection{Singular part of the heat kernel}

Let $u$ be an element of $\ccD:=\ccD^{\gamma,\eta,p}_{T,\rw}(\cU)$, and set $f=u\cdot\Xi \in \ccD^{\gamma+\alpha,\eta+\alpha,p}_{T,\rw}$. For any given $\gamma'>0$, we define the abstract convolution map as follows:
\begin{equs}
\big(\cP_+^{\gamma'} f\big)(t,x) &= \cI(f(t,x)) \label{Eq:DefConvAbstract}\\
&\quad+ \sum_{\zeta\in \cA(\cF)}\sum_{|k| < (\zeta+2)\wedge {\gamma'}} {X^k\over k!} \big\langle \Pi_{t,x} \cQ_\zeta f(t,x) , D^{k}P_+\big((t,x)-\cdot\big) \big\rangle\\
&\quad+ \sum_{|k| < {\gamma'}} \frac{X^k}{k!} \big\langle \cR f - \Pi_{t,x} f(t,x) , D^k P_+\big((t,x)-\cdot\big) \big\rangle \;.
\end{equs}
The well-definiteness of this operator is a consequence of the next result, which is the second main technical step of the present work.

\begin{theorem}\label{Th:Integration}
Take $c\in (0,\frac{\kappa}{2})$ and set $\gamma' =\gamma +2+\alpha-c$, $\eta'=\eta+2+\alpha-c$. We assume that $\gamma',\eta'\notin\N$. Let $u\in\ccD=\ccD^{\gamma,\eta,p}_{T,\rw}(\cU)$ and set $f=u\cdot\Xi \in\ccD^{\gamma+\alpha,\eta+\alpha,p}_{T,\rw}(\cF)$. Then, under Assumption \ref{AssumptionW} on the weights, we have $\cP_+^{\gamma'} f \in \ccD':=\ccD^{\gamma',\eta',p}_{T,\rw}(\cU)$ and the bound
\begin{equs}
\$\cP_+ f\$_{\ccD'} \lesssim \$\Pi\$(1+\$\Gamma\$)\$u\$_{\ccD}\;.
\end{equs}
holds uniformly over all $T$ in a compact set of $\R_+$, all $\ell$ in a compact set of $\R$, all $u\in\ccD$ and all admissible models $(\Pi,\Gamma)$. In addition, we have the identity
\begin{equs}[e:commutInt]
\cR \cP_+ f = P_+ * \cR f \;.
\end{equs}
Moreover, if $(\bar \Pi, \bar \Gamma)$ is another model with the same weight $\wPi$ and if $\bar{u}$ belongs to the corresponding space $\bar{\ccD}$ equipped with the same weights $\wun,\wde$, then we have the bound
\begin{equs}
\;&\$\cP_+ f;\bar{\cP}_+ \bar{f}\$_{\ccD', \bar{\ccD}'} \lesssim \$\Pi\$(1+\$\Gamma\$)\$u;\bar{u}\$_{\ccD,\bar{\ccD}}\\
\;&+\big(\$\Pi-\bar{\Pi}\$(1+\$\bar{\Gamma}\$)+ \$\bar\Pi\$\$\Gamma-\bar\Gamma\$\big)\$\bar u\$_{\ccD}\;,
\end{equs}
uniformly over all $T$ in a compact set of $\R_+$, all $\ell$ in a compact set of $\R$, all $\Pi,\bar\Pi,\Gamma,\bar\Gamma$ and all $u,\bar u$.
\end{theorem}
Before we proceed to the proof of the theorem, we collect a few technical facts. Let us denote by $\cB^r_-$ the subset of $\cB^r$ whose elements are supported in the half-space $\{t\leq 0\}$. Using Theorem \ref{Th:ReconstructionWeight}, we immediately get\begin{equ}\label{Eq:I1}
\bigg\| \sup_{\eta \in \cB^r_-} \Big|\frac{\langle \cR f, \eta^{\lambda}_{t,x}\rangle}{\rw_t(x)} \Big| \bigg\|_{L^p} \lesssim \lambda^{\eta+\alpha-c}\;,
\end{equ}
uniformly over all $t\in (0,T]$, all $\lambda\in(0,1]$ and all $f\in\ccD^{\gamma+\alpha,\eta+\alpha,p}_{T,\rw}$, as well as
\begin{equ}\label{Eq:I2}
{}\bigg\| \sup_{\eta \in \cB^r_-} \Big|\frac{\langle \cR f- \Pi_{t-\lambda^2,x} f(t-\lambda^2,x), \eta^{\lambda}_{t,x}\rangle}{\rw_t(x)} \Big| \bigg\|_{L^p} \lesssim \lambda^{\gamma+\alpha-c} t^{\frac{\eta-\gamma}{2}}\;,
\end{equ}
uniformly over all $t\in [4\lambda^2,T]$, all $\lambda\in(0,1]$ and all $f\in\ccD^{\gamma+\alpha,\eta+\alpha,p}_{T,\rw}$. These two bounds will be applied repeatedly to the function $P_0\big((t,x)-\cdot\big) \in \cB^r_-$ as well as its rescalings $P_n$, $n\geq 0$.

For all $z,z' \in \R^{d+1}$, all $k\in\N^{d+1}$ such that $|k| < \gamma'$, and all $n\geq 0$, we define
\begin{equs}
P_{n;z,z'}^{k,\gamma'}(\cdot) := D^k P_n(z-\cdot) - \sum_{\ell:|k|+|\ell| < \gamma'} \frac{(z-z')^\ell}{\ell !} D^{k+\ell} P_n(z'-\cdot)\;.
\end{equs}
Using the classical Taylor formula, one obtains the following identities:
\begin{equs}[Eq:TaylorTime]
{}&P_{n;tx,sx}^{k,\gamma'}(\cdot) =\sum_{\substack{\ell=(\ell_0,0,\ldots,0)\\ \gamma' < |k|+2\ell_0 < \gamma'+2}}(t-s)^\ell\\
&\int_0^1 (1-u)^{|\ell|-1}\frac{|\ell|}{\ell!} D^{k+\ell} P_n\big((s+u(t-s),x)-\cdot\big) du\;,
\end{equs}
and
\begin{equs}[Eq:TaylorSpace]
{}&P_{n;ty,tx}^{k,\gamma'}(\cdot) =\sum_{\substack{\ell=(0,\ell_1,\ldots,\ell_d)\\ \gamma' < |k|+|\ell| < \gamma'+1}}(y-x)^\ell\\
&\int_0^1 (1-u)^{|\ell|-1}\frac{|\ell|}{\ell!} D^{k+\ell} P_n\big((t,x+u(y-x))-\cdot\big) du\;,
\end{equs}
for all $(t,x),(s,y) \in \R^{d+1}$. In these equations and later on in the proof of the theorem, we use the notation $(y-x)^\ell$ and $(t-s)^\ell$ for $(z-z')^\ell$ where $z=(0,y)$, $z'=(0,x)$ in the first case, and $z=(t,0)$, $z'=(s,0)$ in the second case. Notice that in the two formulae (\ref{Eq:TaylorTime}) and (\ref{Eq:TaylorSpace}), we do not consider space and time translations simultaneously. For space-time translations, the situation is slightly more involved due to the scaling $\s$ so we rely on the following result.
\begin{lemma}[Prop 11.1~\cite{Hairer2014}]\label{Lemma:TaylorSpaceTime}
Let $\partial \gamma'$ be the set of indices
\begin{equ}
\{\ell'\in\N^{d+1}: |\ell'| >\gamma', |\ell'-e_{\m(\ell')}| < \gamma'\} \;,
\end{equ}
where $e_i$ is the unit vector of $\R^{d+1}$ in the direction $i\in\{0,\ldots,d\}$, and $\m(\ell'):=\inf\{i: \ell'_i \ne 0\}$. For all $z,z'\in\R^{d+1}$ and all $k\in\N^{d+1}$ such that $|k|<\gamma'$, we have
\begin{equs}
P_{n;z,z'}^{k,\gamma'}(\cdot) &= \sum_{\ell:k+\ell \in \partial \gamma'} \int_{\R^{d+1}} D^{k+\ell} P_n(z'+h-\cdot) \mu^{k+\ell}(z-z',dh)\;.
\end{equs}
Here, $\mu^{k+\ell}(z-z',dh)$ is a signed measure on $\R^{d+1}$, supported in the set $\{\tilde{z}\in\R^{d+1}: \tilde{z}_i \in [0,z_i-z'_i]\}$ and whose total mass is given by $\frac{(z-z')^{k+\ell}}{(k+\ell)!}$.
\end{lemma}
For the sake of readibility, we drop the superscript $\gamma'$ in the operator $\cP_+^{\gamma'}$.
\begin{proof}[of Theorem \ref{Th:Integration}]
From now on, the symbol $\lesssim$ will be taken uniformly over all $\ell$ in a given compact set of $\R$ and all $T$ in a given compact set of $\R_+$. Also, the implicit constant associated to this symbol always dominates the constant of (\ref{CondP1}) as well as all the constants associated with Definition \ref{Def:Weights} for the corresponding weights. We provide a complete proof of the statement concerning a single model. To prove the part with two different models, the arguments work almost verbatim given the following two identities:
\begin{equs}
\Pi_z \cQ_\zeta a - \bar\Pi_z\cQ_\zeta\bar{a} &= \Pi_z\cQ_\zeta\big(a-\bar{a}\big) + \big(\Pi_z-\bar\Pi_z\big)\cQ_\zeta \bar a\;,\\
\big(\Pi_{z'} \cQ_\zeta \Gamma_{z',z}-\bar\Pi_{z'}\cQ_\zeta\bar\Gamma_{z',z}\big)\bar a &= \Pi_{z'} \cQ_\zeta \big(\Gamma_{z',z}-\bar\Gamma_{z',z}\big)\bar a + \big(\Pi_{z'}-\bar\Pi_{z'}\big)\cQ_\zeta\bar\Gamma_{z',z}\bar a\;.
\end{equs}
Let $u\in\ccD$ and set $f=u\cdot\Xi$. For simplicity, we assume that $\$u\$=1$. The proof is divided into four steps. We will use repeatedly Lemma \ref{Lemma:Kernel} without further mention.\\
For all $n\geq 0$ and all $(t,x)\in (0,T]\times\R^d$, we define
\begin{equs}
\big(\cP_n^{\gamma'} f\big)(t,x) &:= \sum_{\zeta\in \cA(\cF)}\sum_{|k| < (\zeta+2)\wedge {\gamma'}} {X^k\over k!} \big\langle \Pi_{t,x} \cQ_\zeta f(t,x) , D^{k}P_n\big((t,x)-\cdot\big) \big\rangle\\
&\quad+ \sum_{|k| < {\gamma'}} \frac{X^k}{k!} \big\langle \cR f - \Pi_{t,x} f(t,x) , D^k P_n\big((t,x)-\cdot\big) \big\rangle \;.
\end{equs}
We will make sense of (\ref{Eq:DefConvAbstract}) by showing that the series of the coefficients on the monomials of $\big(\cP_n^{\gamma'} f\big)(t,x)$ is absolutely convergent. We distinguish three types of terms in the $\ccD$-norm: the \textit{local terms} that appear at the first line of (\ref{Eq:NormccD}), the terms of \textit{translation in space} at the second line, and the terms of \textit{translation in time} at the third line. We argue differently for each of them.


\medskip
\noindent\textit{First step: local terms.} For all non-integer values $\zeta\in \cA_{<\gamma'}(\cU)$, we have:
\begin{equs}
\bigg\| \frac{|\cI f(t,x)|_{\zeta}}{t^{\frac{(\eta'-\zeta)\wedge 0}{2}}\wun_t(x,\zeta)}\bigg\|_{L^p} \lesssim \bigg\| \frac{|u(t,x)|_{\zeta-2-\alpha}}{t^{\frac{(\eta-\zeta+2+\alpha)\wedge 0}{2}}\wun_t(x,\zeta-2-\alpha)}\bigg\|_{L^p} \leq 1\;,
\end{equs}
where we have used Condition (\ref{CondP2}) and the fact that $\eta'-\zeta$ and $\eta'+c-\zeta$ have the same sign. Therefore, the desired bound follows.

We turn to the integer levels $k$ such that $|k| < \gamma'$. We distinguish two sub-cases. First, if $t\leq 4\cdot 2^{-2n}$, we write $k!\cQ_k(\cP_n f)(t,x)$ as:
\begin{equs}[Eq:Punctual1]
\langle\cR f , D^kP_n\big((t,x)-\cdot\big)\rangle -\sum_{\zeta \leq |k|-2}\langle \Pi_{t,x} \cQ_{\zeta}f(t,x) , D^kP_n\big((t,x)-\cdot\big) \rangle\;.\;\;
\end{equs}
Using $(\ref{Eq:I1})$, we get
\begin{equ}
\bigg\| \frac{\big\langle \cR f, D^kP_n\big((t,x)-\cdot\big) \big\rangle}{\wun_t(x,|k|)} \bigg\|_{L^p} \lesssim 2^{-n(\eta'-|k|)}\;,
\end{equ}
uniformly over all the corresponding $n$ and $t$. Since $\eta'\notin\N$, the sum over these $n$ yields a bound of order $t^{\frac{(\eta'-|k|)\wedge 0}{2}}$, as required. We now bound the second term of (\ref{Eq:Punctual1}). When $\zeta=|k|-2$,
this term has a zero contribution since $P_n$ kills polynomials of degree $r$. On the other hand, we use (\ref{CondP3}) to get for all $\zeta < |k|-2$
\begin{equ}
\bigg\| \frac{\big\langle \Pi_{t,x} \cQ_{\zeta}f(t,x), D^kP_n\big((t,x)-\cdot\big) \big\rangle}{\wun_t(x,|k|)} \bigg\|_{L^p} \lesssim 2^{-n(2+\zeta-|k|)}t^{\frac{\eta+\alpha-\zeta}{2}}\;,
\end{equ}
uniformly over all the corresponding $n$ and $t$. Summing over all the corresponding $n$ yields a bound of the required order.

We now treat the case $t\geq 4\cdot 2^{-2n}$. We set $t_n=t-2^{-2n}$, and write $k!\cQ_k(\cP_n f)(t,x)$ as:
\begin{equs}[Eq:Punctual2]
{}&\langle \cR f-\Pi_{t_n,x} f(t_n,x) , D^kP_n\big((t,x)-\cdot\big)\rangle\\
&-\sum_{\zeta \leq |k|-2}\langle \Pi_{t,x} \cQ_{\zeta}\big(f(t,x)-\Gamma_{t,t_n}^xf(t_n,x)\big) , D^kP_n\big((t,x)-\cdot\big)\rangle\\
&+\sum_{\zeta > |k|-2}\langle \Pi_{t_n,x}\cQ_\zeta f(t_n,x) , D^kP_n\big((t,x)-\cdot\big)\rangle\;.
\end{equs}
The first and second terms can be treated easily using (\ref{Eq:I2}) and (\ref{CondP3}) respectively. We now deal with the third term. Using (\ref{CondP1}), we get for all $\zeta > |k|-2$
\begin{equs}
{}&\bigg\|\frac{\langle \Pi_{t_n,x}\cQ_\zeta f(t_n,x) , D^kP_n\big((t,x)-\cdot\big)\rangle}{\wun_t(x,|k|)}\bigg\|_{L^p} \lesssim t^{\frac{\eta+\alpha-\zeta}{2}}2^{-n(2+\zeta-|k|-c)}\;,
\end{equs}
uniformly over all $n$ such that $t\geq 4\cdot 2^{-2n}$. Since $c < \kappa/2$, we have $2+\zeta-|k| - c > 0$, so that the sum over these $n$ yields the required bound.

\medskip
\noindent\textit{Second step: translation in space.} We now look at $(\cP_+^{\gamma'} f)(t,y)-\Gamma_{y,x}^t(\cP_+^{\gamma'} f)(t,x)$ with $|x-y| \leq 1$. If $\zeta\in\cA_{<\gamma'}(\cU)\backslash\N$, then the only contribution comes from $\cI$ and we have:
\begin{equs}
{}&\bigg\|\frac{\int_{y\in B(x,\lambda)}\lambda^{-d}|\cI \big(f(t,y) - \Gamma_{y,x}^t f(t,x)\big)|_\zeta dy}{t^{\frac{\eta'-\gamma'}{2}}\lambda^{\gamma'-\zeta}\wde_t(x,\zeta)}\bigg\|_{L^p}\\
\lesssim &\; \bigg\|\frac{\int_{y\in B(x,\lambda)}\lambda^{-d}| \big(u(t,y) - \Gamma_{y,x}^t u(t,x)\big)|_{\zeta-\alpha-2}dy}{t^{\frac{\eta-\gamma}{2}}\lambda^{\gamma-\zeta+\alpha+2}\wde_t(x,\zeta-\alpha-2)}\bigg\|_{L^p}\;,
\end{equs}
where we have used (\ref{CondP2}) and the identity $\eta'-\eta=\gamma'-\gamma=2+\alpha-c$ with $c>0$. The required bound follows.

We turn to the integer levels $k$ with $|k|< \gamma'$. We first treat the case $\lambda^2 \leq t \leq 36\cdot 2^{-2n}$. By Taylor's formula, we write $k!\cQ_k\big((\cP_n f)(t,y)-\Gamma_{y,x}^t (\cP_n f)(t,x)\big)$ as:
\begin{equs}[Eq:Transla1]
{}& \langle \cR f , P_{n;ty,tx}^{k,\gamma'}\rangle -\langle \Pi_{t,x} f(t,x), P_{n;ty,tx}^{k,\gamma'} \rangle\\
&-\sum_{\zeta \leq |k| - 2} \langle \Pi_{t,y}\cQ_\zeta\big(f(t,y)-\Gamma_{y,x}^t f(t,x)\big) , D^k P_n\big((t,y)-\cdot\big) \rangle \;.
\end{equs}
Using (\ref{Eq:TaylorSpace}), we deduce that for any distribution $g$, we have
\begin{equ}\label{Eq:Observation}
\big|\big\langle g, P_{n;ty,tx}^{k,\gamma'} \big\rangle\big| \lesssim \sup_{\eta\in\cB^r_-} \big|\big\langle g, \eta_{t,x}^{C2^{-n}} \big\rangle\big| |y-x|^{\lceil \gamma'\rceil-|k|} 2^{-n(2-\lceil \gamma'\rceil)}\;,
\end{equ}
uniformly over all $y\in B(x,\lambda)$ and all $n\geq 0$, for some constant $C$ independent of everything. 
Using (\ref{Eq:I1}), we thus get
\begin{equ}
\bigg\| \frac{ \int_{y\in B(x,\lambda)} \lambda^{-d} \big|\big\langle \cR f, P_{n;ty,tx}^{k,\gamma'} \big\rangle\big| dy}{\wde_t(x,|k|)} \bigg\|_{L^p} \lesssim \lambda^{\lceil \gamma'\rceil-|k|}2^{-n(\eta'-\lceil \gamma' \rceil)}\;,
\end{equ}
uniformly over all $\lambda^2\leq t \leq 36\cdot 2^{-2n}$. Since $\eta'-\gamma' < 0$, the sum over all these $n$ yields a bound of order $t^{\frac{\eta'-\gamma'}{2}}\lambda^{\gamma'-|k|}$. We turn to the second term of (\ref{Eq:Transla1}). Using (\ref{CondP4}) and (\ref{Eq:Observation}), we get for all $\zeta \in\cA_{<\gamma+\alpha}(\cF)$
\begin{equ}
\bigg\| \frac{\int_{y\in B(x,\lambda)} \lambda^{-d}\big|\big\langle \Pi_{t,x} \cQ_\zeta f(t,x), P_{n;ty,tx}^{k,\gamma'} \big\rangle\big| dy}{\wde_t(x,|k|)} \bigg\|_{L^p} \lesssim \lambda^{\lceil \gamma'\rceil-|k|}2^{-n(2+\zeta-\lceil \gamma' \rceil)} t^{\frac{\eta+\alpha-\zeta}{2}}\;.
\end{equ}
Since $2+\zeta < \gamma'$, the sum over all $n$ such that $t\leq 36\cdot 2^{-2n}$ yields a bound of the right order. Regarding the third term of (\ref{Eq:Transla1}), notice that it actually vanishes whenever $\zeta=|k|-2$ since $P_n$ kills polynomials of order $r$. We use (\ref{CondP3}) to obtain for every $\zeta < |k|-2$
\begin{equs}
{}&\bigg\| \frac{\int_{y\in B(x,\lambda)} \lambda^{-d}\big|\big\langle \Pi_{t,y}\cQ_\zeta\big(f(t,y)-\Gamma_{y,x}^tf(t,x)\big) , D^k P_n\big((t,y)-\cdot\big) \big\rangle\big| dy}{\wde_t(x,|k|)} \bigg\|_{L^p}\\
&\lesssim \bigg\| \frac{\int_{y\in B(x,\lambda)}\lambda^{-d}|f(t,y)-\Gamma_{y,x}^tf(t,x)|_\zeta dy}{\wde_t(x,|k|)} \bigg\|_{L^p} 2^{-n(2+\zeta-|k|)}\\
&\lesssim t^{\frac{\eta-\gamma}{2}} \lambda^{\gamma+\alpha-\zeta} 2^{-n(2+\zeta-|k|)}\;,
\end{equs}
uniformly over all the corresponding parameters. Summing over the corresponding $n$, one gets a bound of the right order.

We now turn to the case $\lambda^2 \leq 4\cdot 2^{-2n} < 36\cdot 2^{-2n} \leq t$. Recall that $2^{-n}+\lambda$ is the size of the support of the test functions involved in (\ref{Eq:Observation}). We set $t_n=t-9\cdot 2^{-2n}$, and we observe that $t_n \geq 3(2^{-n}+\lambda)^2$. Then, we write $k!\cQ_k\big((\cP_n f)(t,y)-\Gamma_{y,x}^t (\cP_n f)(t,x)\big)$ as:
\begin{equs}
{}& \langle \cR f - \Pi_{t_n,x} f(t_n,x) , P_{n;ty,tx}^{k,\gamma'}\rangle - \langle \Pi_{t,x} \big(f(t,x)- \Gamma_{t,t_n}^x f(t_n,x)\big) , P_{n;ty,tx}^{k,\gamma'}\rangle \\
&-\sum_{\zeta \leq |k| - 2} \big\langle \Pi_{t,y}\cQ_\zeta\big(f(t,y)-\Gamma_{y,x}^t f(t,x)\big) , D^{k} P_n\big((t,y)-\cdot\big) \big\rangle\;. \label{Eq:Transla2}
\end{equs}
The first two terms can be easily bounded using (\ref{Eq:Observation}), together with (\ref{Eq:I2}) and (\ref{CondP4}) respectively. The third term coincides with the third term of (\ref{Eq:Transla1}), and the bound follows from the same arguments.

In the case $4\cdot 2^{-2n} \leq \lambda^2 \leq t$, we set $t_n=t-2^{-2n}$ and write $k!\cQ_k\big((\cP_n f)(t,y)-\Gamma_{y,x}^t (\cP_n f)(t,x)\big)$ as:
\begin{equs}
{}& \langle \cR f - \Pi_{t_n,y} f(t_n,y) , D^{k} P_n\big((t,y)-\cdot\big)\rangle \\
&- \langle \cR f - \Pi_{t_n,x} f(t_n,x) , \sum_{|k|+|\ell| < \gamma'}\frac{(y-x)^\ell}{\ell !}D^{k+\ell} P_n\big((t,x)-\cdot\big)\rangle\\
&- \sum_{\zeta \leq |k| - 2}\langle \Pi_{t,y}\cQ_\zeta \big(f(t,y)-\Gamma_{t,t_n}^y f(t_n,y)\big) , D^{k} P_n\big((t,y)-\cdot\big)\rangle \\
&+\sum_{\zeta > |k| - 2}\langle \Pi_{t,y}\cQ_\zeta \Gamma_{t,t_n}^y \big(f(t_n,y)-\Gamma_{y,x}^{t_n}f(t_n,x)\big) , D^{k} P_n\big((t,y)-\cdot\big)\rangle \label{Eq:Transla3}\\
&-\sum_{\zeta > |k| - 2}\langle \Pi_{t,y}\cQ_\zeta \Gamma_{y,x}^t \big(f(t,x)-\Gamma_{t,t_n}^x f(t_n,x)\big) , D^{k} P_n\big((t,y)-\cdot\big)\rangle\\
&+\langle \Pi_{t,x}\big(f(t,x)- \Gamma_{t,t_n}^x f(t_n,x)\big) , \sum_{|k|+|\ell| < \gamma'}\frac{(y-x)^\ell}{\ell !}D^{k+\ell} P_n\big((t,x)-\cdot\big)\rangle\;.
\end{equs}
The bounds for the two first terms follow easily from (\ref{Eq:I2}). The third term vanishes when $\zeta=|k|-2$ since $P_n$ kills polynomials of order $r$. On the other hand, for all $\zeta < |k|-2$ we have
\begin{equs}
{}&\bigg\| \frac{\int_{y\in B(x,\lambda)} \lambda^{-d}\big|\big\langle \Pi_{t,y}\cQ_\zeta \big(f(t,y)-\Gamma_{t,t_n}^y f(t_n,y)\big) , D^{k} P_n\big((t,y)-\cdot)\big) \big\rangle\big| dy}{\wde_t(x,|k|)} \bigg\|_{L^p}\\
& \lesssim \bigg\| \frac{\int_{y\in B(x,\lambda)} \lambda^{-d}\big|f(t,y)-\Gamma_{t,t_n}^y f(t_n,y)\big|_\zeta dy}{\wun_t(x,\zeta)} \bigg\|_{L^p} 2^{-n(2+\zeta-|k|)}\\
&\lesssim \bigg\| \frac{\big|f(t,x)-\Gamma_{t,t_n}^x f(t_n,x)\big|_\zeta }{\wun_t(x,\zeta)} \bigg\|_{L^p} 2^{-n(2+\zeta-|k|)}\\
&\lesssim t^{\frac{\eta-\gamma}{2}} 2^{-n(\gamma'-|k|)}\;,
\end{equs}
where we have used (\ref{CondP4}) at the second line and Jensen's inequality at the third line. Summing over all $n$ such that $4\cdot 2^{-2n} \leq \lambda^2$, one gets a bound of the right order. Regarding the fourth term of (\ref{Eq:Transla3}), we have for all $\gamma+\alpha > \beta \geq \zeta >|k|-2$
\begin{equs}
{}& \bigg\| \frac{\int_{y\in B(x,\lambda)}\lambda^{-d} \big|\big\langle \Pi_{t,y}\cQ_\zeta \Gamma_{t,t_n}^y\cQ_\beta \big(f(t_n,y)-\Gamma_{y,x}^{t_n}f(t_n,x)\big) , D^{k} P_n\big((t,y)-\cdot)\big) \big\rangle\big| dy}{\wde_t(x,|k|)} \bigg\|_{L^p}\\
& \lesssim \bigg\| \frac{\int_{y\in B(x,\lambda)}\lambda^{-d} \big|f(t_n,y)-\Gamma_{y,x}^{t_n}f(t_n,x)\big|_\beta dy}{\wde_{t_n}(x,\beta)} \bigg\|_{L^p} 2^{-n(2-|k|+\beta-c)}\\
& \lesssim 2^{-n(2-|k|+\beta-c)} t^{\frac{\eta-\gamma}{2}} \lambda^{\gamma+\alpha-\beta}\;,
\end{equs}
where we have used (\ref{CondP1}). $c$ being small, we have $2+\beta-|k|-c > 0$ so that the sum over all the corresponding $n$ yields a bound of order $\lambda^{\gamma'-|k|}t^{\frac{\eta-\gamma}{2}}$ as desired. The fifth term of (\ref{Eq:Transla3}) is treated similarly, using (\ref{CondP4}). The bound of the sixth term follows easily from (\ref{CondP4}) as well.

\medskip
\noindent\textit{Third step: translation in time.} We need to control $(\cP_+ f)(t,x) - \Gamma_{t,s}^x(\cP_+ f)(s,x)$ for all $t > s > 0$ such that $(t-s) < s$. We start with the non-integer levels $\zeta\in \cA_{<\gamma'}(\cU)$, for which we have:
\begin{equs}
{}\bigg\| \frac{\big|\cI(f(t,x)-\Gamma_{t,s}^xf(s,x))\big|_\zeta}{(t-s)^{\frac{\gamma'-\zeta}{2}}s^{\frac{\eta'-\gamma'}{2}}\wun_t(x,\zeta)}\bigg\|_{L^p}\lesssim\bigg\| \frac{\big|u(t,x)-\Gamma_{t,s}^x u(s,x)\big|_{\zeta-2-\alpha}}{(t-s)^{\frac{\gamma-\zeta+2+\alpha}{2}}s^{\frac{\eta-\gamma}{2}}\wun_t(x,\zeta-2-\alpha)}\bigg\|_{L^p} \;,
\end{equs}
where we have used (\ref{CondP2}) and the identity $\gamma'-\gamma=\eta'-\eta=2+\alpha-c$ with $c>0$. This ensures the required bound.

We now turn to the terms at integer levels $k$ with $|k|<\gamma'$. Actually we need to distinguish three sub-cases. First, we assume that $t-s < s\leq 36\cdot 2^{-2n}$ and we write $\cQ_k \big((\cP_n f)(t,x)-\Gamma_{t,s}^x(\cP_n f)(s,x)\big)$ as:
\begin{equs}[Eq:RegTime1]
{}& \langle\cR f , P_{n;tx,sx}^{k,\gamma'}\rangle - \langle \Pi_{s,x} f(s,x) , P_{n;tx,sx}^{k,\gamma'} \rangle\\
&-\sum_{\zeta\leq |k|-2}\langle \Pi_{t,x}\cQ_\zeta\big(f(t,x)-\Gamma_{t,s}^x f(s,x)\big), D^k P_n\big((t,x)-\cdot\big)\rangle\;.
\end{equs}
By (\ref{Eq:TaylorTime}), we deduce that there exists $\delta > \gamma'+c$, such that for any distribution $g$ we have
\begin{equ}\label{Eq:ObservationTime}
\big|\big\langle g, P_{n;tx,sx}^{k,\gamma'} \big\rangle\big| \lesssim \sup_{\eta\in\cB^r_-} \big|\big\langle g, \eta_{t,x}^{2^{-n}+\sqrt{t-s}} \big\rangle\big| |t-s|^{\frac{\delta-|k|}{2}} 2^{-n(2-\delta)}\;,
\end{equ}
uniformly over all $s,t,n,\lambda$ as above. This being given, the bounds of the two first terms of (\ref{Eq:RegTime1}) follow easily from (\ref{Eq:I1}) and (\ref{CondP1}). Regarding the third term, we notice that the values $\zeta$ such that $\zeta=|k|-2$ have a zero contribution, since $P_n$ kills polynomials of degree $r$. On the other hand, for all $\zeta < |k|-2$, we use (\ref{CondP3}) to get
\begin{equs}
{}&\bigg\| \frac{\big\langle \Pi_{t,x}\cQ_\zeta\big(f(t,x)-\Gamma_{t,s}^x f(s,x)\big), D^k P_n\big((t,x)-\cdot\big)\big\rangle}{\wun_t(x,|k|)} \bigg\|_{L^p}\\
&\lesssim s^{\frac{\eta-\gamma}{2}} (t-s)^{\frac{\gamma+\alpha-\zeta}{2}} 2^{-n(2+\zeta-|k|)}\;.
\end{equs}
The sum over the corresponding $n$ yields a bound of order $s^{\frac{\eta-\gamma}{2}} (t-s)^{\frac{\gamma'-|k|}{2}}$ as required.

Second, we treat the case $t-s \leq 4\cdot 2^{-2n} < 36\cdot 2^{-2n} \leq s$. Set $s_n=t-9\cdot 2^{-2n}$, notice that $s_n \geq 3(2^{-n}+\sqrt{t-s})^2$. We write $k!\cQ_k \big((\cP_n f)(t,x)-\Gamma_{t,s}^x(\cP_n f)(s,x)\big)$ as:
\begin{equs}[Eq:RegTime2]
{}& \langle\cR f - \Pi_{s_n,x} f(s_n,x) , P_{n;tx,sx}^{k,\gamma'}\rangle - \langle \Pi_{s,x}\big(f(s,x)-\Gamma_{s,s_n}^x f(s_n,x)\big) , P_{n;tx,sx}^{k,\gamma'}\rangle\\
&-\sum_{\zeta\leq |k|-2}\langle \Pi_{t,x}\cQ_\zeta\big(f(t,x)-\Gamma_{t,s}^x f(s,x)\big), D^k P_n\big((t,x)-\cdot\big)\rangle\;.
\end{equs}
The bound of the first term is a direct consequence of (\ref{Eq:I2}) and (\ref{Eq:ObservationTime}), while the third term coincides with the third term of (\ref{Eq:RegTime1}) and the calculation made above applies. Regarding the second term, by (\ref{CondP1}) and (\ref{Eq:ObservationTime}) we have for all $\zeta\in\cA(\cF)$
\begin{equs}
{}&\bigg\| \frac{\big\langle \Pi_{s,x}\cQ_\zeta \big(f(s,x)-\Gamma_{s,s_n}^x f(s_n,x)\big) , P_{n;tx,sx}^{k,\gamma'}\big\rangle}{\wun_t(x,|k|)} \bigg\|_{L^p}\\
&\lesssim s^{\frac{\eta-\gamma}{2}} (t-s)^{\frac{\delta-|k|-c}{2}} 2^{-n(2+\gamma+\alpha-\delta)}\;.
\end{equs}
Since $2+\gamma+\alpha-\delta < 0$, the sum over the corresponding $n$ of the last expression yields a bound of order $s^{\frac{\eta-\gamma}{2}} (t-s)^{\frac{\gamma'-|k|}{2}}$ as required.

Finally, we consider the case $4\cdot 2^{-2n} \leq t-s \leq s$. We set $s_n=s-2^{-2n}$, $t_n=t-2^{-2n}$, and we write $k!\cQ_k \big((\cP_n f)(t,x)-\Gamma_{t,s}^x(\cP_n f)(s,x)\big)$ as:
\begin{equs}[Eq:RegTime3]
{}& \langle\cR f - \Pi_{t_n,x} f(t_n,x) , D^k P_n\big((t,x)-\cdot\big)\rangle \\
&- \langle \cR f- \Pi_{s_n,x} f(s_n,x) , \sum_{|k|+|\ell| < \gamma'} \frac{(t-s)^{\ell}}{\ell!} D^{k+\ell} P_n\big((s,x)-\cdot\big)\rangle \\
&- \sum_{\zeta\leq |k|-2}\langle \Pi_{t,x}\cQ_\zeta\big(f(t,x) - \Gamma_{t,t_n}^x f(t_n,x)\big) , D^k P_n\big((t,x)-\cdot\big)\rangle\\
&+ \sum_{\zeta > |k|-2} \langle \Pi_{t,x} \cQ_\zeta \Gamma_{t,t_n}^x\big(f(t_n,x)-\Gamma_{t_n,s}^x f(s,x)) , D^k P_n\big((t,x)-\cdot\big) \rangle \\
&+ \langle \Pi_{s,x}\big(f(s,x)- \Gamma_{s,s_n}^x f(s_n,x)\big) , \sum_{|k|+|\ell| < \gamma'} \frac{(t-s)^{\ell}}{\ell!} D^{k+\ell} P_n\big((s,x)-\cdot\big)\rangle \;.
\end{equs}
The required bound for the first two terms follows easily from (\ref{Eq:I2}), while the third term can be bounded using (\ref{CondP3}). Let us treat the fourth term. For all $\beta \geq \zeta > |k|-2$, using (\ref{CondP1}) we have
\begin{equs}
{}&\bigg\| \frac{\langle \Pi_{t,x} \cQ_\zeta \Gamma_{t,t_n}^x \cQ_\beta\big(f(t_n,x)-\Gamma_{t_n,s}^x f(s,x)) , D^k P_n\big((t,x)-\cdot\big)  \rangle}{\wun_t(x,|k|)}\bigg\|_{L^p}\\
&\lesssim s^{\frac{\eta-\gamma}{2}} (t-s-2^{-2n})^{\frac{\gamma+\alpha-\beta}{2}} 2^{-n(2+\beta-|k|-c)}\;.
\end{equs}
Since $c$ is small, we have $2-c+\beta-|k|>0$. Therefore, the sum over all $n$ such that $4\cdot 2^{-2n}\leq (t-s)$ is bounded by a term of order $s^{\frac{\eta-\gamma}{2}}(t-s)^{\frac{\gamma'-|k|}{2}}$ as required. Finally, the fifth term of (\ref{Eq:RegTime3}) can be bounded using (\ref{CondP1}).

\medskip
\noindent\textit{Fourth step: equality with the convolution.} Let us show that $\cR \cP_+ f = P_+ * \cR f$. By the uniqueness of the reconstruction theorem (Theorem \ref{Th:ReconstructionWeight}), it suffices to show that
\begin{equs}\label{Eq:EgConv}
\bigg\|\sup_{\eta\in\cB^r}\frac{\big|\langle (P_+ * \cR f) - \Pi_{t,x} (\cP_+ f)(t,x) , \eta^\lambda_{t,x} \rangle \big|}{\w_{t+\lambda^2}(x)}\bigg\|_{L^p} \lesssim \lambda^{\gamma'} t^{\frac{\eta'-\gamma'}{2}} \;,
\end{equs}
uniformly over all $\lambda\in (0,1]$ and all $t\in [3\lambda^2,T-\lambda^2]$. Using (\ref{Eq:CondModel2}) and (\ref{Eq:DefConvAbstract}), it is elementary to get:
\begin{equ}
\big\langle (P_+ * \cR f) - \Pi_{t,x} (\cP_+ f)(t,x) , \eta_{t,x}^\lambda \big\rangle = \int_{s,y} \eta_{t,x}^\lambda(s,y) \sum_{n\geq 0} R_n(t,x,s,y) ds\,dy\;,
\end{equ}
where
\begin{equs}
R_n(t,x,s,y) &= \langle \cR f - \Pi_{t,x} f(t,x) , P_n\big((s,y)-\cdot\big) \rangle \\
&- \sum_{|\ell| < \gamma'} \frac{(s-t,y-x)^\ell}{\ell!} \langle \cR f - \Pi_{t,x} f(t,x) , D^\ell P_n\big((t,x)-\cdot\big) \rangle\;.
\end{equs}
By the scaling properties of $\eta^\lambda$, we have
\begin{equs}\label{Eq:EquaConvo}
{}&\bigg\|\sup_{\eta\in\cB^r}\frac{\big|\big\langle (P_+ * \cR f) - \Pi_{t,x} (\cP_+ f)(t,x) , \eta_{t,x}^\lambda \big\rangle\big|}{\w_{t+\lambda^2}(x)}\bigg\|_{L^p}\\
&\lesssim \sum_{n\geq 0}\bigg\|\int_{(s,y)\in B\big((t,x),\lambda\big)} \lambda^{-2-d} \frac{|R_n(t,x,s,y)|}{\w_{t+\lambda^2}(x)} ds\,dy \bigg\|_{L^p}\;,
\end{equs}
uniformly over all the parameters. Then, we distinguish three cases. First, if $3\lambda^2 \leq t \leq 36\cdot 2^{-2n}$, we write
\begin{equ}
R_n(t,x,s,y) = \langle \cR f, P^{0,\gamma'}_{n,sy,tx} \rangle - \langle \Pi_{t,x} f(t,x) , P^{0,\gamma'}_{n,sy,tx} \rangle\;.
\end{equ}
By Lemma \ref{Lemma:TaylorSpaceTime}, we deduce that for any distribution $g$ we have
\begin{equs}\label{Eq:ObservationSpaceTime}
{}&\int_{(s,y)\in B\big((t,x),\lambda\big)}\lambda^{-2-d}\big|\big\langle g, P_{n;sy,tx}^{0,\gamma'} \big\rangle\big| ds\,dy \\
&\lesssim \sup_{\eta\in\cB^r_-} \big|\big\langle g, \eta_{t+\lambda^2,x}^{2^{-n}+2\lambda} \big\rangle\big| \sum_{\ell\in\partial\gamma'}\lambda^{|\ell|} 2^{-n(2-|\ell|)}\;,
\end{equs}
uniformly over all the parameters. Therefore, arguments very similar to those presented below (\ref{Eq:Transla1}) ensure that
\begin{equ}
\bigg\|\int_{(s,y)\in B\big((t,x),\lambda\big)} \lambda^{-2-d} \frac{|R_n(t,x,s,y)|}{\w_{t+\lambda^2}(x)} ds\,dy \bigg\|_{L^p} \lesssim \sum_{\ell\in\partial\gamma'}\lambda^{|\ell|} 2^{-n(\eta'-|\ell|)}\;,
\end{equ}
so that the sum over the corresponding $n$ yields a bound of order $\lambda^{\gamma'} t^{\frac{\eta'-\gamma'}{2}}$. Second, if $3\lambda^2 \leq 3\cdot 2^{-2n} < 36\cdot 2^{-2n} \leq t$, we set $t_n= t+\lambda^2-(2^{-n}+2\lambda)^2$. Notice that $t_n\geq 3(2^{-n}+2\lambda)^2$. Then, we write
\begin{equs}
R_n(t,x,s,y) &= \langle \cR f-\Pi_{t_n,x} f(t_n,x) , P^{0,\gamma'}_{n,sy,tx} \rangle\\
&+ \langle \Pi_{t,x}\big(f(t,x)-\Gamma^x_{t,t_n} f(t_n,x)\big) , P^{0,\gamma'}_{n,sy,tx} \rangle\;,
\end{equs}
and the arguments below (\ref{Eq:Transla2}) can easily be adapted to obtain a bound of order $\lambda^{\gamma'}t^{\frac{\eta'-\gamma'}{2}}$ as above. Finally, when $3\cdot 2^{-2n} \leq 3\lambda^2 \leq t$, the desired bound follows from the arguments presented below (\ref{Eq:Transla3}). This completes the proof of the theorem.
\end{proof}

\subsection{Smooth part of the heat kernel}

We now deal with the smooth part $P_-$ of the heat kernel defined in Lemma \ref{Lemma:Kernel}. For any $u\in\ccD$, we set $f=u\cdot\Xi$ and we let $\cP_-\cR f$ denote the map
\[ (t,x) \mapsto \sum_{k\in\N^{d+1}, |k| < \gamma'} \frac{X^k}{k!} \big\langle \cR f , D^k P_-\big((t,x)-\cdot\big)\big\rangle\;,\]
which takes values in the polynomial regularity structure. The following result shows that this is an element of $\ccD'$. Here we consider the weights defined in (\ref{Eq:Weights}), but the only important feature of these weights is that they do not grow faster than $e^{\frac{|x|^2}{T}}$.
\begin{proposition}\label{Prop:IntegrationSmooth}
Let $u\in \ccD=\ccD^{\gamma,\eta}_{T,\rw}(\cU)$ and $f=u\cdot\Xi$. Then, $\cP_-\cR f\in\ccD'=\ccD^{\gamma',\eta',p}_{T,\rw}$ and we have
\begin{equs}\label{Eq:Smooth1}
\$ \cP_- \cR f \$_{\ccD'} \lesssim \$\Pi\$ (1+\$\Gamma\$)\$ u \$_{\ccD}
\end{equs}
uniformly over all $T$ in a compact domain of $(0,\infty)$, all $\ell$ in a compact domain of $\R$, all $u\in\ccD$ and all admissible models $(\Pi,\Gamma)$. Moreover, if $(\bar \Pi, \bar \Gamma)$ is another admissible model with the same weight $\wPi$ and if $\bar{u}$ belongs to the corresponding space $\bar{\ccD}$, then we have the bound
\begin{equs}\label{Eq:Smooth2}
\$\cP_- \cR f;\cP_- \bar{\cR} \bar{f}\$_{\ccD',\bar{\ccD}'} &\lesssim  \$\Pi\$(1+\$\Gamma\$)\$u;\bar{u}\$_{\ccD,\bar{\ccD}}\\
\;&+\big(\$\Pi-\bar{\Pi}\$(1+\$\bar{\Gamma}\$)+ \$\bar\Pi\$\$\Gamma-\bar\Gamma\$\big)\$\bar u\$_{\ccD}\;,\qquad
\end{equs}
uniformly over all $T,\ell$ as above, all admissible models $(\Pi,\Gamma)$, $(\bar \Pi, \bar \Gamma)$, and all $u\in\ccD$, $\bar{u}\in\bar{\ccD}$.
\end{proposition}
\begin{proof}
Suppose that
\begin{equs}\label{Eq:BoundSmooth}
\sup_{t\in(0,T]}\sup_{|k| < \gamma'+2} \bigg\| \frac{ \big\langle \cR f , D^k P_-\big((t,x)-\cdot\big) \big\rangle}{\w_t(x)} \bigg\|_{L^p} \lesssim \$\Pi\$ (1+\$\Gamma\$)\$ u \$_{\ccD} \;,
\end{equs}
uniformly over all $T$, $\ell$, $(\Pi,\Gamma)$ and $u$ as in the statement. We stress that this implies (\ref{Eq:Smooth1}). Indeed, for the local terms of the norm this is immediate. Regarding the space translations, we have for every $k\in\N^{d+1}$ such that $|k| < \gamma'$ and all $x,y\in\R^d$:
\begin{equs}
\cQ_k\Big( \cP_- \cR f(t,y) - \Gamma^t_{y,x} \cP_- \cR f(t,x) \Big) = \langle \cR f , P^{k,\gamma'}_{-,ty,tx}\rangle\;,
\end{equs}
where $P^{k,\gamma'}_{-,ty,tx}$ is the function obtained from (\ref{Eq:TaylorSpace}) upon replacing $P_n$ by $P_-$. This being given, a simple application of Jensen's inequality shows that
\begin{equs}
{}&\bigg\| \int_{y\in B(x,\lambda)} \lambda^{-d}\frac{ \big|\big\langle \cR f , P^{k,\gamma'}_{-,ty,tx} \big\rangle\big|}{\w_t(x)} dy\bigg\|_{L^p}\\
&\lesssim \sum_{\ell \in \partial \gamma'} \bigg\| \frac{ \big\langle \cR f , D^{\ell} P_-\big((t,x)-\cdot\big) \big\rangle}{\w_t(x)} \bigg\|_{L^p} \lambda^{|\ell|-|k|}\;,
\end{equs}
so that the desired bound holds. Concerning the time translation, we have for every $k\in\N^{d+1}$ such that $|k| < \gamma'$ and all $0 < t-s < s$:
\begin{equs}
\cQ_k\Big( \cP_- \cR f(t,x) - \Gamma^x_{t,s} \cP_- \cR f(s,x) \Big) = \langle \cR f , P^{k,\gamma'}_{-,tx,sx}\rangle\;,
\end{equs}
where $P^{k,\gamma'}_{-,tx,sx}$ is the function obtained from (\ref{Eq:TaylorTime}) upon replacing $P_n$ by $P_-$. Similarly as above, a simple application of Jensen's inequality shows that
\begin{equs}
{}&\bigg\| \frac{ \big|\big\langle \cR f , P^{k,\gamma'}_{-,tx,sx} \big\rangle\big|}{\w_t(x)} \bigg\|_{L^p}\\
&\lesssim \sup_{u\in [s,t]}\sum_{\ell \in \partial \gamma'} \bigg\| \frac{ \big\langle \cR f , D^{\ell} P_-\big((u,x)-\cdot\big) \big\rangle}{\w_t(x)} \bigg\|_{L^p} |t-s|^{\frac{|\ell|-|k|}{2}}\;,
\end{equs}
and the desired bound follows.

We now prove (\ref{Eq:BoundSmooth}). Let $\tilde{\phi}:[-1,1]\rightarrow\R$ be a smooth function such that for all $x\in\R$, $\sum_{i\in\Z}\tilde{\phi}(x-i) = 1$. Then, we define $\phi(t,x)=\tilde{\phi}(t)\prod_{i=1}^d\tilde{\phi}(x_i)$ for every $(t,x)\in\R^{d+1}$, so that we obtain $\sum_{i\in\Z,j\in\Z^d}\phi\big((t-i,x-j)\big) = 1$. In particular, we have
\begin{equs}
D^k P_-\big((t,x)-\cdot\big) = \sum_{i\in\Z,j\in\Z^d} D^k P_-\big((t,x)-\cdot\big) \phi\big((t-i,x-j)-\cdot\big)\;.
\end{equs}
Since $P_-(t,x)$ is smooth and equals the heat kernel outside the parabolic unit ball, the following bound
\begin{equs}
\left\|D^k P_-\big((t,x)-\cdot\big) \phi\big((t-i,x-j)-\cdot\big) \right\|_{\cC^r} \lesssim e^{-\frac{(|j|^2-d)_+}{8t}} \;,
\end{equs}
holds uniformly over all $t\in(0,T]$, all $k\in\N^{d+1}$ such that $|k| < \gamma'+2$ and all $(i,j)\in\Z^{d+1}$. The expression (\ref{Eq:Weights}) of the weights yield that $\rw_t(x) = e^{(t+\ell)(1+|x|)}$. Using (\ref{Eq:BoundReconstructionWeight}) and setting $C=\$\Pi\$ (1+\$\Gamma\$)\$ u \$_{\ccD}$, we get
\begin{equs}
{}\Big\|\frac{\big\langle \cR f , D^k P_-\big((t,x)-\cdot\big)\big\rangle}{\w_t(x)}\Big\|_{L^p}
&\lesssim  C\sum_{i=-1}^{T+1}\sum_{j\in\Z^d} e^{-\frac{(|j|^2-d)_+}{8t}} \Big\| \frac{\rw_t(x-j)}{\rw_t(x)} \Big\|_{L^p}\\
&\lesssim C \sum_{j\in\Z^d} e^{(t+\ell)|j|-\frac{(|j|^2-d)_+}{8t}}\\
&\lesssim C\;,
\end{equs}
uniformly over all $t\in(0,T]$, all $T$ in a compact domain of $\R_+$, all $k\in\N^{d+1}$ such that $|k|< \gamma'+2$. This ends the proof of (\ref{Eq:Smooth1}). To obtain (\ref{Eq:Smooth2}), we proceed similarly. Using (\ref{Eq:BoundReconstTwoModels}), the same calculation as above gives
\begin{equs}
{}\Big\|\frac{\big\langle \cR f-\bar{\cR} \bar{f} , D^k P_-\big((t,x)-\cdot\big)\big\rangle}{\w_t(x)}\Big\|_{L^p}
&\lesssim \$\Pi\$(1+\$\Gamma\$)\$u;\bar{u}\$\\
&+\big(\$\Pi-\bar{\Pi}\$(1+\$\Gamma\$)+\$\bar{\Pi}\$\$\Gamma-\bar\Gamma\$\big)\$\bar u\$ \;,
\end{equs}
uniformly over all $t\in(0,T]$, all $T$ in a compact domain of $\R_+$, all $k\in\N^{d+1}$ such that $|k|< \gamma'+2$. This ends the proof.
\end{proof}

\subsection{Initial condition}\label{SubsectionIC}
We take (\ref{Eq:Weights}) as our choice of weights. Recall that $\ell$ is involved in the weight  at time $0$. We define $\cC^{\eta,p}_{\rw_0}(\R^d)$ as the space of distributions $f$ on $\R^d$ such that
\begin{equ}
\sup_{\lambda \in (0,1]} \bigg\| \sup_{\phi \in \cB^r(\R^d)} \frac{|\langle f , \phi^\lambda_{x} \rangle|}{\lambda^\eta \rw_0(x)} \bigg\|_{L^p(dx)} < \infty\;.
\end{equ}
When $\rw_0(x) =1$, this space coincides with the usual Besov space $\cB^\alpha_{p,\infty}(\R^d)$.

Given $u_0\in \cC^{\eta,p}_{\rw_0}(\R^d)$, we define $v=\cP u_0$ as follows:
\[ v(t,x) := \sum_{\substack{k\in\N^{d+1}\\|k|< \gamma'}} \frac{X^k}{k!} \langle u_0 , D^k P(t,x-\cdot) \rangle \;.\]
This is the lift into the polynomial regularity structure of the smooth map $(t,x)\mapsto \big(P(t,\cdot)*u_0\big)(x)$.
\begin{lemma}\label{LemmaIC}
Let $u_0\in \cC^{\eta,p}_{\rw_0}(\R^d)$ then $v=\cP u_0$ belongs to $\ccD$.
\end{lemma}
\begin{proof}
The contribution coming from the smooth part of the heat kernel is handled similarly as in the proof of Proposition \ref{Prop:IntegrationSmooth} so we do not provide the details. We focus on the contribution due to the singular part of the heat kernel. By hypothesis, we have
\[ \Big\| \frac{\langle u_0 , D^k P_n(t,x-\cdot)\rangle}{\rw_0(x)} \Big\|_{L^p} \lesssim 2^{-n(\eta-|k|)}\;, \]
uniformly over all $t>0$, all $n\geq 0$ and all $k\in\N^{d+1}$ such that $|k|< \gamma+2$. Notice that the definition of the kernels $P_n$ ensures that the left hand side actually vanishes whenever $t>2^{-2n}$. Therefore, summing over $n\geq 0$ the latter bound yields
\[ \Big\| \frac{\langle u_0 , D^k P_+(t,x-\cdot)\rangle}{\rw_0(x)} \Big\|_{L^p} \lesssim t^{\frac{\eta-|k|}{2}}\;, \]
uniformly over all $t>0$. This yields the required bound for the local terms of the norm, while the bounds on the time and space translation terms follow from the same arguments as in the proof of Proposition \ref{Prop:IntegrationSmooth}.
\end{proof}

\section{Solution map and renormalisation}\label{sec:final}

We are now in position to obtain a fixed point for the solution map:
\begin{equs}[2][e:fp]
\cM_{T,v}:\ccD &\rightarrow \ccD\\
u&\mapsto (\cP_+ + \cP_-)(u\cdot\Xi) + v
\end{equs}
where $v$ is a given element in $\ccD$. In practice, we will take $v=\cP u_0$ with $u_0\in\cC^{\eta,p}_{\rw_0}$ as in Lemma \ref{LemmaIC}. Recall that the weight $\rw_0$ depends on the parameter $\ell \in \R$. We start with a simple lemma.
\begin{lemma}\label{Lemma:ReconstructionPoly}
Let $u\in\ccD^{\gamma,\eta,p}_{\rw,T}(\cU)$. Then, $\cR u$ is a function and we have $\cR u(t,x)=\cQ_0 u(t,x)$ together with $\cR u(t,\cdot) \in \cC^{\eta,p}_{\rw_t}(\R^d)$. If in addition $u$ only takes values in the strictly positive levels of the polynomial regularity structure, then $u=0$.
\end{lemma}
\begin{proof}
Observe that uniformly over all $\lambda \in (0,1]$, all $t\in (2\lambda^2,T-\lambda^2]$ and all $x_0\in\R^d$, we have
\begin{equs}
{}&\Big\| \int_{(s,y)\in B((t,x),\lambda)} \lambda^{-d-2}|u(s,y) - u(t,x)|_0 \,ds\,dy  \Big\|_{L^p_{x_0,1}}\\
&\leq \sup_{s\in(t-\lambda^2,t+\lambda^2)} \Big\| \int_{y\in B(x,\lambda)}\lambda^{-d} |u(s,y) - \Gamma_{y,x}^s u(s,x)|_0 \,dy  \Big\|_{L^p_{x_0,1}}\\
&+ \Big\| \int_{(s,y)\in B((t,x),\lambda)}\lambda^{-d-2} |\Gamma_{y,x}^s (u(s,x) - \Gamma_{s,t}^x u(t,x))|_0 \,ds\,dy  \Big\|_{L^p_{x_0,1}}\\ &+ \sum_{\zeta > 0}\Big\| \int_{(s,y)\in B((t,x),\lambda)}\lambda^{-d-2} |\Gamma_{sy,tx}\cQ_\zeta u(t,x)|_0\, ds\,dy  \Big\|_{L^p_{x_0,1}}\\
&\lesssim \sup_{i=1,2}\sup_{\beta\in\cA} \wi_{t+\lambda^2}(x,\beta) \lambda^{\zeta_0}\;,
\end{equs}
where $\zeta_0$ is the smallest non-zero element of $\cA(\cU)$. Then, we write
\begin{equs}
\langle \cQ_0 u(\cdot) - \Pi_{t,x} u(t,x) , \eta_{t,x}^\lambda \rangle& = \int_{s,y} \cQ_0\big(u(s,y)-u(t,x)\big)\eta_{t,x}^\lambda(s,y) ds\,dy\\
&- \sum_{\zeta > 0}\langle \Pi_{t,x}\cQ_\zeta u(t,x) , \eta_{t,x}^\lambda\rangle\;,
\end{equs}
so that, taking the $L^p_{x_0,1}$-norm, one gets a bound of order $\lambda^{\zeta_0}$ times some weight. From the uniqueness of the reconstruction, we deduce that $\cR u(\cdot)=\cQ_0 u(\cdot)$ on $(0,T)\times\R^d$. It is then immediate to check that $\cR u(t,\cdot)$ belongs to $\cC^{\eta,p}_{\rw_t}(\R^d)$.

Recall that $\gamma \in (1,2)$. We now assume that $u(t,x) = \sum_{k\in\N^{d+1}:|k|=1} \cQ_k\big(u(t,x)\big) X^k$. Let $e_i,i=1\ldots d$ be the unit vector in the space direction $i$. We start with the following simple observation. There exists a constant $C>0$ such that
\begin{equ}
\int_{y\in B(0,\lambda)} \lambda^{-d} \Big| \sum_{i=1}^d y_i a_i \Big| dy \geq C \lambda |a|\;,
\end{equ}
uniformly over all $\lambda \in (0,1]$ and all $a\in\R^d$. This being given, we take $a= \sum_{i=1}^d (\cQ_{e_i} u(t,x))e_i$ and use the equivalence of norms in $\R^d$ to get
\begin{equs}
\Big\| \sum_{i=1}^d |\cQ_{e_i} u(t,x)| \Big\|_{L^p_{x_0,1}} &\lesssim \lambda^{-1} \Big\| \int_{y\in B(x,\lambda)} \lambda^{-d}\sum_{i=1}^d |(y-x)_i\cQ_{e_i} u(t,x)|dy \Big\|_{L^p_{x_0,1}}\\
&\lesssim \lambda^{-1}\Big\| \int_{y\in B(x,\lambda)} \lambda^{-d}|u(t,y)-\Gamma^t_{y,x} u(t,x)|_0 dy \Big\|_{L^p_{x_0,1}}\\
&\lesssim \lambda^{\gamma-1} \wde_{t}(x_0,0)\;,
\end{equs}
uniformly over all $\lambda\in (0,1]$, all $t\in(2\lambda^2,T-\lambda^2]$ and all $x_0\in\R^d$. Therefore, the l.h.s.~vanishes. This concludes the proof.
\end{proof}
\begin{theorem}\label{Th:FixedPt}
For any $T>0$ and any $u_0\in \cC^{\eta,p}_{\rw_0}$, the equation $u = \cM_{T,v}(u)$ admits a unique solution in $\ccD$. Furthermore, the map $v\mapsto u$ is Lipschitz continuous, while the map $(v,\Pi,\Gamma)\mapsto u$ is locally Lipschitz continuous.
\end{theorem}
\begin{proof}
We first introduce a shift map on the models and the modelled distributions. For all $s\geq 0$, we let $\Pi^{\downarrow s}$ and $\Gamma^{\downarrow s}$ be defined as follows
\begin{equ}
\langle \Pi^{\downarrow s}_{z} \tau , \varphi \rangle := \langle \Pi_{z+(s,0)} \tau , \varphi(\cdot+s,\cdot)\rangle \;,\quad \Gamma^{\downarrow s}_{z,z'} \tau = \Gamma_{z+(s,0),z'+(s,0)} \tau\;.
\end{equ}
We let $\ccD^{\downarrow s,\gamma,\eta,p}_{\rw,T}$ be the space of modelled distributions associated with the shifted model $(\Pi^{\downarrow s},\Gamma^{\downarrow s})$ and the shifted weights $\rw^{\downarrow s}$ defined by setting
\begin{equ}
{\rw}^{\downarrow s,(i)}_t(x,\zeta) := \wi_{t+s}(x,\zeta) \;.
\end{equ}
This amounts to shifting the parameter $\ell$ by $s$, in the definition (\ref{Eq:Weights}) of the weights. Formally, one should also write $\cR^{\downarrow s}$ and $\cP^{\downarrow s}$ for the convolution and reconstruction operators associated with the shifted model, but we refrain from doing that for the sake of readability.

Recall that the spaces $\ccD$ and $\ccD'$ differ by their parameters $\eta,\gamma$ and $\eta',\gamma'$. Since $\eta'-\eta=\gamma'-\gamma > 0$, we deduce that there exists $\rho > 0$ such that $\$ \cdot \$_{\ccD'} \leq T^\rho \$ \cdot \$_{\ccD}$. Until the end of the proof, we will be working in the spaces $\ccD^{\gamma,\eta,p}_{\rw,T}$ as well as their shifted counterparts and we will play with only two parameters, namely $T$ and $\ell$. Recall that $\ell$ is the parameter involved in the weight at time $0$. We will use the notation $\ccD_{T,\ell}$ instead of $\ccD^{\gamma,\eta,p}_{\rw,T}$ for simplicity.

Using Theorem \ref{Th:Integration} and Proposition \ref{Prop:IntegrationSmooth}, we deduce the existence of $C > 0$ such that
\begin{equ}
\$ \cM_{T,v}(u)-\cM_{T,v}(\bar{u}) \$_{\ccD^{\downarrow s}_{T,\ell}} = \$ (\cP_+ + \cP_-)\big((u-\bar u)\Xi\big) \$_{\ccD^{\downarrow s}_{T,\ell}} \leq C\,T^\rho\$ u-\bar{u} \$_{\ccD^{\downarrow s}_{T,\ell}} \;,
\end{equ}
as well as
\begin{equ}\label{Eq:FixedPt}
\$ \cM_{T,v}(u) \$_{\ccD^{\downarrow s}_{T,\ell}} = \$ (\cP_+ + \cP_-)\big(u\Xi\big)+v \$_{\ccD^{\downarrow s}_{T,\ell}} \leq C\,T^\rho\$ u\$_{\ccD^{\downarrow s}_{T,\ell}}+\$ v\$_{\ccD^{\downarrow s}_{T,\ell}} \;,\quad
\end{equ}
uniformly over all $s,T$ in a compact set of $\R_+$, all $\ell$ in a compact set of $\R$ and all $u,\bar{u},v \in \ccD^{\downarrow s}_{T,\ell}$. The constant $C$ does however depend on the realisation of the model through 
the quantities appearing in Lemma~\ref{lem:normModel}.

Fix a ``target'' final time $T>0$ and $\ell_0\in\R$. Taking $T^*$ small enough, we deduce that $\cM_{T^*,v}$ is a contraction on $\ccD_{T^*,\ell}^{\downarrow s}$ uniformly over all $\ell \in [\ell_0,\ell_0+T]$, all $s \in [0,T]$ and all $v\in\ccD_{T^*,\ell}^{\downarrow s}$. Fix $u_0\in\cC^{\eta,p}_{\rw_0}$ and let $v=\cP u_0 \in\ccD_{T^*,\ell}$. The map $\cM_{T^*,v}$ admits a unique fixed point $u^* \in \ccD_{T^*,\ell_0}$. If $T^*>T$ we are done, otherwise we take $s\in (0,T^*)$ and we define $\ell^* = \ell_0+s < \ell_0 + T$, $u_s := \cR u(s,\cdot)$ and $v^* := \cP u_s$. By Lemma \ref{Lemma:ReconstructionPoly} and \ref{LemmaIC}, we know that $v^*\in \ccD_{T^*,\ell^*}$. The map $\cM_{T^*,v^*}$ admits a unique fixed point $u^{**}\in\ccD^{\downarrow s}_{T^*,\ell^*}$.
We then set $u(t,\cdot)=u^*(t,\cdot)$ when $t\in (0,T^*]$ and $u(t,\cdot)=u^{**}(t-s,\cdot)$ when $t\in(T^*,T^*+s]$.
It follows in the same way as in \cite[Prop.~7.11]{Hairer2014} that $u$ is indeed the unique solution to the 
fixed point problem $\cM_{T^*+s,v}(u) = u$, and that this construction can be iterated until one
reaches the final time $T$. Note that the linearity of the problem was exploited in an essential way here, since
this is what guarantees that the time $T^*$ of local well-posedness does not depend on the initial condition.

Regarding the joint dependence on the model and the initial condition, we obtain similarly as above and thanks to the same results that for all $R>0$, there exists $T^*>0$ such that
\begin{equs}
\$ u;\bar{u} \$_{\ccD^{\downarrow s},\bar\ccD^{\downarrow s}} &\leq \$ \Pi-\bar\Pi \$ + \$ \Gamma-\bar\Gamma \$ + \$ v ; \bar v \$_{\ccD^{\downarrow s},\bar\ccD^{\downarrow s}}\;,
\end{equs}
uniformly over all $s$ in a compact set of $\R_+$, and over all $(\Pi,\Gamma)$, $(\bar \Pi,\bar \Gamma)$ and $v,\bar v\in\ccD_{T^*,\ell_0}^{\downarrow s}$, such that the norms of all these elements are bounded by $R$. This yields the local Lipschitz continuity of the solution map on $(0,T^*]$. Iterating the argument as above, we obtain the local Lipschitz continuity over any arbitrary interval $(0,T]$.
\end{proof}

Let $v=\cP u_0$ with $u_0\in\cC^{\eta,p}_{\rw_0}$. It is easily seen from Theorems \ref{Th:ReconstructionWeight} and \ref{Th:Integration} that the unique fixed point of $\cM_{T,v}$ associated with the canonical model $(\Pi^{(\eps)}, F^{(\eps)})$ coincides, upon reconstruction, with the solution to the well-posed SPDE (\ref{e:Eeps}) presented in the introduction. However, the sequence of canonical models $(\Pi^{(\eps)}, F^{(\eps)})$ does not converge when $\epsilon\rightarrow 0$, due to the ill-defined products involving the white noise.

\begin{theorem}
For every $\epsilon \in (0,1]$, there exists a renormalised model $(\hat \Pi^\epsilon, \hat F^\epsilon)$ such that:\begin{itemize}
\item the unique fixed point of $\cM_{T,v}$ associated to $(\hat \Pi^\epsilon, \hat F^\epsilon)$ coincides, upon reconstruction, with the classical solution of (\ref{e:Eepshat}),
\item the sequence $(\hat \Pi^\epsilon, \hat F^\epsilon)$ converges to an admissible model $(\hat{\Pi},\hat{F})$, that is, there exists $C,\delta>0$ such that uniformly over $\epsilon\in (0,1]$ we have
\begin{equs}
\$ \hat\Pi^\epsilon-\hat\Pi \$ + \$ \hat\Gamma^\epsilon-\hat\Gamma \$ \leq C\epsilon^\delta\;.
\end{equs}
\end{itemize}
\end{theorem}

\begin{proof}
This result is due to Hairer and Pardoux~\cite[Th 4.5]{Etienne} in the case of (SHE). The case of (PAM) is treated similarly \textit{mutatis mutandis}. Let us briefly explain why the solution to \eqref{e:fp}
yields the classical solution to \eqref{e:Eepshat} when applied to the renormalised model $(\hat \Pi^\epsilon, \hat F^\epsilon)$.

We first note that, for any space-time point $z$, the renormalised model fulfils the following identities:
\begin{equs}[Eq:Pihat]
\hat \Pi^\epsilon_z( \Xi ) (z) = \xi_\epsilon(z) \;,\quad \hat \Pi^\epsilon_z( \Xi \cI(\Xi) ) (z) = -c_\epsilon \;,\quad \hat \Pi^\epsilon_z( \Xi \cI(\Xi\cI(\Xi)) ) (z) = 0 \;,\qquad\\
\hat \Pi^\epsilon_z( \Xi \cI(\Xi\cI(\Xi\cI(\Xi))) ) (z) = -c^{(1)}_\epsilon \;,\quad \hat \Pi^\epsilon_z( \Xi \cI(X_i\Xi) ) (z) = 0 \;,\qquad
\end{equs}
where $c^{(1)}_\epsilon=c^{(1,1)}_\epsilon+c^{(1,2)}_\epsilon$, see (\ref{Eq:RenormCsts}) for the values 
of these constants.

Furthermore, iterating \eqref{e:fp} shows that any solution $U$ to $\CM_{T,v}(U) = U$ will necessarily be of the form
\begin{equ}
U(z) = u(z)\big(\one + \cI(\Xi) + \cI(\Xi\cI(\Xi)) + \cI(\Xi\cI(\Xi\cI(\Xi))) \big)
+ \sum_{|k|=1}\d_k u(z) \bigl(X^k + \cI(X^k\Xi)\bigr) \;,
\end{equ}
for some continuous functions $u$ and $\d_k u$. Recalling that, for fixed $\eps > 0$, 
the reconstruction operator associated to the renormalised model  
is given by $(\CR F)(z) = (\hat \Pi^\epsilon_z F(z))(z)$, it then follows from
\eqref{Eq:Pihat} that
\begin{equ}
\bigl(\CR \Xi U\bigr)(z) = u(z)(\xi_\eps(z) - C_\eps)\;.
\end{equ}
Combining this with \eqref{e:commutInt} then concludes the proof. 
\end{proof}

We are now in position to conclude the proof of the main result of this article.

\begin{proof}[of Theorem \ref{Th:Main}]
The local Lipschitz continuity of the solution map stated in Theorem \ref{Th:FixedPt} together with the convergence of the renormalised models obtained in the previous theorem ensure that the sequence of renormalised solutions converge to a limit $\hat{u}\in\ccD^{\gamma,p}_{\rw,T}$, for any initial condition $u_0\in\cC^{\eta,p}_{\rw_0}$. By Theorem \ref{Th:Reconstruction}, we deduce the convergence of the reconstructed solution $\hat\cR^\epsilon \hat{u}^\epsilon$ towards $\hat\cR \hat{u}$ in the space $\ccE^{\eta-c,p}_{\rw,T}$.

Finally, a simple computation shows that the Dirac mass at some given point $x_0$ belongs to $\cC^{\eta,p}_{\rw_0}$ as soon as $p\leq \frac{d}{d+\eta}$, whatever weight $\rw_0$ one chooses. Since $\eta$ needs to be greater than $-1/2$ for our result to hold, one can choose a Dirac mass when $p =1$ for instance. This concludes the proof.
\end{proof}

\bibliographystyle{Martin}
\bibliography{library_mSHE}

\end{document}